\documentclass[reqno]{amsart}
\usepackage{amsmath}
\usepackage{amsthm}
\usepackage{amssymb}
\usepackage{dsfont}
\usepackage{graphicx}
\usepackage[english]{babel}
\usepackage{layout}
\usepackage{color}
\usepackage{amsaddr}

\usepackage[colorlinks,citecolor=blue,urlcolor=black]{hyperref}

\newtheorem{theorem}{Theorem}[section]
\newtheorem{corollary}[theorem]{Corollary}
\newtheorem{lemma}[theorem]{Lemma}
\newtheorem{proposition}[theorem]{Proposition}
\theoremstyle{definition}
\newtheorem{definition}[theorem]{Definition}
\newtheorem{remark}[theorem]{Remark}

\newtheorem*{theorem*}{Theorem}

\def\Xint#1{\mathchoice
	{\XXint\displaystyle\textstyle{#1}}%
	{\XXint\textstyle\scriptstyle{#1}}%
	{\XXint\scriptstyle\scriptscriptstyle{#1}}%
	{\XXint\scriptscriptstyle\scriptscriptstyle{#1}}%
	\!\int}
\def\XXint#1#2#3{{\setbox0=\hbox{$#1{#2#3}{\int}$ }
		\vcenter{\hbox{$#2#3$ }}\kern-.57\wd0}}

\def\dashint{\Xint-}

\newcommand{\dx}{\,\mathrm{d}x}
\newcommand{\e}{\varepsilon}
\newcommand{\dist}{{\rm{dist}}}
\newcommand{\Lw}{\mathcal{L}(\omega)}
\newcommand{\R}{\mathbb{R}}
\newcommand{\w}{\omega}
\newcommand{\Ard}{\mathcal{A}^R(D)}

\newcommand\blfootnote[1]{%
	\begingroup
	\renewcommand\thefootnote{}\footnote{#1}%
	\addtocounter{footnote}{-1}%
	\endgroup
}

\begin{document}

\author{Matthias Ruf}
\address[Matthias Ruf]{D\'epartement Math\'ematique, Universit\'e libre de Bruxelles (ULB), CP 214, boulevard du Triomphe, 1050 Brussels, Belgium}
\email{matthias.ruf@ulb.ac.be}

\title{Discrete stochastic approximations of the Mumford-Shah functional}

\begin{abstract}
We propose a new $\Gamma$-convergent discrete approximation of the Mumford-Shah functional. The discrete functionals act on functions defined on stationary stochastic lattices and take into account general finite differences through a non-convex potential. In this setting the geometry of the lattice strongly influences the anisotropy of the limit functional. Thus we can use statistically isotropic lattices and stochastic homogenization techniques to approximate the vectorial Mumford-Shah functional in any dimension.
\end{abstract}


\maketitle

\blfootnote{\\ \copyright 2019. This version is made available under the CC-BY-NC-ND 4.0 license \url{http://creativecommons.org/licenses/by-nc-nd/4.0/}}

%\tableofcontents
\section{Introduction}
The Mumford-Shah functional has its origin in image segmentation problems \cite{MS}. Given a rectangle (or more generally a bounded domain $D\subset\R^2$) and a function $g:D\to\R$ representing the gray level of an image, one aims at minimizing the functional
\begin{equation*}
M\!S(u,K)=\int_{D\backslash K}|\nabla u|^2\,\mathrm{d}x+\beta\,\mathcal{H}^1(K)+\gamma\int_D|u-g|^2\,\mathrm{d}x.
\end{equation*}
Here $K$ is the union of a finite number of points and a finite set of smooth arcs joining these points with no other intersections. The function $u$ is supposed to be differentiable on $D\backslash K$, but may have discontinuities on $K$. Then the pair $(u,K)$ is an approximation of the image. $K$ represents the sharp edges in the image, while the smooth part $u$ yields a cartoon-like total image since it rules out fine textures away from $K$. Existence and regularity of minimizing pairs $(u,K)$ is far from being trivial. In \cite{DeG-Ca-Le} it was shown that there exists a minimizing pair $(u,K)$ among all closed sets $K$ and $u\in C^1(D\backslash K)$ provided $g\in L^{\infty}(D)$. As commonly done in variational problems, one first has to enlarge the set of competitors in order to obtain compactness of minimizing sequences. This leads to the nowadays well-known formulation of the Mumford-Shah functional for $SBV$-functions, which was first introduced in \cite{AmDeG}: given $u\in SBV(D)$, the Mumford-Shah functional takes the form
\begin{equation}\label{intro:weakMS}
M\!S(u)=\int_{D}|\nabla u|^2\,\mathrm{d}x+\beta\,\mathcal{H}^1(S_u)+\gamma\int_D|u-g|^2\,\mathrm{d}x.
\end{equation}
Here $S_u$ denotes the discontinuity set of $u$. The closed set $K$ then can be recovered setting $K=\overline{S_u}$ since minimizers have an essentially closed discontinuity set (see \cite[Lemma 5.2]{DeG-Ca-Le}). However it is still unknown if $K$ can be taken as a finite union of regular arcs. We refer the interested reader to the recent survey articles \cite{Fo,Le} for known regularity results for minimizers.

Besides the regularity of minimizers, there is the natural question how to minimize the Mumford-Shah functional (\ref{intro:weakMS}) in practice. A very popular approach is given by the Ambrosio-Tortorelli approximation \cite{AmTo,AmTo2}, where the surface term is replaced by a Modica-Mortola-type approximation with an additional variable. More precisely, given a small parameter $\e>0$ and $0<\eta_{\e}\ll\e$ one defines an elliptic approximation $AT_{\e}:W^{1,2}(D)\times W^{1,2}(D)\to [0,+\infty]$ by
\begin{equation*}
AT_{\e}(u,v)=\int_D (\eta_{\e}+v^2)|\nabla u|^2\,\mathrm{d}x+\frac{\beta}{2}\int_D\e|\nabla v|^2+\frac{1}{\e}(v-1)^2\,\mathrm{d}x+\gamma\int_D|u-g|^2\,\mathrm{d}x.
\end{equation*}
In \cite{AmTo2} it is shown that the family $AT_{\e}$ approximates the Mumford-Shah functional (\ref{intro:weakMS}) in the sense of $\Gamma$-convergence (we refer to the monographs \cite{GCB,DM} for details on this type of convergence). In particular, up to subsequences, the $u$-component of any global minimizer $(u_{\e},v_{\e})$ of $AT_{\e}$ converges to a global minimizer of $M\!S$. This approach was recently extended to second order penalizations, that means to replace the term $\e|\nabla v|^2$ either by $\e^3|\nabla^2 v|^2$ or $\e^3(\Delta v)^2$, where in the second case one puts additional boundary conditions on $v$ (see \cite{BuEpZe} for more details or \cite{Ba} for an anisotropic version).

Instead of introducing a second variable, Braides and Dal Maso constructed non-local approximations. In \cite{BrDM} they showed that the sequence of functionals $N\!L_{\e}:W^{1,2}(D)\to [0,+\infty)$ defined by
\begin{equation*}
N\!L_{\e}(u)=\frac{1}{\e}\int_D f\Big(\e\dashint_{B_{\e}(x)\cap D}|\nabla u(y)|^2\,\mathrm{d}y\Big)\,\mathrm{d}x+\gamma\int_D|u-g|^2\,\mathrm{d}x
\end{equation*}
$\Gamma$-converges to $M\!S$ provided $f$ is continuous, increasing and satisfies
\begin{equation}\label{intro:f}
\lim_{t\to 0}\frac{f(t)}{t}=1,\quad\quad\lim_{t\to +\infty}f(t)=f_{\infty}<\infty.
\end{equation}
In this case it turns out that $\beta=2f_{\infty}$.

Note that both approximations are defined on more regular, but still infinite-dimensional spaces. Hence one has to discretize these spaces to numerically solve the minimization problems for the approximating functionals. On the one hand, $\e$ should be taken very small in order to obtain almost sharp interfaces. On the other hand, to guarantee that finite elements/differences yield the same asymptotic behavior as the continuum approximations, it is proposed to take the mesh-size to be infinitesimal with respect to $\e$ (see \cite{BeCo,Bou}). Indeed, for the Modica-Mortola approximation of the perimeter, such a choice is known to be necessary to preserve isotropy \cite{BrYi}. Very recently, in \cite{BBZ} this result was extended to finite difference discretizations of the Ambrosio-Tortorelli functional. Thus continuum approximations require in general a very fine mesh, which increases the computational effort. However, in dimension one there exists a direct approximation based on finite differences. In this case the small parameter $\e$ represents the mesh-size of a one-dimensional grid $\e\mathbb{Z}$. Given a function $u:\e\mathbb{Z}\cap(0,1)\to\R$, we define the functional $F_{\e}$ by
\begin{equation*}
F_{\e}(u)=\sum_{i=1}^{\lceil 1/\e\rceil-2}\min\left\{\e\Big|\frac{u(\e (i+1))-u(\e i)}{\e}\Big|^2,\beta\right\}+\gamma\sum_{i=1}^{\lceil 1/\e\rceil-1}\e |u(\e i)-g_{\e}(\e i)|^2,
\end{equation*}
where $g_{\e}$ is a suitable discretized version of $g\in L^{\infty}$. The proof of convergence to the one-dimensional version of the Mumford-Shah functional can be found for example in \cite[Chapter 8.3]{GCB}. The functional $F_{\e}$ above has a natural extension to higher dimensions. Indeed, given $u:\e\mathbb{Z}^d\cap D\to\R$, one sets
\begin{equation}\label{intro:discrete}
F_{\e}(u)=\frac{1}{2}\sum_{\substack{\e i,\e j\in\e\mathbb{Z}^d\cap D\\ |i-j|=1}}\e^{d-1}\min\left\{\e\Big|\frac{u(\e i)-u(\e j)}{\e}\Big|^2,\beta\right\}+\gamma\sum_{\e i\in\e\mathbb{Z}^d\cap D}\e^d |u(\e i)-g_{\e}(\e i)|^2.
\end{equation} 
However, in higher dimensions the anisotropy of the lattice $\mathbb{Z}^d$ leads to anisotropic surface integrals. For $d=2$ Chambolle proved in \cite{Ch95} that the functionals $F_{\e}$ $\Gamma$-converge to an anisotropic version of the Mumford-Shah functional given by
\begin{equation*}
F(u)=\int_D|\nabla u|^2\,\mathrm{d}x+\beta\int_{S_u}|\nu_u|_1\,\mathrm{d}\mathcal{H}^1+\gamma\int_D|u-g|^2\,\mathrm{d}x,
\end{equation*}
where $|\cdot|_1$ denotes the $l_1$-norm of the normal vector $\nu_u$ at $x\in S_u$ (we remark that with the results obtained in this paper the $\Gamma$-convergence above can be extended to any dimension). 

In order to avoid the anisotropy, there have been found two approaches: on the one hand, inspired by the nonlocal approximation of the Mumford-Shah functional studied in \cite{Go}, one can consider long-range interactions in (\ref{intro:discrete}) instead of only nearest neighbors. This has been analyzed in \cite{Ch99} and indeed anisotropy can be reduced but the functional to be minimized gets more complex as the number of interactions grows. On the other hand, Chambolle and Dal Maso considered functionals defined on piecewise affine functions with respect to a whole class of two-dimensional triangulations. More precisely, let $\mathcal{T}_{\e}(D,\theta)$ be the set of all finite triangulations containing $D$ such that for each triangle the inner angles are at least $\theta$ and the side lengths are between $\e$ and $w(\e)$, where $w(\e)\geq 6\e$ satisfies $\lim_{\e\to 0}w(\e)=0$. Denote by $V_{\e}(D,\theta)$ the set of continuous functions that are piecewise affine with respect to some $\mathcal{T}\in\mathcal{T}_{\e}(D,\theta)$. In \cite{ChDM} it is shown that there exists $0<\theta_0<60^{\circ}$ such that for all $0<\theta<\theta_0$ the functionals $F_{\e,\theta}$ defined on $V_{\e}(D,\theta)$ by
\begin{equation*}
F_{\e,\theta}(u)=
\frac{1}{\e}\int_D f(\e |\nabla u|^2)\,\mathrm{d}x+\gamma\int_D|u-g|^2\,\mathrm{d}x
\end{equation*}
$\Gamma$-converge to the Mumford-Shah functional when $f$ satisfies (\ref{intro:f}). In this case it holds that $\beta=f_{\infty}\sin(\theta)$. We remark that the triangulation is not fixed, so it is part of the minimization problem to find the optimal one (see \cite{BoCh} for details on numerical minimization for slightly modified functionals). Moreover, in contrast to the approximations mentioned before, this result is restricted to dimension two.
 
There are also different approaches to minimize the Mumford-Shah functional, for instance level-set methods, graph cut algorithms or convex relaxation techniques. We do not go into details but refer the reader to \cite{PCBC} and references therein.

The motivation for this work relies on the more recent paper \cite{StCr}, in which Cremers and Strekalovskiy propose another discrete functional based on finite differences along with a very fast algorithm to compute its minimizers in real-time. Although convergence of the algorithm has not been proven so far, it is demonstrated that it works well in practice. The discrete functional has a similar form to (\ref{intro:discrete}), but takes into account non-pairwise interactions and suitably scaled it reads as
\begin{equation}\label{fd:intro}
\tilde{F}_{\e,g}(u)=\sum_{\e x\in\e \mathbb{Z}^d\cap D}\e^{d-1}\min\Big\{\alpha\e \sum_{i=1}^d\Big|\frac{u(\e x+\e e_i)-u(\e x)}{\e}\Big|^2,1\Big\}+\gamma\sum_{\e x\in\e\mathbb{Z}^d\cap D}\e^d |u(\e x)-g_{\e}(\e x)|^2.
\end{equation}
The authors of \cite{StCr} conjecture that $\tilde{F}_{\e,g}$ approximates the Mumford-Shah functional.  

In this paper we analyze the asymptotic behavior of a more general family of discrete functionals. Inspired by the structure of $\tilde{F}_{\e,g}$ above, we allow the functionals to depend not only on pairwise but general finite differences. Furthermore, we replace the periodic lattice by so-called stochastic lattices. As a first consequence of our analysis, which is described more in detail below, we can identify the $\Gamma$-limit of the functionals $\tilde{F}_{\e,g}$ for $d=2$. In particular, we show that it differs from the Mumford-Shah functional due to an anisotropic surface integral (the last point can be verified in any dimension; see Remark \ref{higherdim}). Motivated by the fast algorithm for discrete approximations presented in \cite{StCr}, we then construct a random family of discrete functionals for that we can prove $\Gamma$-convergence to the Mumford-Shah functional almost surely (a.s.). The basic idea is quite simple: since the anisotropy in the $\Gamma$-limit of the family of functionals in (\ref{intro:discrete}) stems mostly from the lattice, we replace $\mathbb{Z}^d$ by a more isotropic point set. Since there exist no isotropic, countable sets, we need to go beyond deterministic models and consider realizations of random point sets. Those have the flexibility to be isotropic at least in distribution.
 
\bigskip
\noindent{\bf The main approximation result} 

\medskip
\noindent We consider random, countable point sets $\Lw\subset\R^d$ that satisfy the following geometric constraints:
\begin{itemize}
	\item[(i)] There exists $R>0$ such that $\dist(x,\mathcal{L}(\w))< R$ for all $x\in\R^d$;
	\item [(ii)] There exists $r>0$ such that $\dist(x,\Lw\backslash\{x\})\geq r$ for all $x\in\Lw$.
\end{itemize}
Given a small parameter $\e>0$ (again representing a kind of mesh-size) and $q>1$, one possible approximation is the family of random functionals $F_{\e,g}(\w)$ defined on functions $u:\e\Lw\cap D\to\R^m$ by 
\begin{equation}\label{intro:defF}
F_{\e,g}(\w)(u)=\sum_{\substack{(x,y)\in\mathcal{N}(\w)\\ \e x,\e y\in D}}\e^{d-1}f\Big(\e\Big|\frac{u(\e x)-u(\e y)}{\e}\Big|^{2}\Big)+\sum_{\e x\in\e\Lw\cap D}\e^d|u(\e x)-g_{\e}(\w)(\e x)|^q,
\end{equation}
where $\mathcal{N}(\w)$ denotes the set of Voronoi neighbors (see Definition \ref{defgoodedges}), $f$ is a function satisfying (\ref{intro:f}) and $g_{\e}(\w)$ is a suitable discretization of some given $g\in L^q(D,\R^m)$. We require that the random point set $\mathcal{L}$ is stationary and isotropic, that means $\mathcal{L}$ and $R\mathcal{L}+z$ have the same statistics for all $z\in\mathbb{Z}^d$ and all $R\in SO(d)$. If the shift operation is realized by an ergodic group action and $g_{\e}(\w)\to g$ in $L^q(D,\R^m)$, our main result, which is stated in full generality in Theorem \ref{MSapprox}, can be summarized as follows:
\begin{theorem*}
Under the above assumptions, there exist three positive constants $c_1,c_2,c_3$ such that with probability $1$ the functionals $F_{\e,g}(\w)$ $\Gamma$-converge with respect to the $L^1(D,\R^m)$-topology to the deterministic functional $F_g$ defined by
\begin{equation*}
F_g(u)=
\begin{cases}
\displaystyle c_1\int_D|\nabla u|^2\,\mathrm{d}x+c_2\mathcal{H}^{d-1}(S_u)+c_3\int_D|u-g|^q\,\mathrm{d}x &\mbox{if $u\in L^q(D,\R^m)\cap GSBV^2(D,\R^m)$,}
\\
+\infty &\mbox{otherwise.}
\end{cases}
\end{equation*} 
\end{theorem*}
Given the probabilistic assumptions above, the result is quite robust. For example, the same limit (with different constant $c_2$) can be proven for the random version of (\ref{fd:intro}) given by
\begin{equation*}
\tilde{F}_{\e,g}(\w)=\sum_{\e x\in\e\Lw\cap D}\e^{d-1}f\Big(\sum_{\substack{(x,y)\in\mathcal{N}(\w)\\ \e y\in D}}\e\Big|\frac{u(\e x)-u(\e y)}{\e}\Big|^{2}\Big)+\sum_{\e x\in\e\Lw\cap D}\e^d|u(\e x)-g_{\e}(\w)(\e x)|^q.
\end{equation*} 
Some remarks are in order:
\begin{itemize}
	\item[(i)] A point process that satisfies all our assumptions is given by the random parking process \cite{glpe,Pe}.
	\item[(ii)] The coefficients $c_i$ are not explicit but are derived from three abstract homogenization formulas. However, for fixed $\mathcal{L}$ one can still tune them since $c_1$ and $c_2$ are proportional to $f^{\prime}(0)$ and $f_{\infty}$ respectively, while for $c_3$ we can multiply the second term in (\ref{intro:defF}) by some factor.
	\item[(iii)] We will prove the convergence for more general finite differences (see Section \ref{Sec:prelim}). Those require more technical notation that we want to avoid in this introduction. Hence we restrict the description of our analysis below to pairwise interactions via Voronoi neighbors.
	\item[(iv)] In the proof we will identify $u$ with a piecewise constant function on Voronoi cells. However this is not needed for minimizing the functional $F_{\e,g}(\w)$. In particular one only has to determine the Voronoi neighbors, but no volume of cells or piecewise affine interpolations on Delaunay triangulations. One can also avoid the determination of the Voronoi neighbors using a $k$-NN algorithm, but $k$ should not be too small (see also Remark \ref{r.measurability} (ii) \& (iii)).
	\item[(v)] We prove that global minimizers of $F_{\e,g}$ converge to minimizers of $F_g$ in $L^q(D,\R^m)$. Note that this is not the natural compactness to be expected from finite energy sequences.
	\item[(vi)] The discrete functionals are still non-convex which cannot be avoided since the $\Gamma$-limit of any sequence of convex functionals remains convex.
	\item[(vii)] In our setting the case of vector-valued $u$ corresponds to color images. Our arguments cannot be generalized straightforward to models for linearized elasticity.
	\item[(viii)] A different randomization of the functional (\ref{intro:discrete}) has been considered in another context in \cite{BraPi}, where a random choice between the potential $f(s)=\min\{s,1\}$ and $\tilde{f}(s)=s$ is analyzed. However, isotropy of the limit functional remained an open problem.
	\item[(ix)] Another approach to construct discrete approximations could be based on random point clouds similar to \cite{GTS}, where the authors prove an approximation result for total variation-type functionals. Point clouds have the advantage that they can be generated very fast. However, for point clouds one usually needs interactions with range $\sim (\log(n)/n)^{\frac{1}{d}}$ for $n$ points compared to $\sim(1/n)^{\frac{1}{d}}$ for stochastic lattices. Otherwise the corresponding graph will not be connected. %Mathematically this means that discrete functionals on point clouds probably can be compared via the techniques developed in \cite{GTS} to the non-local functionals discussed in \cite[Section 7]{Go}. 
	Besides the nonlocal structure there are also many redundant points due to clustering.
\end{itemize}
\medskip
\noindent {\bf Plan of the paper}

\medskip
\noindent We now give a short overview of the paper and explain briefly the steps to prove our main approximation theorem. Section \ref{Sec:prelim} is divided into three preliminary parts. First we recall the necessary function spaces that we need for our analysis. In the second part we introduce in a rigorous way the stochastic point sets that we use to define our approximating functionals $F_{\e,g}(\w)$. In the last part we introduce the class of functionals under consideration (we omit the fidelity term in most parts of the paper). For this introduction we assume the functionals to be of the form 
\begin{equation*}
F_{\e}(\w)(u)=\sum_{\substack{(x,y)\in\mathcal{N}(\w)\\ \e x,\e y\in D}}\e^{d-1}f\left(\e\Big|\frac{u(\e x)-u(\e y)}{\e}\Big|^{p}\right),
\end{equation*}
where for the sake of generality we take a general exponent $p>1$. The localized versions of these functionals will be the main objects to be studied in the subsequent sections.

In Section \ref{s.presentation} we present three general results that we prove on the way towards the main approximation result and which might be of independent interest. Assuming only the geometric properties of a single realization $\Lw$, we prove in Theorem \ref{t.sepofscales} that (up to subsequences) the $\Gamma$-limit of the family $F_{\e}(\w)$ always has the form of a free discontinuity functional, that means it is finite only on $GSBV^p(D,\R^m)$, where it can be written as
\begin{equation}\label{intro:subseq}
F(\w)(u)=\int_Dh(x,\nabla u)\,\mathrm{d}x+\int_{S_u}\varphi(x,\nu_u)\,\mathrm{d}\mathcal{H}^{d-1}.
\end{equation}
Moreover, the density $h$ of the $\Gamma$-limit coincides with the density of the $\Gamma$-limit of the convex functionals 
\begin{equation*}
E_{\e}(\w)(u)=f^{\prime}(0)\sum_{\substack{(x,y)\in\mathcal{N}(\w)\\ \e x,\e y\in D}}\e^{d}\Big|\frac{u(\e x)-u(\e y)}{\e}\Big|^{p},
\end{equation*}
while the surface density $\varphi$ of the functional $F(\w)$ in \eqref{intro:subseq} is given by surface density of the $\Gamma$-limit of the Ising-type energies defined on functions $u:\e\Lw\to\{\pm e_1\}$ by
\begin{equation*}
I_{\e}(\w)(u)=\frac{f_{\infty}}{2}\sum_{\substack{(x,y)\in\mathcal{N}(\w)\\ \e x,\e y\in D}}\e^{d-1}|u(\e x)-u(\e y)|.
\end{equation*}
In Theorem \ref{mainthm1} we state a general stochastic homogenization result for the functionals $F_{\e}(\w)$ in the case of stationary, ergodic stochastic lattices. In particular, the $\Gamma$-limit of $F_{\e}(\w)$ exists a.s., is deterministic and on its domain it is of the form
\begin{equation*}
F(u)=\int_Dh(\nabla u)\,\mathrm{d}x+\int_{S_u}\varphi(\nu_u)\,\mathrm{d}\mathcal{H}^{d-1},
\end{equation*}
that means in contrast to \eqref{intro:subseq} the densities $h$ and $\varphi$ do not depend on $x$ and are deterministic. Theorem \ref{convfull} contains the $\Gamma$-convergence including the convergence of minimizers when we add the discrete fidelity term in the stationary, ergodic setting.

In Section \ref{s.applications} we apply the results of Section \ref{s.presentation}. On the one hand, we identify the $\Gamma$-limit of the functionals in  \eqref{fd:intro}. On the other hand, with Theorem \ref{MSapprox} we obtain our main approximation result about the Mumford-Shah functional  when we assume additionally that the stochastic lattice is isotropic in distribution.

Section \ref{s.proofs} contains the proof of Theorem \ref{t.sepofscales}. While the integral form \eqref{intro:subseq} of any $\Gamma$-limit
is obtained by standard techniques combining the abstract methods of $\Gamma$-convergence with an integral representation theorem, the characterizations of the integrands $h$ and $\varphi$ by the $\Gamma$-limits of the sequences $E_{\e}(\w)$ and $I_{\e}(\w)$ is the most delicate step in this paper. Although similar results have been obtained in a continuum setting (see \cite{BrDeVi, CDMSZ17, GiPo}), we cannot use interpolation and copy the argument. This has several reasons: on the one hand, in dimensions larger than two, piecewise affine interpolations on Delaunay tessellations might be degenerate due to very flat tetrahedrons. On the other hand, even in a planar setting Voronoi cells can have very short boundary sides, so that the discrete functional overestimates the length of interfaces. Moreover, fine constructions based on geometric measure theory can be incompatible with the prescribed lattice structure. Thus our arguments, which are nevertheless inspired by the continuum case, need to use the discrete environment as long as possible. The complete strategy is explained more in detail at the beginning of Section \ref{s.proofs}.

In Section \ref{s.stochhom} we prove the remaining results of Section \ref{s.presentation}. With the characterizations proven in the previous section, Theorem \ref{mainthm1} is a straightforward consequence of the results on discrete-to-continuum stochastic homogenization for elastic and Ising-type energies obtained in the two papers \cite{ACG2,ACR}, respectively. Also adding the fidelity term is quite straightforward.

The appendix contains a technical argument how to choose $\Gamma$-converging diagonal sequences in our special setting when the functionals are not equicoercive.

\section{Setting of the problem and preliminaries }\label{Sec:prelim}
We first introduce some notation that will be used in this paper. Given a measurable set $B\subset\R^d$ we denote by $|B|$ its $d$-dimensional Lebesgue measure, while more generally $\mathcal{H}^{k}(A)$ stands for the $k$-dimensional Hausdorff measure. We denote by $\mathds{1}_B$ the characteristic function of $B$. If $B$ is finite, $\#B$ means its cardinality. Given an open set $O\subset\R^d$, we denote by $\mathcal{A}(O)$ the family of all bounded, open subsets of $O$, while $\mathcal{A}^R(O)$ means the bounded, open subsets with Lipschitz boundary. For $x\in\R^d$ or $y\in\R^m$ we denote by $|x|$ and $|y|$ the Euclidean norm. Given a matrix $\xi\in\mathbb{R}^{m\times d}$, we let $|\xi|$ be its Frobenius norm. As usual $B_{\varrho}(x_0)$ denotes the open ball with radius $\varrho$ centered at $x_0\in\R^d$. We simply write $B_{\varrho}$ when $x_0=0$. Given $\nu\in S^{d-1}$, we let $\nu_1=\nu,\nu_2,\dots,\nu_d$ be an orthonormal basis of $\R^d$ and we define the cube $Q_{\nu}$ as
\begin{equation*}
Q_{\nu}=\left\{z\in\R^d:\;|\langle z,\nu_i\rangle| <1/2\right\},
\end{equation*}
where the brackets $\langle\cdot\rangle$ denote the scalar product. Given $x_0\in\R^d$ and $\varrho>0$, we set $Q_{\nu}(x_0,\varrho)=x_0+\varrho Q_{\nu}$. For $x_0\in \R^d$, $\nu\in S^1$ and $a,b\in\R^m$ we define the function $u_{x_0,\nu}^{a,b}$ by the formula
\begin{equation}\label{eq:purejump}
u_{x_0,\nu}^{a,b}(x):=\begin{cases} a &\mbox{if $\langle x-x_0,\nu\rangle >0$,}
\\
b &\mbox{otherwise.}
\end{cases}
\end{equation}
The notation ${\rm co}(x_1,\dots,x_d)$ means the convex hull of finitely many points in $\R^d$. We will use $\|u\|_{L^p(A)}$ for the $L^p(A,\R^m)$-norm. There should be no confusion about the dimension $m$. The symbol $\otimes$ stands for the outer product of vectors, that is, for any $a\in\R^m,b\in\R^{d}$ we have $a\otimes b\in\mathbb{R}^{m\times d}$ with $(a\otimes b)_{ij}:=a_{i}b_{j}$.
Finally, the letter $C$ stands for a generic positive constant that may change every time it appears.
\subsection{Generalized special functions of bounded variation}\label{subsec:bv}
We briefly recall the function spaces we are going to use in this paper. We refer to \cite{Amb,AFP,DMFrTo} for more details. 

Let $O\subset\R^d$ be an open set. We denote by $BV(O,\R^m)\subset L^1(O,\R^m)$ the space of vector-valued {\it functions of bounded variation}. We write $Du$ for the matrix-valued distributional derivative of $u$, which can be decomposed as
\begin{equation*}
Du(B)=\int_B\nabla u\,\mathrm{d}x+\int_{S_u\cap B}(u^+(x)-u^-(x))\otimes\nu_u(x)\,\mathrm{d}\mathcal{H}^{d-1}+D^cu(B).
\end{equation*}
In the above formula $\nabla u$ denotes the absolutely continuous part of $Du$ with respect to the Lebesgue measure, $S_u$ is the so-called jump set of $u$ with (measure-theoretic) normal vector $\nu_u\in S^{d-1}$ and $u^-,u^+$ denote the one-sided traces of $u$ at $\mathcal{H}^{d-1}$-a.e. $x\in S_u$. The remainder $D^cu$ is called the Cantor-part, but it will play no role in this paper.

Indeed, the space of {\it special functions of bounded variation} is defined as the set of those $u\in BV(O,\R^m)$ such that $D^cu=0$. We write $u\in SBV(O,\R^m)$. Given $p\in(1,+\infty)$, we define $SBV^p(O,\R^m)\subset SBV(O,\R^m)$ as the set of those functions such that $\nabla u\in L^p(O,\mathbb{R}^{m\times d})$ and $\mathcal{H}^{d-1}(S_u)<+\infty$. Due to a lack of compactness in many free discontinuity problems, we have to enlarge this class. We say that a Borel-function $u:O\to\R^m$ is a {\it generalized special function of bounded variation}, if $\Phi\circ u\in SBV_{\rm loc}(O,\R^m)$ for every function $\Phi\in C^1(\R^m,\R^m)$ such that $\nabla\Phi$ has compact support and write $u\in GSBV(O,\R^m)$. In this case, the approximate differential $\nabla u(x)$ still exists a.e. and there is a well-defined jump set $S_u$, which is countably $(d-1)$-rectifiable. Finally, we set $GSBV^p(O,\R^m)$ as those functions $u\in GSBV(O,\R^m)$ such that $\nabla u\in L^p(O,\mathbb{R}^{m\times d})$ and $\mathcal{H}^{d-1}(S_u)<+\infty$. As shown in \cite[Section 2]{DMFrTo}, the set $GSBV^p(O,\R^m)$ is a vector space and, if $\Phi\in C^1(\R^m,\R^m)$ is such that $\nabla\Phi$ has compact support, then $\Phi\circ u\in SBV^p(O,\R^m)$. Moreover, given $u\in GSBV^p(O,\R^m)$, one can define a Borel-function $\nu_u:S_u\to S^{d-1}$ and two Borel-functions $u^+,u^-:S_u\to\R^m$ still satisfying a weak trace condition for $\mathcal{H}^{d-1}$-a.e. $x\in S_u$.

For our analysis we make use of a special family of smooth truncations as in \cite{CDMSZ17} which essentially allows to reduce many proofs to the space $SBV^p$. Consider a function $\phi\in C^{\infty}_c(\R)$ such that $\phi(t)=t$ for all $|t|\leq 1$, $\phi(t)=0$ for $t\geq 3$ and $\|\phi^{\prime}\|_{\infty}\leq 1$. We define the function $\Phi\in C^{\infty}_c(\R^m,\R^m)$ by
\begin{equation*}
\Phi(u)=\begin{cases}
\phi(|u|)\frac{u}{|u|} &\mbox{if $u\neq 0$,}
\\
0 &\mbox{if $u=0$.}
\end{cases}
\end{equation*} 
As shown at the beginning of \cite[Section 4]{CDMSZ17} the function $\Phi$ is $1$-Lipschitz. Given $k>0$ we further set $\Phi_k(u)=k\Phi(\frac{u}{k})$, which is still $1$-Lipschitz. Then we have the following approximation result, which is a consequence of dominated convergence and \cite[Propositions 1.1-1.3 \& Theorem 3.7]{Amb}.
\begin{lemma}\label{truncation}
Let $u\in GSBV^p(D,\R^m)\cap L^1(D,\R^m)$ and let $k>0$. Defining the truncation $T_k u=\Phi_k(u)$, the function $T_k u$ belongs to $SBV^p(D,\R^m)\cap L^{\infty}(D,\R^m)$ and 
\begin{itemize}
\item[(i)] $\displaystyle\lim_{k\to +\infty}T_ku = u\quad$ a.e. and in $L^1(D,\R^m)$;
\item[(ii)] $\nabla T_ku(x)=\nabla\Phi_k(u(x))\nabla u(x)\quad$ for a.e. $x\in D$;
\item[(iii)] $\displaystyle S_{T_ku}\subset S_u$, $\lim_{k\to +\infty}\mathcal{H}^{d-1}(S_{T_ku})=\mathcal{H}^{d-1}(S_u)$ and $\nu_{u}=\pm\nu_{T_ku}\quad$  $\mathcal{H}^{d-1}$-a.e. on $S_{T_ku}$.
\end{itemize}
\end{lemma}
%\begin{proof}
%(i) follows by dominated convergence. (ii) is a consequence of \cite[Proposition 1.2 (i)]{Amb}. The  inclusion in (iii) is shown in \cite[Proposition 1.1 (iii)]{Amb}, while the convergence is a consequence of the inclusion and lower semicontinuity \cite[Theorem 3.7]{Amb}. The orientation property holds by the construction of the normal fields in \cite[Proposition 1.3]{Amb}.
%\end{proof}

\subsection{Stochastic lattices}
Next we introduce the random point sets that we use for the discrete approximations. Throughout this paper we let $\Omega$ be a probability space with a complete $\sigma$-algebra $\mathcal{F}$ and probability measure $\mathbb{P}$. We call a random variable $\mathcal{L}:\Omega\to(\R^d)^{\mathbb{N}}$ a stochastic lattice. The following definition, which has been introduced in \cite{BLBL} in the context of quantum models (and is known in a deterministic setting as Delone sets), essentially forbids clustering of points as well as arbitrarily big empty regions in space. 

\begin{definition}[Admissible lattices]\label{defadmissible}
Let $\mathcal{L}$ be a stochastic lattice. $\mathcal{L}$ is called admissible if there exist $R>r>0$ such that the following two conditions hold a.s.:
\begin{itemize}
	\item[(i)] $\dist(x,\mathcal{L}(\w))< R\quad$ for all $x\in\R^d$;
	\item [(ii)] $\dist(x,\Lw\backslash\{x\})\geq r\quad$ for all $x\in\Lw$.
\end{itemize}
\end{definition}

\begin{remark}\label{voronoi}
We also make use of the associated Voronoi tessellation $\mathcal{V}(\w)=\{\mathcal{C}(x)\}_{x\in\Lw}$, where the (random) Voronoi cells with nuclei $x\in\Lw$ are defined as
\begin{equation*}
\mathcal{C}(x)=\{z\in\R^d:\;|z-x|\leq |z-y|\quad\text{for all }y\in\Lw\}.
\end{equation*}	
If $\Lw$ is admissible, then \cite[Lemma 2.3]{ACR} yields the inclusions $B_{\frac{r}{2}}(x)\subset \mathcal{C}(x)\subset B_R(x)$.
\end{remark}
Next we introduce some notions from ergodic theory that build the basis for stochastic homogenization.

\begin{definition}
	We say that a family of measurable functions $\{\tau_z\}_{z\in \mathbb{Z}^d},\tau_z:\Omega\to\Omega$, is an additive group action on $\Omega$ if
	\begin{equation*}
	\tau_0={\rm id}\quad\text{and}\quad\tau_{z_1+z_2}=\tau_{z_2}\circ\tau_{z_1}\quad\text{for all  } z_1,z_2\in\mathbb{Z}^d.
	\end{equation*} 
	An additive group action is called measure preserving if
	\begin{equation*}
	\mathbb{P}(\tau_z^{-1} B)=\mathbb{P}(B)\quad \text{ for all  } B\in\mathcal{F},\,z\in\mathbb{Z}^d.
	\end{equation*}
	Moreover, $\{\tau_z\}_{z\in\mathbb{Z}^d}$ is called ergodic if, in addition, for all $B\in\mathcal{F}$ we have the implication
	\begin{equation*}
	(\tau_z(B)=B\quad\forall z\in \mathbb{Z}^d)\quad\Rightarrow\quad\mathbb{P}(B)\in\{0,1\}.
	\end{equation*}
\end{definition}

In terms of a stochastic lattice the probabilistic properties read as follows: 

\begin{definition}\label{defstatiolattice}
	A stochastic lattice $\mathcal{L}$ is said to be stationary if there exists an additive, measure preserving group action $\{\tau_z\}_{z\in\mathbb{Z}^d}$ on $\Omega$ such that for all $z\in\mathbb{Z}^d$
	\begin{equation*}
	\mathcal{L}\circ\tau_z=\mathcal{L}+z.
	\end{equation*}
	If in addition $\{\tau_z\}_{z\in\mathbb{Z}^d}$ is ergodic, then $\mathcal{L}$ is called ergodic, too.
	\\
	We call $\mathcal{L}$ isotropic, if for every $R\in SO(d)$ there exists a measure preserving function $\tau^{\prime}_R:\Omega\to\Omega$ such that
	\begin{equation*}
	\mathcal{L}\circ\tau^{\prime}_R=R\mathcal{L}.
	\end{equation*}
\end{definition}

In order to define gradient-like structures, we equip a stochastic lattice with a set of directed edges. We summarize the necessary properties in a separate definition:
\begin{definition}[Admissible edges]\label{defgoodedges}
Let $\mathcal{L}$ be an admissible stochastic lattice and $\mathcal{E}\subset\mathcal{L}^2$. We call $\mathcal{E}$ admissible edges if for all $i,j\in\mathbb{N}$ the set $\{\w\in\Omega:\;(\Lw_i,\Lw_j)\in\mathcal{E}(\w)\}$ is $\mathcal{F}$-measurable and
\begin{itemize}
	\item[(i)]  there exists $M>R$ such that a.s. \begin{equation}\label{finiterange}
	\sup\{|x-y|:\; (x,y)\in\mathcal{E}(\w)\}< M;
	\end{equation}
	\item[(ii)] the Voronoi neighbors defined by $\mathcal{N}(\w):=\{(x,y)\in\Lw^2:\;\mathcal{H}^{d-1}(\mathcal{C}(x)\cap\mathcal{C}(y))\in (0,+\infty)\}$ are contained in $\mathcal{E}(\w)$ up to symmetrizing, that means
	\begin{equation}\label{nncontained}
	\mathcal{N}(\w)\subset \mathcal{E}(\w)\cup \{(y,x)\in\Lw^2:\;(x,y)\in\mathcal{E}(\w)\}.
	\end{equation}
\end{itemize}
If $\mathcal{L}$ is stationary or isotropic, we say that the edges $\mathcal{E}$ are stationary or isotropic if $\mathcal{E}\circ\tau_z=\mathcal{E}+(z,z)$ for all $z\in\mathbb{Z}^d$ or $\mathcal{E}\circ\tau^{\prime}_R=R\mathcal{E}$ for all $R\in SO(d)$. 	
\end{definition}
Up to enlarging $M$, by Remark \ref{voronoi} we may assume in addition that
\begin{equation}\label{neighbours}
\sup_{x\in\Lw}\#\{y\in\Lw:\;(x,y)\in\mathcal{E}(\w)\text{ or }(y,x)\in\mathcal{E}(\w)\}\leq M.
\end{equation}
\begin{remark}\phantomsection\label{r.measurability}
\begin{itemize}
	\item[(i)] In the proof of \cite[Lemma A.2]{ACR} it is shown, that the choice $\mathcal{E}(\w)=\mathcal{N}(\w)$ satisfies the measurability assumption. Then we can add for example non-Voronoi neighbors by selecting them based on a maximal distance.
	\item[(ii)] In order to avoid the computation of the Voronoi neighbors (or the Delaunay triangulation), one can also use a $k$-NN algorithm (taking all points in case of a tie to avoid anisotropy). By a naive volume bound based on Remark \ref{voronoi} it suffices to take
	\begin{equation*}
	k= \left\lceil\left(4Rr^{-1}+1\right)^d-\left(2Rr^{-1}-1\right)^d\right\rceil-2.
	\end{equation*}
	Indeed, if $(x,y)\in\mathcal{N}(\w)$ is a pair of Voronoi neighbors, then for any point $z\in\mathcal{C}(x)\cap\mathcal{C}(y)$ it holds that $|z-x|< R$ and the ball $B_{|x-z|}(z)$ is contained in $B_{2|x-z|}(x)$ and contains no other point of $\Lw$. For all other points $x^{\prime}\in\Lw$ with $|x^{\prime}-x|\leq |y-x|$ (including $x$) the balls $B_{\frac{r}{2}}(x^{\prime})$ are pairwise disjoint and contained in the set $B_{2|x-z|+\frac{r}{2}}(x)\backslash B_{|x-z|-\frac{r}{2}}(z)$. The volume of this set is monotone in $|x-z|\geq r/2$, which implies the claimed bound. For the random parking model we have the optimal ratio $R/r=2$ at the packing limit, which yields $k=70$ for $d=2$. This bound is far from being optimal. Nevertheless, a more detailed treatment of numerical issues is beyond the scope of this paper.
	\item[(iii)] We use \eqref{nncontained} only for proving Lemma \ref{l.paths}. Hence in \eqref{nncontained} the set $\mathcal{N}(\w)$ can be replaced by any set of edges such that this lemma remains valid.  
\end{itemize}	
\end{remark}

Having introduced the random framework, we will need it again only in Section \ref{s.stochhom}. To simplify the notation, we will drop the dependence on $\w$ for some quantities. If so, we tacitly assume that we have a realization satisfying properties (i) and (ii) of both Definitions \ref{defadmissible} and \ref{defgoodedges}.

\subsection{A generalized weak-membrane energy}
In order to discretize vectorial Mumford-Shah-type functionals we basically follow the approach used on periodic lattices for the scalar case. However we go beyond pairwise interactions. Due to the possibly non-ordered edges $\mathcal{E}(\w)$ this requires some notation. For $M\in\mathbb{N}$ satisfying (\ref{neighbours}) we denote by 
\begin{equation*}
\mathcal{P}_{+}(M)=\{\mathfrak{p}:[0,+\infty)\to\mathbb{N}_0:\,\#\mathfrak{p}^{-1}(\mathbb{N})<+\infty,\sum_{v\in \mathfrak{p}^{-1}(\mathbb{N})}\mathfrak{p}(v)\leq M\}
\end{equation*}
the set of all multisets over $[0,+\infty)$ with at most $M$ elements. Note that if $v_1,\dots, v_k\in[0,+\infty)$ with $k\leq M$, then up to permutation we can identify these points with a unique $\mathfrak{p}\in\mathcal{P}_+(M)$ setting $\mathfrak{p}(v_i)=\#\{j:\,v_j=v_i\}$ for all $1\leq i\leq k$ and zero elsewhere. In this sense we sometimes use the more common notation $\mathfrak{p}=\{v_i\}_b$, where the $b$ indicates a badge in which elements can occur several times in contrast to ordinary sets. For $\mathfrak{p}\in\mathcal{P}_+(M)$ we set $\|\mathfrak{p}\|_1=\sum_{v\in \mathfrak{p}^{-1}(\mathbb{N})}\mathfrak{p}(v)v$. In this paper we fix a bounded function $f:\mathcal{P}_{+}(M)\to [0,+\infty)$ satisfying the following two structural assumptions: there exists $0<\alpha<+\infty$ such that
\begin{equation}\label{linearatzero}
\lim_{\|\mathfrak{p}\|_1\to 0}\frac{f(\mathfrak{p})}{\|\mathfrak{p}\|_1}=\alpha
\end{equation}
and $f$ is monotone increasing in the following sense: for all $v,v^{\prime}\in [0,+\infty)^M$ with $v_i\leq v_i^{\prime}$ for all $1\leq i\leq M$ and for any $1\leq k\leq M$ we have
\begin{equation}\label{monotone}
f(\{v_1,\dots,v_k\}_b)\leq f(\{v_1^{\prime},\dots,v_k^{\prime}\}_b).
\end{equation} 
We also assume that $f$ is lower semicontinuous in the sense that for all sequences $(v^n)_n\subset[0,+\infty)^M$ converging to some $v\in[0,+\infty)^M$ and for all $1\leq k\leq M$ it holds that
\begin{equation}\label{flsc}
f(\{v_1,\dots,v_k\}_b)\leq \liminf_{n\to +\infty}f(\{v_1^n,\dots,v_k^n\}_b).
\end{equation}

\begin{remark}\label{onf}
Note that from the boundedness of $f$, the monotonicity assumption and the property (\ref{linearatzero}) it follows that there exist constants $C_f>c_f>0$ such that
\begin{equation}\label{realcut}
c_f\min\{\|p\|_1,1\}\leq f(p)\leq C_f\min\{\|p\|_1,1\}.
\end{equation}
Moreover, again by boundedness and monotonicity, for every $1\leq l\leq k\leq M$ there exists the limit
\begin{equation}\label{limitatinf}
\beta(l,k)=\lim_{N\to +\infty}f(l\mathds{1}_{\{N\}}+(k-l)\mathds{1}_{\{0\}})>0.
\end{equation}
While we will frequently use the property (\ref{monotone}) for our analysis, the lower semicontinuity (\ref{flsc}) will just guarantee the existence of minimizers for the discrete approximations.
\end{remark}

In order to define the discrete approximation of a continuous functional we scale a stochastic lattice by a small parameter $\e>0$. Let us fix a reference set $D\subset\R^d$, which we assume to be a bounded Lipschitz domain, and a growth exponent $p\in (1,+\infty)$. Given $u:\e\Lw\to\R^m$, an open set $A\in\mathcal{A}(D)$ and $\eta>0$, we define the function $\eta|\nabla_{\w,\e}(u,A)|^p:\e\Lw\to\mathcal{P}_{+}(M)$ by 
\begin{equation*}
\eta|\nabla_{\w,\e}(u,A)|^p(\e x) = \Big\{\eta\Big|\frac{u(\e x)-u(\e y)}{\e}\Big|^p:\;(x,y)\in\mathcal{E}(\w),\,\e x,\e y\in A\Big\}_b.
\end{equation*}
Then we define the localized discrete approximations (which we also call energies) as
\begin{equation*}
F_{\e}(\w)(u,A)=\sum_{\e x\in \e\Lw\cap A}\e^{d-1} f\left(\e|\nabla_{\w,\e}(u,A)|^p(\e x)\right).
\end{equation*}
We simply write $F_{\e}(\w)$ for $F_{\e}(\w)(\cdot,D)$, which will be the functional of interest in this paper.
\begin{remark}\label{examples}
We chose the abstract framework	above for two reasons.
\begin{itemize}
	\item[(i)] We take directed edges to define $\eta|\nabla_{\w,\e}(u,A)|^p(\e x)$ in order to include the functional (\ref{fd:intro});
	\item[(ii)] we define the function $f$ on multisets to handle pairwise and non-pairwise gradients simultaneously. For pairwise interactions, we set $f(\mathfrak{p})=\sum_{v\in \mathfrak{p}^{-1}(\mathbb{N})}\mathfrak{p}(v)f_0(v)$ with $f_0$ satisfying (\ref{intro:f}). The other example of the introduction is given by $f(\mathfrak{p})=f_0(\|\mathfrak{p}\|_1)$. 
\end{itemize}
\end{remark}
As we aim at using the abstract theory of $\Gamma$-convergence, we will identify a discrete variable with its piecewise constant interpolation on the Voronoi cells, that means with functions of the class
\begin{equation*}
\mathcal{PC}_{\e}^{\w}=\{u:\R^d\to\R^m:\;u_{|\e \mathcal{C}(x)}\text{ is constant for all }x\in\Lw\}.
\end{equation*}
With a slight abuse of notation we extend the functional to $F_{\e}(\w):L^{1}(D,\R^m)\times\mathcal{A}(D)\to[0,+\infty]$ by
\begin{equation}\label{deffunctional}
F_{\e}(\w)(u,A)=
\begin{cases} 
F_{\e}(\w)(u,A) &\mbox{if $u\in\mathcal{PC}_{\e}^{\w}$,}\\
+\infty &\mbox{otherwise.}
\end{cases}
\end{equation}

\begin{remark}\label{lp}
We work in the space $L^{1}$ due to the applications we have in mind. A priori there is no equicoercivity in this space and therefore it seems more natural to define the functionals for example on measurable functions. However, we will show in Lemma \ref{compact} that this can be circumvented using a fidelity term that is part of the Mumford-Shah functional anyway.
\end{remark}

\section{Presentation of the general results}\label{s.presentation}
Having introduced the necessary notation, we now give an overview of the theoretical results that we prove on the way to obtain the discretization of the Mumford-Shah functional. Readers who are interested specifically in the Mumford-Shah functional can go to Section \ref{subsec:MS}.

%\subsection*{Individual integral representation} 

%\begin{theorem}\label{mainrep}
%Given any sequence $\e\to 0$ there exists a subsequence $\e_n$ such that for all $A\in\Ard$ the functionals $F_{\e_n}(\w)(\cdot,A)$ $\Gamma$-converge in the $L^1(D,\R^m)$-topology to a functional $F(\w)(\cdot,A):L^1(D,\R^m)\to[0,+\infty]$ that is finite only for $A\in\Ard$ such that $u\in GSBV^p(A,\R^m)$. For such $(u,A)$ it can be written as
%\begin{equation*}
%F(\w)(u,A)=
%\int_A h(x,\nabla u(x))\dx+\int_{S_u\cap A}\varphi(x,\nu)\,\mathrm{d}\mathcal{H}^{d-1},
%\end{equation*}
%where, for $x_0\in D$, $\nu\in S^1$ and $\xi\in\R^{d\times m}$, the integrands are given by
%\begin{equation*}
%\begin{split}
%h(x,\xi)&=\limsup_{\rho\to 0}\frac{1}{\rho_d}m(\w)(\xi(\cdot-x_0),Q_{\nu}(x_0,\rho)),
%\\
%\varphi(x_0,\nu)&=\limsup_{\rho\to 0}\frac{1}{\rho^{d-1}}m(\w)(u_{x_0,\nu}^{e_1,0},Q_{\nu}(x_0,\rho)),
%\end{split}
%\end{equation*} 
%with the function $u_{x_0,\nu}^{a,b}$ defined in \eqref{eq:purejump}
%and the function $m(\w)(v,A)$ defined for any $\bar{u}\in SBV^p(D,\R^m)$ and $A\in\Ard$ by
%\begin{equation*}
%m(\w)(\bar{u},A)=\inf\{F(\w)(v,A):\;v\in SBV^p(A,\R^m),\,v=\bar{u}\text{ in a neighborhood of }\partial A\}.
%\end{equation*}
%\end{theorem}

\subsection{Integral representation and separation of bulk and surface effects}
Our first result establishes a general $\Gamma$-convergence result (up to subsequences), which holds without any probabilistic assumptions on the lattice, but pointwise for admissible lattices and edges. We show that any $\Gamma$-limit has an integral representation which is characterized by the separate $\Gamma$-limits of an associated bulk energy and a surface energy.  More precisely, a formal linearization of $F_{\e}(\w)$ at $u=0$ yields the bulk energy $E_{\e}(\w)(\cdot,A):\mathcal{PC}_{\e}^{\w}\to[0,+\infty]$ defined by
\begin{equation}\label{auxelastic}
E_{\e}(\w)(u,A)=
\alpha\sum_{\substack{(x,y)\in\mathcal{E}(\w)\\ \e x,\e y\in A}}\e^d\Big|\frac{u(\e x)-u(\e y)}{\e}\Big|^p,
\end{equation}
where $\alpha$ is given by (\ref{linearatzero}). At least when $u$ is the discretization of a Lipschitz function, the functional $E_{\e}(\w)(u,A)$ should be a good approximation of $F_{\e}(\w)(u,A)$ since $\|\e|\nabla_{\w,\e}(u,A)|^p\|_1$ vanishes when $\e\to 0$. On the contrary, when $u$ is the discretization of a (macroscopically) piecewise constant function, the elements in $\e|\nabla_{\w,\e}(u,A)|^p$ either equal zero or blow up. Hence one expects that $F_{\e}(\w)(u,A)$ should be well-approximated by the functional $I_{\e}(\w):\{v:\e\Lw\to\{\pm e_1\}\}\to [0,+\infty]$ given by the formula
\begin{equation}\label{Ising}
I_{\e}(\w)(v,A)=\sum_{\e x\in \e\Lw\cap A}\e^{d-1}\beta\Big(\sum_{\substack{\e y\in\e\Lw\cap A\\ (x,y)\in\mathcal{E}(\w)}}\frac{1}{2}|v(\e x)-v(\e y)|,\#\{\e y\in\e\Lw\cap A:\, (x,y)\in\mathcal{E}(\w)\}\Big),
\end{equation}
where the function $\beta$ is given by (\ref{limitatinf}) and $\beta(0,k):=0$ for all $1\leq k\leq M$. Indeed, the function $\beta$ takes into account how many components of the discrete gradient blow up while all others remain zero. 

In order to formulate the theorem we recall the $\Gamma$-convergence results obtained for the energies $E_{\e}(\w)$ and $I_{\e}(\w)$ (both extended to $L^p(D)$ via $+\infty$). The following result was proven in \cite[Theorem 3]{ACG2}\footnote{Interactions via a random edge set $\mathcal{E}(\w)$ are not covered by the results in \cite{ACG2}. However, the proof works for general finite range interactions with $p$-growth and coercive Voronoi neighbor interactions. Due to (\ref{nncontained}) those assumptions are fulfilled.}:
\begin{theorem}[\cite{ACG2}]\label{ACGmain}
For every sequence $\e\to 0$ there exists a subsequence $\e_n$ such that for every $A\in\Ard$ the functionals $E_{\e_n}(\w)(\cdot,A)$ 
$\Gamma(L^p(D,\R^m))$-converge to a functional $E(\w):L^p(D,\R^m)\times\Ard\to [0,+\infty]$ that is finite only on $W^{1,p}(A,\R^m)$, where it takes the form
\begin{equation*}
E(\w)(u,A)=
\int_A q(x,\nabla u)\,\mathrm{d}x
\end{equation*}
for some non-negative Carath\'eodory-function $q$ that is quasiconvex in the second variable for a.e. $x\in D$ and satisfies the growth conditions
\begin{equation*}
\frac{1}{C}|\xi|^p-C\leq q(x,\xi)\leq C(|\xi|^p+1).
\end{equation*}
\end{theorem} 

The $\Gamma$-convergence of functionals of the type \eqref{Ising} was treated in \cite[Theorem 3.2]{ACR}\footnote{In \cite{ACR} the proofs are given only for pairwise interactions, but as already mentioned in \cite[Section 6.3]{ACR}, up to minor modifications they cover also multi-body interactions. Hence we decided not to repeat the arguments.}:
\begin{theorem}[\cite{ACR}]\label{t.ACRmain}
For every sequence $\e\to 0$ there exists a subsequence $\e_n$ such that for every $A\in\Ard$ the functionals $I_{\e_n}(\w)(\cdot,A)$ $\Gamma(L^p(D,\R^m))$-converge to a functional $I(\w):L^p(D,\R^m)\times\Ard\to [0,+\infty]$ that is finite only on $BV(A,\{\pm e_1\})$, where it takes the form
\begin{equation*}
I(\w)(u,A)=
\int_A s(x,\nu_u)\,\mathrm{d}\mathcal{H}^{d-1}
\end{equation*}
for some  measurable function $s$ that satisfies the growth conditions
\begin{equation*}
\frac{1}{C}\leq s(x,\nu)\leq C.
\end{equation*}
\end{theorem} 
Our first main result states that the $\Gamma$-limit of $F_{\e}(\w)$ can be characterized by the $\Gamma$-limits of $E_{\e}(\w)$ and $I_{\e}(\w)$ in the following sense:
\begin{theorem}\label{t.sepofscales}
Let $\e_n\to 0$. The $\Gamma(L^1(D))$-limit of $F_{\e_n}(\w)$ exists if and only if the $\Gamma(L^p(D))$-limits of $E_{\e_n}(\w)(\cdot,D)$ and $I_{\e_n}(\w)(\cdot,D)$ exist. In this case the $\Gamma$-limit $F(\w)$ is finite only $GSBV^p(D,\R^m)$, where it is given by
\begin{equation*}
F(\w)(u)=\int_D q(x,\nabla u(x))\dx +\int_{S_u}s(x,\nu_u)\,\mathrm{d}\mathcal{H}^{d-1}
\end{equation*}	
with the integrands given by Theorems \ref{ACGmain} and \ref{t.ACRmain}.
\end{theorem}
\begin{remark}\label{r.blowupformulas}
Both Theorems \ref{ACGmain} and \ref{t.ACRmain} provide asymptotic formulas for the integrands $q$ and $s$, respectively. To state them, we need some notation. Given a set $A\in\Ard$ and $\delta>0$, we set
\begin{equation}\label{discretebdry}
\partial_{\delta}A:=\{x\in\R^d:\;\dist(x,\partial A)\leq \delta\}.
\end{equation} 
For a pointwise well-defined function $\bar{u}\in L^{\infty}_{\rm loc}(\R^d,\R^m)$, we define the set of discrete functions taking the boundary value $\bar{u}$ on $\partial_{\delta}A$ as
\begin{equation}\label{bdryclass}
\mathcal{PC}_{\e,\delta}^{\w}(\bar{u},A)=\{u\in\mathcal{PC}_{\e}^{\w}:\,u(\e x)=\bar{u}(\e x)\text{ for all }\e x\in\e\Lw\cap\partial_{\delta}A\}.
\end{equation}
Then, as shown in Step 1 of the proof of \cite[Theorem 2]{ACG2}, for a.e. $x_0\in D$ and every $\xi\in\mathbb{R}^{m\times d}$ it holds that
\begin{equation*}
q(x_0,\xi)=\lim_{\varrho\to 0}\varrho^{-d}\lim_{n\to+\infty}\left(\inf\left\{E_{\e_n}(\w)(v,Q_{e_1}(x_0,\varrho)):\,v\in\mathcal{PC}_{\e_n,M\e_n}^{\w}(\xi(\cdot-x_0),Q_{e_1}(x_0,\varrho))\right\}\right),
\end{equation*}	
where $M$ denotes the maximal range of interactions given by Definition \ref{defgoodedges}. 
\\
The formula we use for $s$ can be found at the beginning of the proof of \cite[Theorem 5.8]{BCR}. It states that for all $x_0\in D$ and $\nu\in S^1$ we have 
\begin{equation*}
s(x_0,\nu)=\limsup_{\varrho\to 0}\frac{1}{\varrho^{d-1}}\lim_{\delta\to 0}\limsup_{n\to+\infty}\left(\inf\left\{I_{\e_n}(\w)(v,Q_{\nu}(x_0,\varrho)):\,v\in\mathcal{PC}_{\e_n,\delta}^{\w}(u^{-e_1,e_1}_{x_0,\nu},Q_{e_1}(x_0,\varrho))\right\}\right),
\end{equation*}
where $u_{x_0,\nu}^{-e_1,e_1}$ is defined in \eqref{eq:purejump}. Note that the minimization problem defining $s$ is automatically restricted to functions with values in $\{\pm e_1\}$.
\end{remark}

\subsection{Stochastic homogenization}
The second main result uses the statistical properties of the lattice and the edges (more precisely only stationarity and ergodicity) in order to avoid passing to subsequences for the $\Gamma$-convergence result. It is a straightforward consequence of Theorem \ref{t.sepofscales} and the homogenization results proven in \cite{ACG2,ACR}.

\begin{theorem}\label{mainthm1}
Assume that $\mathcal{L}$ is a stationary and ergodic stochastic lattice with admissible stationary edges in the sense of Definitions \ref{defadmissible} and \ref{defgoodedges}. Then $\mathbb{P}$-a.s. the functionals $F_{\e}(\w)$ $\Gamma$-converge in the $L^1(D,\R^m)$-topology to a deterministic functional $F:L^1(D,\R^m)\to[0,+\infty]$ with domain $L^1(D,\R^m)\cap GSBV^{p}(D,\R^m)$, where it is given by
\begin{equation*}
F(u)=
\int_Dh(\nabla u)\dx+\int_{S_u}\varphi(\nu_u)\,\mathrm{d}\mathcal{H}^{d-1}
\end{equation*}
for some convex, $p$-homogeneous function $h$ and some convex, one-homogeneous function $\varphi$.
\end{theorem}
\begin{remark}\phantomsection\label{r.ondensities}
\begin{itemize}
	\item[(i)] By Theorem \ref{t.sepofscales} the functions $h$ and $\varphi$ are given by the formulas in Remark \ref{r.blowupformulas} for $q(0,\cdot)$ and $s(0,\cdot)$, respectively. However, we can take $\varrho=1$ in both formulas and in the formula for $s(0,\cdot)$ one can avoid the additional limit in $\delta$ by setting $\delta=2M\e$. Moreover, every remaining $\limsup$-expression can be replaced by a limit.
	\item[(ii)] The above $\Gamma$-convergence result also holds locally on each $A\in\Ard$ for the same set of $\w$.
	\item[(iii)] Theorem \ref{mainthm1} still holds if we drop the ergodicity assumption. Then the integrands are $\tau$-invariant, but possibly random. 
\end{itemize}
\end{remark}

\subsection{Convergence of minimizers in the stationary, ergodic setting}
Now we add a discrete fidelity term to our approximating functional $F_{\e}(\w)$, which will approximate the continuum fidelity term that penalizes the distance to the measured image. 

In order to define the discrete approximation, we consider a discrete measurement of a given continuum function. More precisely, throughout this section we fix an exponent $q>1$ and consider a sequence $g_{\e}(\w):\e\Lw\to\R^m$, for which we assume that there exists $g\in L^q(D,\R^m)$ such that $\mathbb{P}$-a.s.
\begin{equation}\label{approxassumption}
g_{\e}(\w)\to g \quad\text{in }L^q(D,\R^n).
\end{equation}
\begin{remark}\label{examplesequence}
For every given $g\in L^q(D,\R^m)$, we find a sequence with this approximation property by first extending $g=0$ on $\R^d\backslash D$ and then setting
\begin{equation*}
g_{\e}(\w)(\e x)=\frac{1}{|B_{\e}(\e x)|}\int_{ B_{\e}(\e x)}g(z)\,\mathrm{d}z.
\end{equation*}	
To see this, we may assume that $m=1$. It is a consequence of the (generalized) Lebesgue differentiation theorem (see \cite[Remark 1.160]{FoLe}) that $g_{\e}(\w)\to g$ a.e. on $D$. Since $g_{\e}$ is bounded in modulus by the maximal function of $g$ (which belongs itself to $L^q(\R^d)$), we deduce (\ref{approxassumption}) from dominated convergence.
\end{remark}

Given a sequence $g_{\e}(\w)$ satisfying \eqref{approxassumption}, we introduce $F_{\e,g}(\w):L^1(D,\R^m)\to [0,+\infty]$ defined by
\begin{equation}\label{eq:defF_e,g}
F_{\e,g}(\w)(u)=
\begin{cases}
\displaystyle F_{\e}(\w)(u)+\sum_{\e x\in\e\Lw\cap D}\e^d|u(\e x)-g_{\e}(\w)(\e x)|^q &\mbox{if $u\in\mathcal{PC}_{\e}^{\w}$,}\\
+\infty &\mbox{otherwise,}
\end{cases}
\end{equation}
where $F_{\e}(\w)$ is defined in (\ref{deffunctional}). Note that we chose a discrete fidelity term not depending on the measure of the Voronoi cells. The identification of the $\Gamma$-limit of $F_{\e,g}(\w)$ is contained in the following theorem.
\begin{theorem}\label{convfull}
Let $q\in(1,+\infty)$ and $g_{\e}(\w)$ satisfy (\ref{approxassumption}). Under the assumptions of Theorem \ref{mainthm1}, there exists a constant $\gamma>0$ such that $\mathbb{P}$-a.s. the functionals $F_{\e,g}(\w)$ defined in \eqref{eq:defF_e,g} $\Gamma$-converge with respect to the $L^1(D,\R^m)$-topology to the functional $F_g:L^1(D,\R^m)\to [0,+\infty]$ with domain $L^q(D,\R^m)\cap GSBV^p(D,\R^m)$, where it is defined by
\begin{equation*}
F_{g}(u)=\int_D h(\nabla u)\,\mathrm{d}x+\int_{S_u}\varphi(\nu_u)\,\mathrm{d}\mathcal{H}^{d-1}+\gamma\int_D|u-g|^q\,\mathrm{d}x
\end{equation*}
with the functions $h$ and $\varphi$ given by Theorem \ref{mainthm1}.
\\
Moreover, for each $\e>0$ there exists a global minimizer $\hat{u}_{\e}\in\mathcal{PC}_{\e}^{\w}$ of the functional $F_{\e,g}(\w)$ and if $u_{\e}\in\mathcal{PC}_{\e}^{\w}$ is any sequence such that
\begin{equation*}
\lim_{\e\to 0}\Big(F_{\e,g}(\w)(u_{\e})- \min_{u\in\mathcal{PC}_{\e}^{\w}}F_{\e,g}(\w)(u)\Big)=0,
\end{equation*}
then a.s. it is compact in $L^q(D,\R^m)$ and each cluster point as $\e\to 0$ is a global minimizer of $F_g$.
\end{theorem}
\begin{remark}
Up to a further subsequence, the statement of Theorem \ref{convfull} remains valid in the non-stationary setting of Theorem \ref{t.sepofscales}. However, the $\Gamma$-limit contains a heterogeneous fidelity term of the form $\int_D\gamma(x)|u-g|^q\dx$ for some positive function $\gamma\in L^{\infty}(D)$ such that also $1/\gamma\in L^{\infty}(D)$. In the case of a stationary, but non-ergodic group action the function $\gamma$ does not depend on $x$, but might be random.
\end{remark}
\section{Applications and the main result}\label{s.applications}
We postpone the technical proofs of the results of Section \ref{s.presentation} to the last two sections. First we use them to derive the $\Gamma$-limit of the functionals in \eqref{fd:intro} and prove our main result, that is, the approximation of the Mumford-Shah functional announced in the introduction.

\subsection{The \texorpdfstring{$\Gamma$}{}-limit for forward differences on \texorpdfstring{$\mathbb{Z}^2$}{}}
Our aim is to analyze the asymptotic behavior of the discrete functional proposed in \cite{StCr}. It is based on the square lattice $\mathcal{L}=\mathbb{Z}^2$ with edges that yield standard forward differences, that is, $\mathcal{E}_{FD}=\{(x,x+e_i):\,x\in\mathbb{Z}^2,\,i=1,2\}$, where $\{e_i\}_{i=1,2}$ denotes the standard basis of $\R^2$. The discrete approximation is then defined for functions $u:\e\mathbb{Z}^2\to\R^m$ and after rescaling and dropping the fidelity term it reads as
\begin{equation}\label{defforwarddiff}
\tilde{F}_{\e}(u,A)=\sum_{\e x\in\e \mathbb{Z}^2\cap A}\e\min\{\alpha\e^{-1} (|u(\e x+\e e_1)-u(\e x)|^2+|u(\e x+\e e_2)-u(\e x)|^2),1\}.
\end{equation}
The set of edges satisfies (\ref{nncontained}) as the Voronoi neighbors are given by all $x,y\in\mathbb{Z}^2$ such that $|x-y|=1$. Moreover $\tilde{F}_{\e}$ has the required structure to apply our results which can be seen by setting $f(\mathfrak{p})=\min\{\alpha\|\mathfrak{p}\|_1,1\}$. In \cite{StCr} Cremers and Strekalovskiy conjectured that $\tilde{F}_{\e}$ approximates the Mumford-Shah functional. With the results of Section \ref{s.presentation} we can identify the $\Gamma$-limit of $\tilde{F}_{\e}$, which differs from the Mumford-Shah functional due to an anisotropic surface integral.
\begin{corollary}\label{notMS}
The functionals $\tilde{F}_{\e}$ defined in (\ref{defforwarddiff}) $\Gamma$-converge with respect to the $L^1(D,\R^m)$-topology to the functional $F_0:L^1(D,\R^m)$ with domain $L^1(D,\R^m)\cap GSBV^2(D,\R^m)$, where it is given by
\begin{equation*}
F_0(u)=\alpha\int_D|\nabla u|^2+\int_{S_u}\varphi_0(\nu)\,\mathrm{d}\mathcal{H}^1,
\end{equation*} 
with the function $\varphi_0:\R^2\to [0,+\infty)$ defined by
\begin{equation*}
\varphi_0(\nu)=\begin{cases}
|\nu_1|+|\nu_2| &\mbox{if $\nu_1\cdot\nu_2< 0$,}\\
\max\{|\nu_1|,|\nu_2|\} &\mbox{if $\nu_1\cdot\nu_2\geq 0$.}
\end{cases}
\end{equation*}
\end{corollary}
\begin{proof}
As outlined above, we can apply Theorem \ref{mainthm1} to the sequence $\tilde{F}_{\e}$. Moreover, from Theorem \ref{t.sepofscales} we deduce that the function $h$ in Theorem \ref{mainthm1} is given by the density of the $\Gamma$-limit of the sequence $E_{\e}$ defined in (\ref{auxelastic}). With our choice of $f$ and $\mathcal{E}_{FD}$ the functional $E_{\e}$ takes the form
\begin{equation*}
E_{\e}(u)=\alpha\sum_{\substack{\e x,\e y\in \e\mathbb{Z}^2\cap D\\ |x-y|=1}}\frac{\e^2}{2}\Big|\frac{u(\e x)-u(\e y)}{\e}\Big|^2.
\end{equation*}
In this case \cite[Remark 5.3]{AC} yields that
\begin{equation*}
\Gamma(L^p(D))\hbox{-}\lim_{\e\to 0}E_{\e}(u)=\alpha\int_D|\nabla u|^2\,\mathrm{d}x
\end{equation*}
for all $u\in W^{1,2}(D,\R^m)$, so that $h(\xi) = \alpha|\xi|^2$ and it remains to identify the surface integrand $\varphi$ in Theorem \ref{mainthm1} as $\varphi_0$. By convexity, the function $\varphi$ is continuous, so it suffices to treat the case $\nu_1\cdot\nu_2\neq 0$. Note that for forward differences the functional $I_{\e}$ defined in (\ref{Ising}) is given by
\begin{equation*}
I_{\e}(u,A) = \frac{1}{2}\sum_{\e x\in\e\mathbb{Z}^2\cap A}\e \max\{|u(\e x+\e e_i)-u(\e x)|:\,i\in\{1,2\}\text{ with }\e x+\e e_i\in A\},
\end{equation*}
where $u:\e\mathbb{Z}^2\to\{\pm e_1\}$. From Remarks \ref{r.blowupformulas} and \ref{r.ondensities} we know that $\varphi$ is given by the formula
\begin{equation}\label{formulag0}
\varphi(\nu)=\lim_{\e\to 0}\left(\inf\left\{I_{\e}(v,Q_{\nu}(0,1)):\;v\in\mathcal{PC}_{\e,4\e}(u^{-e_1,e_1}_{0,\nu},Q_{\nu}(0,1))\right\}\right).
\end{equation} 
Consider for fixed $\e<<1$ any function $v\in \mathcal{PC}_{\e,4\e}(u^{-e_1,e_1}_{0,\nu},Q_{\nu}(0,1))$ with values in $\{\pm e_1\}$. We locally construct a function $\tilde{v}\in BV(Q_{\nu}(0,1),\{\pm e_1\})$ as follows: on $Q_{e_1}(\e x,\e)$ with $x\in\mathbb{Z}^2$ such that $Q_{e_1}(\e x,\e)\subset Q_{\nu}(0,t)$ we set
\begin{equation*}
\tilde{v}(y)=
\begin{cases}
v(\e x) &\mbox{if $\prod_{i=1}^2|v(\e x+\e e_i)-v(\e x)|=0$,}
\\
v(\e x) &\mbox{if $\prod_{i=1}^2|v(\e x+\e e_i)-v(\e x)|\neq 0$ and $\langle y-\e x,e_1+e_2\rangle\leq 0$,}
\\
v(\e x+\e e_1) &\mbox{if $\prod_{i=1}^2|v(\e x+\e e_i)-v(\e x)|\neq 0$ and $\langle y-\e x,e_1+e_2\rangle> 0$,}
\end{cases}
\end{equation*}
while we define $\tilde{v}=u_{0,\nu}^{-e_1,e_1}$ on all cubes $Q_{e_1}(\e x,\e)$ with $Q_{e_1}(\e x,\e)\cap \partial Q_{\nu}(0,1)\neq\emptyset$. The latter implies $\tilde{v}=u_{0,\nu}^{-e_1,e_1}$ on $\partial Q_{\nu}(0,1)$ in the sense of traces. Moreover we modified the jump set away from the boundary in such a way that it contains diagonal lines of length $\sqrt{2}$ instead of corners formed by the upper and the right hand side of a cube. Setting $Q_{\e}:=\{y\in Q_{\nu}(0,1):\,{\rm dist}(y,\partial Q_{\nu}(0,1))>2\e\}$, the above construction and the boundary conditions for $v$ imply
\begin{equation}\label{lbg0}
I_{\e}(v,Q_{\nu}(0,1))\geq \int_{S_{\tilde{v}}\cap Q_{\e}}\varphi_0(\nu_{\tilde{v}})\,\mathrm{d}\mathcal{H}^1\geq \int_{S_{\tilde{v}}\cap Q_{\nu}(0,1)}\varphi_0(\nu_{\tilde{v}})\,\mathrm{d}\mathcal{H}^1-C\e.
\end{equation}
Observe that $\varphi_0$ is convex. Hence the functional on the right hand side is $BV$-elliptic in the sense that
\begin{equation*}
\int_{S_{\tilde{v}}\cap Q_{\nu}(0,1)}\varphi_0(\nu_{\tilde{v}})\,\mathrm{d}\mathcal{H}^1\geq \varphi_0(\nu)
\end{equation*}
for all $\tilde{v}\in BV(Q_{\nu}(0,1),\{\pm e_1\})$ such that $\tilde{v}=u_{0,\nu}^{-e_1,e_1}$ on $\partial Q_{\nu}(0,1)$ in the sense of traces (see \cite{AmBrII} for more details). Since $v\in\mathcal{PC}_{\e,4\e}(u^{-e_1,e_1}_{0,\nu},Q_{\nu}(0,1))$ was arbitrary, we conclude from (\ref{formulag0}) and (\ref{lbg0}) that $\varphi(\nu)\geq \varphi_0(\nu)$.
	
In order to prove the reverse inequality, first note that
\begin{equation}\label{minsumbound}
I_{\e}(u,A)\leq\frac{1}{4}\sum_{\substack{\e x,\e y\in\e\mathbb{Z}^2\cap A\\ |x-y|=1}} \e|u(\e x)-u(\e y)|.
\end{equation}
The $\Gamma$-limit of the right hand side of (\ref{minsumbound}) is well-known. It is finite only on $BV(D,\{\pm e_1\})$ and given by $\int_{S_u}|\nu|_1\,\mathrm{d}\mathcal{H}^1$ (see \cite{ABC}). By comparison we obtain $\varphi(\nu)\leq |\nu_1|+|\nu_2|$. This finishes the proof in the case $\nu_1\cdot\nu_2<0$. If $\nu_1\cdot\nu_2>0$, denote by $i_0$ the index such that $|\nu_{i_0}|=\max\{|\nu_1|,\nu_2|\}$ and set $i_1=\{1,2\}\backslash\{i_0\}$. We define a candidate for the minimum problem in (\ref{formulag0}) setting $u_{\nu}(\e x)=u^{-e_1,e_1}_{0,\nu}(\e x)$ for all $x\in\mathbb{Z}^2$. By definition it satisfies the correct boundary conditions. A straightforward analysis shows that for any $x\in\mathbb{Z}^2$ with $u_{\nu}(\e x)\neq u_{\nu}(\e x+\e e_{i_1})$ we have $u_{\nu}(\e x)\neq u_{\nu}(\e x+\e e_{i_0})$. Thus it suffices to count just the interactions along the direction $e_{i_0}$. Those can be bounded by $\e^{-1}|\nu_{i_0}|+C$, so that 
\begin{equation*}
I_{\e}(u_{\nu},Q_{\nu}(0,1))\leq |\nu_{i_0}|+C\e=\max\{|\nu_1|,|\nu_2|\}+C\e.
\end{equation*}
From (\ref{formulag0}) we conclude that $\varphi(\nu)\leq \varphi_0(\nu)$ which finishes the proof.
\end{proof}
\begin{remark}\label{higherdim}
For the $d$-dimensional version of $\tilde{F}_{\e}$ (defined in the introduction), one still has the existence of the $\Gamma$-limit with an anisotropic surface integrand. To see the latter, one first shows that $\varphi(e_1)=1$. Then for the vector $\nu_0=(\frac{1}{\sqrt{2}},\frac{1}{\sqrt{2}},0)$ the discretization of $u_{0,\nu_0}^{-e_1,e_1}$ yields an upper bound $\varphi(\nu_0)\leq \frac{1}{\sqrt{2}}$ (actually equality holds). The precise $\Gamma$-limit in higher dimensions is beyond the scope of this paper.
\end{remark}
\subsection{Approximations of the Mumford-Shah functional}\label{subsec:MS}
Now we use our general results to provide a discretization of the vector-valued Mumford-Shah functional. To this end, we need to take our parameters $p=q=2$. However, it might be of interest to obtain also other exponents for the fidelity term and therefore we consider the general case $q>1$ and just fix $p=2$ to focus on the isotropy issue. We suggest to take as stochastic lattice the so-called {\it random parking process}. For the precise geometric construction of this process by suitably choosing projected points of a homogeneous Poisson point process in dimension $d+1$, we refer the reader to the two papers \cite{Pe,glpe}. Here we recall that the random parking process defines a stochastic lattice $\mathcal{L}_{RP}$ that is admissible, stationary, ergodic and, most important for our applications, isotropic in the sense of Definition \ref{defstatiolattice}. Moreover, we can choose for instance $\mathcal{E}(\w)=\mathcal{N}(\w)$ to obtain stationary and isotropic edges (see also Remark \ref{r.measurability} for other possible choices). We prove our result for general stochastic lattices satisfying these assumptions. Note that the following theorem covers in particular the two functionals presented in the introduction.
\begin{theorem}\label{MSapprox}
	Fix $p=2$ and let $q\in(1,+\infty)$ and $g_{\e}(\w)$ satisfy (\ref{approxassumption}). Assume that $\mathcal{L}$ is an admissible stochastic lattice that is stationary, ergodic and isotropic with admissible stationary and isotropic edges. Then there exist constants $c_1,c_2,c_3>0$ such that $\mathbb{P}$-a.s. the functionals $F_{\e,g}(\w)$ defined in \eqref{eq:defF_e,g} $\Gamma$-converge with respect to the $L^1(D,\R^m)$-topology to the functional $F_g:L^1(D,\R^m)\to [0,+\infty]$ with domain $L^q(D,\R^m)\cap GSBV^2(D,\R^m)$, where it is defined by
	\begin{equation*}
	F_{g}(u)=c_1\int_D |\nabla u|^2\,\mathrm{d}x+c_2\mathcal{H}^{d-1}(S_u)+c_3\int_D|u-g|^q\,\mathrm{d}x.
	\end{equation*}
\end{theorem}
\begin{remark}\label{generalp}
In the scalar case $m=1$, the statement of Theorem 	\ref{MSapprox} is valid for every $p>1$. Indeed, we already know that the function $h$ has to be $p$- homogeneous. Following the proof below, stochastic isotropy implies that it is constant on $S^{d-1}$. Hence $h(\xi)=c_1|\xi|^p$ for some $c_1>0$. The formula for the surface term does not depend on $p$.
\end{remark}
\begin{remark}\label{r.complete}
Theorem \ref{MSapprox} and the convergence of minimizers in Theorem \ref{convfull} yield the full discretization of the Mumford-Shah functional. In practice it is of course impossible to create the stochastic lattice on the whole space but one has to take a finite particle approximation. Moreover the minimization of the discrete functionals $F_{\e}(\w)$ is still nontrivial due to non-convexity. However, first numerical tests have shown promising results and we plan to further investigate our approach in the future.	
\end{remark}
\begin{proof}[Proof of Theorem \ref{MSapprox}]
Due Theorem \ref{convfull} it only remains to show that $h(\xi)=c_1|\xi|^2$ and $\varphi(\nu)=c_2$ for some constants $c_1,c_2>0$. By Theorem \ref{t.sepofscales} the function $h$ is also the density of the $\Gamma$-limit of the functionals $E_{\e}(\w)$ defined in (\ref{auxelastic}). For $p=2$ these are non-negative quadratic forms, so we deduce from \cite[Theorem 11.1]{DM} that $h$ is a non-negative quadratic form, too. We write it explicitly as 
	\begin{equation*}
	h(\xi)=\sum_{i,k=1}^{m}\sum_{j,l=1}^{d}h_{ijkl}\xi_{ij}\xi_{kl},
	\end{equation*}
	where the coefficients satisfy the symmetry condition $h_{ijkl}=h_{klij}$. Since the discrete functional is invariant under orthogonal transformations $u\mapsto Qu$ it holds that $h(Q\xi)=h(\xi)$ for all $\xi\in\mathbb{R}^{m\times d}$ and $Q\in O(m)$. Moreover, reasoning exactly as for the case $m=d$ treated in \cite[Theorem 9]{ACG2} one can further show that ergodicity and isotropy imply $h(\xi R)=h(\xi)$ for all $\xi\in\mathbb{R}^{m\times d}$ and all $R\in SO(d)$. We argue that $h$ depends only on the singular values. To this end, we fix a matrix $\xi\in\mathbb{R}^{m\times d}$ and consider any singular value decomposition $\xi=Q\Sigma V^T$ with orthogonal matrices $Q\in O(m)$ and $V\in O(d)$ and a diagonal matrix $\Sigma\in\mathbb{R}^{m\times d}$. If $V\in SO(d)$ then $h(\xi)=h(\Sigma)$. Otherwise we replace $V$ by a rotation observing that
	\begin{equation*}
	A=QP_m^{1,2}P_m^{1,2}\Sigma P_d^{1,2} P_d^{1,2}V^T,
	\end{equation*}
	where $P_n^{1,2}$ denotes the $n\times n$-matrix which differs from the identity by exchanging the first and the second column. In this case $h(\xi)=h(P_m^{1,2}\Sigma P_d^{1,2})$ since the matrix $P_d^{1,2}V^T$ belongs to $SO(d)$. Set $l=\min\{d,m\}$ and write the singular values as $\lambda(\xi)\in\mathbb{R}^l$ with non-negative coefficients. We conclude that there exists a permutation $P(\xi)\in\{I,P_l^{1,2}\}$ such that  
	\begin{equation*}
	h(\xi)=\sum_{i,k=1}^lh_{iikk}\big(P(\xi)\lambda(\xi)\big)_i\big(P(\xi)\lambda(\xi)\big)_k.
	\end{equation*}
	Thus it is enough to characterize the coefficients $h_{iikk}$. We will test several diagonal matrices $\xi$. To simplify notation, given $v\in\mathbb{R}^l$ we denote by ${\rm diag}(v)\in\mathbb{R}^{m\times d}$ the diagonal matrix with elements $v$. As a first step, note that we can find $Q\in O(m)$ and $R\in SO(d)$ such that ${\rm diag}(e_i)=Q\,{\rm diag}(-e_j)R$. Thus by invariance $h_{iiii}=h_{jjjj}$ for all $i,j=1,\dots,l$. Now consider $i\neq k$. We argue that $h_{iikk}=0$. To this end, we use that there exists a matrix $Q\in O(m)$ such that ${\rm diag}(e_i+e_k)=Q\,{\rm diag}(e_i-e_k)$, which yields again by invariance that
	\begin{equation*}
	h_{iiii}+h_{kkkk}+h_{iikk}+h_{kkii}=h_{iiii}+h_{kkkk}-h_{iikk}-h_{kkii}.
	\end{equation*}
	Due the symmetry condition on the coefficients of $h$ we obtain $h_{iikk}=0$. Setting $c_1=h_{1111}$ we have shown that
	\begin{equation*}
	h(\xi)=c_1\sum_{i=1}^l\lambda_i(\xi)^2=c_1|\xi|^2.
	\end{equation*}
	It follows from Theorem \ref{ACGmain} that $c_1>0$.
	
	We now turn to the surface integrand $\varphi$ and prove that $\varphi(R\nu)=\varphi(\nu)$ for all $R\in SO(d)$. Since $\varphi$ is deterministic by ergodicity, we can take expectations in the asymptotic formula given by Remark \ref{r.blowupformulas} and simplified via Remark \ref{r.ondensities}. Since $\tau^{\prime}_{R}$ is measure preserving, dominated convergence and a change of variables yield 
	\begin{align*}
	\varphi(R\nu)&=\lim_{\e\to 0}\int_{\Omega}\inf\{I_{\e}(\w)(v,Q_{R\nu}(0,1)):\,v\in\mathcal{PC}^{\w}_{\e,2M\e}(u_{0,R\nu}^{-e_1,e_1},Q_{R\nu}(0,1))\}\,\mathrm{d}\mathbb{P}(\w)
	\\
	&=\lim_{\e\to 0}\int_{\Omega}\inf\{I_{\e}(\tau^{\prime}_{R^T}\w)(v,Q_{\nu}(0,1)):\,v\in\mathcal{PC}^{\tau^{\prime}_{R^T}\w}_{\e,2M\e}(u_{0,\nu}^{-e_1,e_1},Q_{\nu}(0,1))\}\,\mathrm{d}\mathbb{P}(\w)
	\\
	&=\lim_{\e\to 0}\int_{\Omega}\inf\{I_{\e}(\w^{\prime})(v,Q_{\nu}(0,1)):\,v\in\mathcal{PC}^{\w^{\prime}}_{\e,2M\e}(u_{0,\nu}^{-e_1,e_1},Q_{\nu}(0,1))\}\,\mathrm{d}\mathbb{P}(\w^{\prime})=\varphi(\nu),
	\end{align*}
	where we used from the first to the second line that by isotropy of $\mathcal{L}$ and $\mathcal{E}$ the discrete functional in \eqref{Ising} satisfies $I_{\e}(\tau^{\prime}_R\w)(u,Q_{R\nu}(0,1))=I_{\e}(\w)(u\circ R,Q_{\nu}(0,1))$ for every $R\in SO(d)$. We finish the proof setting $c_2=\varphi(e_1)$ since Theorem \ref{t.ACRmain} implies $c_2>0$.
\end{proof}

\section{Separation of scales: proof of Theorem \ref{t.sepofscales}}\label{s.proofs}
This section is devoted to the proof of Theorem \ref{t.sepofscales}, which will constitute the most involved part of the paper. In a first part we use \cite[Theorem 1]{BFLM} to represent (up to subsequences) the $\Gamma$-limit on $SBV^p$ as an integral functional. In a second and third part we study asymptotic formulas for the integrands of that representation which allow to conclude the proof. Several times we will need the following property of Voronoi neighbors in a stochastic lattice.

\begin{lemma}\label{l.paths}
Let $\Lw$ be an admissible set of points with constants $R>r>0$ as in Definition \ref{defadmissible}. Then for all $x,y\in\Lw$ there exists a path $P(x,y)=\{x_0=x,x_1,\dots,x_n=y\}$ such that $(x_{i-1},x_{i})\in\mathcal{N}(\w)$ for all $1\leq i \leq n$ and
\begin{equation*}
P(x,y)\subset {\rm co}(x,y)+B_{2R}(0).
\end{equation*}
In particular, there exists a constant $C_{r,R}<+\infty$ such that $\#P(x,y)\leq C_{r,R}|x-y|$.
\end{lemma}
\begin{proof}
For $0<\delta\ll 1$ consider the collection of segments
\begin{equation*}
G_{\delta}(x,y)=\{z+\lambda (y-x):\;z\in B_{\delta}(x),\,0\leq\lambda\leq 1\}.
\end{equation*}
We argue that there exists a segment $s^*=\{z_0+\lambda(y-x):0\leq\lambda\leq 1\}\subset G_{\delta}(x,y)$ that does not intersect any Voronoi facet of the tessellation $\mathcal{V}(\w)$ of dimension less than $d-1$. Indeed, assume by contradiction that the claim is false for all $z_0\in B_{\delta}(x)$. By Remark \ref{voronoi} we find finitely many Voronoi facets of dimension less than $d-1$ whose projections onto the hyperplane containing $x$ and orthogonal to $y-x$ cover a $d-1$-dimensional set. Since projections onto hyperplanes are Lipschitz continuous, we obtain a contradiction. Hence, for $\delta$ small enough, we can construct a path of Voronoi neighbors connecting $x,y$ by suitably numbering the set
\begin{equation*}
P(x,y)=\{z\in \Lw:\;\mathcal{C}(z)\cap s_*\neq\emptyset\},
\end{equation*} 
which satisfies $P(x,y)\subset {\rm co}(x,y)+{B_{R+\delta}}(0)$ as claimed. Combining Remark \ref{voronoi} with a covering argument we obtain in addition 
\begin{equation*}
\#P(x,y)\leq |B_{\frac{r}{2}}(0)|^{-1}(|x-y|+2R)(2R)^{d-1}\leq |B_{\frac{r}{2}}(0)|^{-1}\frac{2(2R)^d}{r}|x-y|=:C_{r,R}|x-y|.
\end{equation*}
\end{proof}
\subsection{Integral representation on SBV\texorpdfstring{$^{\bf p}$}{}}
We are going to prove the following intermediate result: 
\begin{proposition}\label{limitsbvp}
Given any sequence $\e\to 0$ there exists a subsequence $\e_n$ such that for all $A\in\Ard$ the functionals $F_{\e_n}(\w)(\cdot,A)$ $\Gamma$-converge in the $L^1(D,\R^m)$-topology to a functional $F(\w)(\cdot,A):L^1(D,\R^m)\to[0,+\infty]$. If $u\in SBV^p(A,\R^m)$ then $F(\w)(u,A)$ can be written as
\begin{equation*}
F(\w)(u,A)=
\int_Ah(x,\nabla u)\dx+\int_{S_u\cap A}\varphi(x,{u^+}-{u^-},\nu)\,\mathrm{d}\mathcal{H}^{d-1}.
\end{equation*}
where, for $x_0\in D$, $\nu\in S^1$, $a\in\R^m$ and $\xi\in\R^{d\times m}$, the integrands are given by
\begin{equation*}
\begin{split}
h(x_0,\xi)&=\limsup_{\varrho\to 0}\frac{1}{\varrho_d}m(\w)(\xi(\cdot-x_0),Q_{\nu}(x_0,\varrho)),
\\
\varphi(x_0,a,\nu)&=\limsup_{\varrho\to 0}\frac{1}{\varrho^{d-1}}m(\w)(u_{x_0,\nu}^{a,0},Q_{\nu}(x_0,\varrho))
\end{split}
\end{equation*} 
with the function $u_{x_0,\nu}^{a,0}$ defined in \eqref{eq:purejump} and the function $m(\w)(\bar{u},A)$ defined for any $\bar{u}\in SBV^p(D,\R^m)$ and $A\in\Ard$ by
\begin{equation*}
m(\w)(\bar{u},A)=\inf\{F(\w)(v,A):\;v\in SBV^p(A,\R^m),\,v=\bar{u}\text{ in a neighborhood of }\partial A\}.
\end{equation*}
\end{proposition}
In order to prove this result we will analyze the localized $\Gamma\hbox{-}\liminf$ and $\Gamma\hbox{-}\limsup$ $F^{\prime}(\w),F^{\prime\prime}(\w):L^{1}(D,\R^m)\times\mathcal{A}(D)\to [0,+\infty]$ of the functionals $F_{\e}(\w)$, which are defined by
\begin{equation*}%\label{eq:Gammalimsupinf}
\begin{split}
F^{\prime}(\w)(u,A)&=\inf\{\liminf_{\e\to 0}F_{\e}(\w)(u_{\e},A):\;u_{\e}\to u \text{ in }L^{1}(D,\R^m)\},\\
F^{\prime\prime}(\w)(u,A)&=\inf\{\limsup_{\e\to 0}F_{\e}(\w)(u_{\e},A):\;u_{\e}\to u\text{ in }L^{1}(D,\R^m)\}.
\end{split}
\end{equation*}

\begin{remark}\label{proxi}
It is well-known that both functionals are $L^1(D,\R^m)$-lower semicontinuous. Moreover, note that for any $u\in L^{1}(D,\R^m)$ there exists indeed a sequence $u_{\e}\in\mathcal{PC}_{\e}^{\w}$ such that $u_{\e}\to u$ in $L^{1}(D,\R^m)$. For $u\in C_c(D,\R^m)$ this follows from Remark \ref{voronoi}. In the general case one can use a density argument and construct suitable diagonal sequences.
\end{remark} 
Our aim is to apply the integral representation of \cite[Theorem 1]{BFLM} to a slightly modified functional. To this end, below we establish several properties of $F^{\prime}(\w)$ and $F^{\prime\prime}(\w)$ which are necessary in the context of integral representation. However, at first we prove a truncation lemma that allows to reduce many arguments to the case of bounded functions.
\begin{lemma}\label{trunc}
Let $u_{\e}\in\mathcal{PC}_{\e}^{\w}$. For any $k>0$ set $T_ku_{\e}$ as in Lemma \ref{truncation}. Then, for any $A\in\mathcal{A}(D)$, it holds that $F_{\e}(\w)(T_ku_{\e},A)\leq F_{\e}(\w)(u_{\e},A)$. In particular, whenever $u\in L^{\infty}(D,\R^m)$ we can compute $F^{\prime}(\w)(u,A)$ and $F^{\prime\prime}(\w)(u,A)$ considering sequences $u_{\e}\in\mathcal{PC}_{\e}^{\w}$ such that $|u_{\e}(\e x)|\leq 3\|u\|_{\infty}$ for all $x\in\Lw$. Moreover, for all $u\in L^1(D,\R^m)$ we have
\begin{equation*}
\begin{split}
\lim_{k\to +\infty}F^{\prime}(\w)(T_ku,A)&=F^{\prime}(\w)(u,A),\\
\lim_{k\to +\infty}F^{\prime\prime}(\w)(T_ku,A)&=F^{\prime\prime}(\w)(u,A).
\end{split}
\end{equation*}
\end{lemma}
\begin{proof}
For the estimate at the discrete level, it suffices to combine the fact that $|T_ku_{\e}(\e x)-T_ku_{\e}(\e y)|\leq |u_{\e}(\e x)-u_{\e}(\e y)|$ for all $x,y\in\Lw$ with the monotonicity assumption (\ref{monotone}). In order to restrict the class of approximating sequences, we use the first estimate and the fact that any truncated sequence $T_ku_{\e}$ with $k=\|u\|_{\infty}$ still converges to $u$ in $L^1(D,\R^m)$. The continuity property at the limit follows from $L^1(D,\R^m)$-lower semicontinuity of both functionals and the fact that the discrete upper bound is conserved in the limit.
\end{proof}
We next show that $F^{\prime\prime}(\w)$ is local.
\begin{lemma}[Locality of $F^{\prime\prime}$]\label{local}
Let $A\in\Ard$. If $u,v\in L^1(D,\R^m)$ and $u=v$ a.e. on $A$, then $F^{\prime\prime}(\w)(u,A)=F^{\prime\prime}(\w)(v,A)$.
\end{lemma}
\begin{proof}
Due to Remark \ref{proxi} there exist sequences $u_{\e},v_{\e}\in\mathcal{PC}_{\e}^{\w}$ converging to $u$ and $v$ in $L^1(D,\R^m)$, respectively, and such that
\begin{equation*}
\limsup_{\e\to 0}F_{\e}(\w)(u_{\e},A)=F^{\prime\prime}(\w)(u,A),\quad\quad\limsup_{\e\to 0}F_{\e}(\w)(v_{\e},A)=F^{\prime\prime}(\w)(v,A).
\end{equation*}
Define $\tilde{u}_{\e}\in\mathcal{PC}_{\e}^{\w}$ by
\begin{equation*}
\tilde{u}_{\e}(\e x)=\mathds{1}_{A}(\e x)u_{\e}(\e x)+(1-\mathds{1}_A(\e x))v_{\e}(\e x).
\end{equation*}
Since $|\partial A|=0$ and $u_{\e}$ and $v_{\e}$ are equiintegrable, it follows that $\tilde{u}_{\e}\to v$ in $L^1(D,\R^m)$. Then by definition
\begin{equation*}
F^{\prime\prime}(\w)(v,A)\leq\limsup_{\e\to 0}F_{\e}(\w)(\tilde{u}_{\e},A)=\limsup_{\e\to 0}F_{\e}(\w)(u_{\e},A)=F^{\prime\prime}(\w)(u,A).
\end{equation*}
Exchanging the roles of $u$ and $v$ concludes the proof.
\end{proof}
The following two lemmata provide a lower bound for $F^{\prime}$ and an upper bound for $F^{\prime\prime}$. Together with the lower bound we obtain an equicoercivity property under an additional equiintegrability assumption.
\begin{lemma}[Compactness and lower bound]\label{compact}
Assume that $A\in\mathcal{A}^R(D)$ and $u_{\e}\in\mathcal{PC}_{\e}^{\w}$ are such that 
\begin{equation*}
\sup_{\e>0}F_{\e}(\w)(u_{\e},A)<+\infty.
\end{equation*}
If $u_{\e}$ is equiintegrable on $A$, then there exists a subsequence (not relabeled) such that $u_{\e}\to u$ in $L^1(A,\R^m)$ for some $u\in L^1(A,\R^m)\cap GSBV^p(A,\R^m)$. Moreover we have the estimate
\begin{equation*}
\frac{1}{c}\left(\int_A|\nabla u|^p\dx+\mathcal{H}^{d-1}(S_u\cap A)\right)\leq F^{\prime}(\w)(u,A)
\end{equation*}
for some constant $c>0$ independent of $\w,A$ and $u$.
\end{lemma}

\begin{proof}
We first construct a suitable function $v_{\e}\in SBV^p(A,\R^m)$ that is asymptotically close to $u_{\e}$. Given a triangulation $\mathcal{T}^{d}$ of the cube $[0,1]^d$ we construct a periodic triangulation of $\R^d$ via
\begin{equation*}
\mathcal{T}=\{T=z+T_d:\;z\in\mathbb{Z}^d,\,T_d\in\mathcal{T}^d\}.
\end{equation*}
We may assume that ${\rm diam}(T)< R$ for all $T\in\mathcal{T}$. Define $u^{\rm aff}_{\e}$ as a continuous piecewise affine interpolation of $u_{\e}$ on $\e\mathcal{T}$ as follows: for each $z\in\mathbb{Z}^d$ we choose a point $x(z)\in\Lw$ such that $z\in\mathcal{C}(x)$ and set 
\begin{equation*}
u^{\rm aff}_{\e}(\e z)=u_{\e}(\e x(z)).
\end{equation*}
Next we decompose the scaled lattice as $\e\mathcal{L}(\w)=L_{0,\e}\cup L_{1,\e}$, where
\begin{equation*}
L_{0,\e}=\{\e x\in\e\Lw:\;\|\e|\nabla_{\w,\e}(u,A)|^p(\e x)\|_1\leq 1\}.
\end{equation*}
Let us also group the simplices overlapping with $A$ according to
\begin{align*}
\mathcal{T}_1&=\{T\in\mathcal{T}:\;T\cap A\neq\emptyset,\,\inf_{z\in\e T}\dist(z,\partial A)\leq 8R\e\},
\\
\mathcal{T}_2&=\{T\in\mathcal{T}:\;T\cap A\neq\emptyset,\,\inf_{z\in\e T}\dist(z,\partial A)>8R\e\; \text{ and }\inf_{z\in\e T}\dist(z,L_{1,\e})\leq 6R\e\},
\\
\mathcal{T}_3&=\{T\in\mathcal{T}:\;T\cap A\neq\emptyset,\,\inf_{z\in\e T}\dist(z,\partial A)>8R\e\; \text{ and }\inf_{z\in\e T}\dist(z,L_{1,\e})>6R\e\}.
\end{align*}
Given $u^{\rm aff}_{\e}$ we define $v_{\e}$ on the interior of each simplex $\e T\in\e\mathcal{T}$ setting
\begin{equation*}
v_{\e_{|\e T}}=\begin{cases}
u^{\rm aff}_{\e_{|\e T}} &\mbox{if $T\in\mathcal{T}_3$,}\\
0 &\mbox{otherwise.} 
\end{cases}
\end{equation*}
Note that $v_{\e}\in SBV^p(A,\R^m)$. We start estimating the difference of $u_{\e}$ and $v_{\e}$ on $A$. Consider any simplex $\e T$ with $T\in\mathcal{T}_1$. Then $\e T\subset \partial A+B_{9R\e}(0)$. Since $\partial A$ is a Lipschitz boundary it admits a $(d-1)$-dimensional Minkowski content. Hence there exists a constant $C=C_R>0$ such that for $\e$ small enough we have
\begin{equation}\label{measest1}
\Big|\bigcup_{T\in\mathcal{T}_1}\e T\cap A\Big|\leq\left|\{z\in A:\;\dist(z,\partial A)\leq 9R\e\}\right|\leq C\mathcal{H}^{d-1}(\partial A)\,\e.
\end{equation}
Next, if $T\in \mathcal{T}_2$, then there exists $\e x\in L_{1,\e}\cap A$ such that $\e T\subset B_{7R\e}(\e x)$. From (\ref{realcut}) and the definition of $L_{1,\e}$ we deduce
\begin{equation}\label{measest2}
\Big|\bigcup_{T\in\mathcal{T}_2}\e T\cap A\Big|\leq C\e^d\,\#\{\e x\in\e\Lw\cap A:\;\|\e|\nabla_{\w,\e}(u,A)|^p(\e x)\|_1 > 1\}
\leq C\e F_{\e}(\w)(u_{\e},A)\leq C\e. 
\end{equation}  
Finally, if $z\in \e T\cap A$ for some $T\in\mathcal{T}_3$, then by definition $\dist(z,L_{1,\e})>6R\e$ and $\dist(z,\partial A)>8R\e$. Let us write $T={\rm co}(z^0,\dots,z^d)\in\mathcal{T}$ and $z=\sum_i\lambda_i^zz^i$. Choosing $x\in\Lw$ such that $z\in \e (\mathcal{C}(x)\cap T)$, Remark \ref{voronoi} yields
\begin{equation}\label{diffest1}
|v_{\e}(z)-u_{\e}(z)|\leq\sum_{i=0}^d\lambda^{z}_i|u^{\rm aff}_{\e}(\e z^i)-u_{\e}(\e x)|\leq \sum_{y\in\Lw\cap B_{3R}(x)}|u_{\e}(\e y)-u_{\e}(\e x)|
\end{equation}  
except for a null set where $u_{\e}$ is not well-defined. Given $y\in \Lw\cap B_{3R}(x)$, we let $P(y,x)=\{y=x_0,x_1,\dots,x_n=x\}$ be the path of Voronoi neighbors given by Lemma \ref{l.paths}. Since $|x-y|\leq 3R$, we have
\begin{equation*}
|x_i-\e^{-1}z|\leq \dist(x_i,{\rm co}(x,y))+|y-x|+|x-\e^{-1}z|< 2R+3R+R=6R.
\end{equation*}
In particular we conclude that $\e x_i\in L_{0,\e}\cap A$ for all $0\leq i\leq n$. Using (\ref{nncontained}), the definition of the set $L_{0,\e}$ implies that $|u_{\e}(\e x_i)-u_{\e}(\e x_{i+1})|\leq\e^{1-\frac{1}{p}}$ for all $0\leq i\leq n-1$. By Remark \ref{voronoi} the number of Voronoi neighbors in $B_{6R}(\e^{-1}z)$ is uniformly bounded, so that from (\ref{diffest1}) we infer the bound
\begin{equation}\label{distest2}
|v_{\e}(z)-u_{\e}(z)|\leq C\e^{1-\frac{1}{p}}.
\end{equation}
Since we have set $v_{\e}=0$ on all $\e T\in \e(\mathcal{T}_1\cup\mathcal{T}_2)$ and $u_{\e}$ is equiintegrable by assumption, we conclude from (\ref{measest1}), (\ref{measest2}) and (\ref{distest2}) that
\begin{equation}\label{l1close}
\lim_{\e\to 0}\|v_{\e}-u_{\e}\|_{L^1(A)}=0.
\end{equation}
Moreover, the sequence $v_{\e}$ is still equiintegrable on $A$. We will show the convergence for the sequence $v_{\e}$. We start estimating the size of its jump set. 
We have at most two contributions. The first one comes from discontinuities along edges of simplices in $\mathcal{T}_1$. For those we have, again for $\e$ small enough, the estimate
\begin{equation*}
\mathcal{H}^{d-1}\left(\bigcup_{T\in\mathcal{T}_1}\partial \e T\cap A\right)\leq C \e^{-1}|\{z\in \R^d:\,\dist(z,\partial A)\leq 9R\e\}|\leq C\mathcal{H}^{d-1}(\partial A),
\end{equation*}
where we used the Lipschitz regularity of $\partial A$ and a reverse isoperimetric inequality for the finitely many simplices in $\mathcal{T}^d$. The other contribution is given by
\begin{equation*}
\mathcal{H}^{d-1}\left(\bigcup_{T\in\mathcal{T}_2}\partial \e T\cap A\right)\leq\sum_{T\in\mathcal{T}_2}C\e^{d-1}\leq CF_{\e}(\w)(u_{\e},A)\leq C,
\end{equation*}
where the second inequality follows by the same reasoning used in the lines preceding (\ref{measest2}). The last two estimates imply that
\begin{equation}\label{jumpest}
\mathcal{H}^{d-1}(S_{v_{\e}}\cap A)\leq\mathcal{H}^{d-1}\left(\bigcup_{T\in\mathcal{T}_1\cup\mathcal{T}_2}\partial \e T\cap A\right)\leq C.
\end{equation}
To estimate the gradient it suffices to consider simplices $T={\rm{co}}(z^0,\dots, z^d)\in\mathcal{T}_3$. Write any basis vector $e_k$ as $e_k=\sum_{i=1}^d\lambda^k_i(z^i-z^0)$. Due to the periodicity of the triangulation $\mathcal{T}$ the coefficients $\lambda^k_i$ are equibounded with respect to the simplices. Take $x\in\Lw$ such that $u^{\rm aff}_{\e}(\e z^0)=u_{\e}(\e x)$. Since $v_{\e}$ is affine on $\e T$ we have
\begin{align*}
|\partial_k v_{\e_{|\e T}}|=&\Big|\sum_{i=1}^d\lambda^k_i\frac{v_{\e}(\e z^i)-v_{\e}(\e z^0)}{\e}\Big|\leq C\sum_{i=0}^d\Big|\frac{u^{\rm aff}_{\e}(\e z^i)-u^{\rm aff}_{\e}(\e z^0)}{\e}\Big|
\\
\leq &C\sum_{y\in\Lw\cap B_{3R}(x)}\e^{-1}\big|u_{\e}(\e y)-u_{\e}(\e x)\big|\leq C\sum_{\substack{x_i,x_j\in B_{5R}(x)\\(x_i,x_j)\in\mathcal{N}(\w)}}\e^{-1}|u_{\e}(\e x_i)-u_{\e}(\e x_j)|,
\end{align*}
where the last inequality follows by the same reasoning we used for proving (\ref{distest2}). Now observe that if $T\in\mathcal{T}_3$ and $|x_i-x|< 5R$, then $\e x_i\in L_{0,\e}\cap A$. Thus taking the $p$-th power of the above estimate and using (\ref{nncontained}) and (\ref{realcut}) we obtain 
\begin{align*}
|\nabla v_{\e_{|\e T}}|^p&\leq C\sum_{\substack{x_i,x_j\in B_{5R}(x)\\(x_i,x_j)\in\mathcal{N}(\w)}}\e^{-p}|u_{\e}(\e x_i)-u_{\e}(\e x_j)|^p\leq C\e^{-1}\sum_{x_i\in \Lw\cap B_{5R}(x)}\|\e|\nabla_{\w,\e}(u,A)|^p(\e x_i)\|_1
\\
&\leq C\e^{-1}\sum_{x_i\in \Lw\cap B_{5R}(x)}f(\e|\nabla_{\w,\e}(u,A)|^p(\e x_i)).
\end{align*}
We sum the last estimate over all $T\in\mathcal{T}_3$. Since $T\subset B_{7R}(x_i)$ we count each lattice point $x_i$ at most $C$ times and thus conclude
\begin{equation}\label{gradest}
\int_{A}|\nabla v_{\e}(z)|^p\,\mathrm{d}z \leq C\sum_{\e x\in \e\Lw\cap A}\e^{d-1}f(\e|\nabla_{\w,\e}(u,A)|^p(\e x))
= CF_{\e}(\w)(u_{\e},A)\leq C.
\end{equation}
By (\ref{jumpest}) and (\ref{gradest}), the compactness theorem for $GSBV(A,\R^m)$-functions \cite[Theorem 2.2]{Amb} implies that, up to subsequences, $v_{\e}\to u\in GSBV(A,\R^m)$ in measure and by equiintegrability also in $L^1(A,\R^m)$. Moreover $\nabla v_{\e}\rightharpoonup \nabla u$ in $L^p(A,\R^{m\times d})$ and from lower semicontinuity \cite[Theorem 3.7]{Amb} we deduce
\begin{equation*}
\int_{A}|\nabla u|^p\,\mathrm{d}z+\mathcal{H}^{d-1}(S_u\cap A)\leq C\liminf_{\e\to 0}F_{\e}(\w)(u_{\e},A)+C\mathcal{H}^{d-1}(\partial A)\leq C.
\end{equation*} 
Thus by definition $u\in GSBV^p(A)$ which finishes the proof of compactness. In order to prove the lower bound, note that the argument above shows that for any open set $A^{\prime}\subset\subset A$ it holds that
\begin{equation*}
\int_{A^{\prime}}|\nabla u|^p\,\mathrm{d}z+\mathcal{H}^{d-1}(S_u\cap A^{\prime})\leq C\liminf_{\e\to 0}F_{\e}(\w)(u_{\e},A),
\end{equation*}
provided the right hand side is finite and $u_{\e}\to u$ in $L^1(D,\R^m)$. By the definition of $F^{\prime}(\w)$ and the arbitrariness of $A^{\prime}$ we obtain the desired estimate.
\end{proof}

As a next step we estimate $F^{\prime\prime}(\w)$ from above.
\begin{lemma}[Upper bound]\label{bounds}
Let $u\in L^1(D,\R^m)$. There exists a constant $c>0$ independent of $\w$ and $u$ such that for all $A\in\mathcal{A}^R(D)$ with $u\in GSBV^p(A,\R^m)$ it holds that
\begin{equation*}
F^{\prime\prime}(\w)(u,A)\leq c\left(\int_A|\nabla u|^p\dx+\mathcal{H}^{d-1}(S_u\cap A)\right).
\end{equation*}
\end{lemma}

\begin{proof}
Take any ball $B_{L}$ such that $D\subset\subset B_L$. For the moment let us assume that $u\in SBV^p(B_L,\R^m)$ is such that 
\begin{enumerate}
	\item[(i)] $\mathcal{H}^{d-1}(\overline{S_u}\backslash S_u\cap B_L)=0$,
	\item[(ii)] $\overline{S_u}$ is the intersection of $B_L$ with a finite number of pairwise disjoint $(d\mbox{--}1)$-simplices,
	\item[(iii)] $u\in W^{k,\infty}(B_L\backslash\overline{S_u},\R^m)$ for all $k\in\mathbb{N}$.
\end{enumerate}
We define an admissible sequence to bound $F^{\prime\prime}(\w)(u_{|D},A)$ setting
\begin{equation*}
u_{\e}(\e x)=
\begin{cases}
u(\e x) &\mbox{if $\e x\in B_L\backslash\overline{S_u}$,}\\
0 &\mbox{otherwise.}
\end{cases}
\end{equation*} 
Using the properties (ii) and (iii) from above it follows by Remark \ref{voronoi} that $u_{\e}\to u_{|D}$ in $L^1(D,\R^m)$. To bound the energy, consider first the case that $\e x\in\e\Lw\cap A$ is such that $\dist(\e x,\overline{S_u})\geq 3M\e$. Then for all $y\in\Lw$ with $(x,y)\in\mathcal{E}(\w)$ we have by Jensen's inequality and the regularity of $u$ that
\begin{equation*}
\left|\frac{u_{\e}(\e x)-u_{\e}(\e y)}{\e}\right|^p= \left|\int_0^1 \nabla u(\e x+s\e(y-x))(y-x)\,\mathrm{d}s\right|^p\leq |y-x|^p\int_0^1|\nabla u(\e x+s\e(y-x))|^p\,\mathrm{d}s.
\end{equation*}
Integrating both sides over $\e\mathcal{C}(x)$ we infer from Fubini's theorem and Remark \ref{voronoi} the bound
\begin{align}\label{discgrad}
\e^d\Big|\frac{u_{\e}(\e x)-u_{\e}(\e y)}{\e}\Big|^p&\leq C\int_{\e \mathcal{C}(x)}\int_0^1|\nabla u(\e x+s\e(y-x))|^p\,\mathrm{d}s\,\mathrm{d}z\nonumber
\\
&\leq C\left(\int_{\e \mathcal{C}(x)}\int_0^1|\nabla u(z+s\e(y-x))|^p\,\mathrm{d}s\,\mathrm{d}z+c^p_{u}\int_{\e \mathcal{C}(x)}|z-\e x|^p\,\mathrm{d}z\right)\nonumber
\\
&\leq C\left(\int_0^1\int_{\e (\mathcal{C}(x)+s(y-x))}|\nabla u(z)|^p\,\mathrm{d}z\,\mathrm{d}s+c^p_{u}\e^{d+p}\right),
\end{align}
where $c_{u}$ denotes the $L^{\infty}$-norm of $D^2 u$ on $B_L\backslash \overline{S_u}$. Here we used that by Remark \ref{voronoi} we have
\begin{equation}\label{shiftedcell}
t\big(\e x+s\e(y-x)\big)+(1-t)\big(z+s\e(y-x)\big)\in B_{2M\e}(\e x)\subset B_L\backslash\overline{S_u},
\end{equation}
for all $z\in\e\mathcal{C}(x)$ and $s,t\in [0,1]$, so that $c_u$ indeed provides Lipschitz estimates for $\nabla u$. Due to (\ref{realcut}) we have $f(\mathfrak{p})\leq C_f \|\mathfrak{p}\|_1$, so that (\ref{neighbours}), (\ref{discgrad}) and (\ref{shiftedcell}) imply
\begin{equation}\label{regularestimate}
\e^{d-1}f\big(\e|\nabla_{\w,\e}(u,A)|^p(\e x)\big)\leq C\left(\int_{B_{2M\e}(\e x)}|\nabla u(z)|^p\,\mathrm{d}z+c_u^p\e^{d+p}\right).
\end{equation}	

In order to control the contribution from the remaining lattice points, fix a set $A^{\prime}\in\mathcal{A}^R(\R^d)$ such that $A\subset\subset A^{\prime}$. Then, for $\e$ small enough, Remark \ref{voronoi} yields the estimate
\begin{equation*}
\e^{d-1}\#\left\{\e x\in\e\Lw\cap A:\;\dist(\e x,\overline{S_u})<3M\e \right\}\leq \frac{|(\overline{S_{u}}\cap A^{\prime})+B_{4M\e}(0)|}{\e|B_{\frac{r}{2}}(0)|}.
\end{equation*}
Recall that $\overline{S_u}$ is the intersection of $B_L$ with a finite union of pairwise disjoint $(d\hbox{--}1)$-simplices, so that $\overline{S_u}\cap\overline{A^{\prime}}$ admits a $(d-1)$-dimensional Minkowski content. Hence, letting $\e\to 0$, it holds that 
\begin{equation*}
\limsup_{\e\to 0}\frac{|(\overline{S_{u}}\cap A^{\prime})+B_{4M\e}(0)|}{\e|B_{\frac{r}{2}}(0)|}\leq C\mathcal{H}^{d-1}(\overline{S_u}\cap \overline{A^{\prime}})=\mathcal{H}^{d-1}(S_u\cap\overline{A^{\prime}}),
\end{equation*}
where we used assumption (i) in the second identity. Since $f$ is bounded, from (\ref{regularestimate}) and the bound above we conclude that
\begin{align*}
\limsup_{\e\to 0}F_{\e}(\w)(u_{\e},A)&\leq \limsup_{\e\to 0}\sum_{\e x\in\e\Lw\cap A}C\left(\int_{B_{2M\e}(\e x)}|\nabla u(z)|^p\,\mathrm{d}z+c^p_{u}\e^{d+p}\right)+C\,\mathcal{H}^{d-1}(S_u\cap \overline{A^{\prime}})\\
&\leq C\int_{A^{\prime}}|\nabla u(z)|^p\,\mathrm{d}z+C\,\mathcal{H}^{d-1}(S_u\cap \overline{A^{\prime}}).
\end{align*}
Letting $A^{\prime}\downarrow\overline{A}$ in this estimate yields by definition of $F^{\prime\prime}(\w)$ that
\begin{equation}\label{densityest}
F^{\prime\prime}(\w)(u_{|D},A)\leq C\int_{A}|\nabla u(z)|^p\,\mathrm{d}z+C\,\mathcal{H}^{d-1}(S_u\cap \overline{A}).
\end{equation}
From this estimate we can now prove the claim by density. First we assume that $u\in SBV^p(A,\R^m)\cap L^{\infty}(A,\R^m)$. Due to the Lipschitz regularity of $\partial A$ we can use a local reflection argument to extend $u$ to a function $\tilde{u}\in SBV^p(B_L,\R^m)\cap L^{\infty}(B_L,\R^m)$ such that $\mathcal{H}^{d-1}(S_{\tilde{u}}\cap\partial A)=0$. By \cite[Theorem 3.1]{CoTo} applied to the large set $B_L$ we find a sequence $u_n\in SBV^p(B_L,\R^m)$ fulfilling assumptions (i)-(iii) of the first part such that $u_n\to \tilde{u}$ in $L^1(B_L,\R^m)$, $\nabla u_n\to\nabla \tilde{u}$ in $L^p(B_L,\R^{m\times d})$ and $\limsup_n\mathcal{H}^{d-1}(S_{u_n}\cap \overline{A})\leq\mathcal{H}^{d-1}(S_{\tilde{u}})\cap\overline{A})=\mathcal{H}^{d-1}(S_u\cap A)$. From locality and lower semicontinuity of $F^{\prime\prime}(\w)(\cdot,A)$  and (\ref{densityest}) we deduce
\begin{equation*}
F^{\prime\prime}(u,A)=F^{\prime\prime}(\tilde{u}_{|D},A)\leq \liminf_n F^{\prime\prime}(\w)({u_n}_{|D},A)\leq C\int_A|\nabla u(z)|^p\,\mathrm{d}z+C\mathcal{H}^{d-1}(S_u\cap A).
\end{equation*}
It remains to remove the $L^{\infty}$-bound. To this end, given any $u\in GSBV^p(A,\R^m)\cap L^1(D,\R^m)$, we consider the truncated sequence $T_ku\in SBV^p(A,\R^m)\cap L^{\infty}(A,\R^m)$. Then $u_k\to u$ in $L^1(D,\R^m)$ and, as in the previous reasoning, the claim follows by lower semicontinuity of $F^{\prime\prime}(\w)(\cdot,A)$, Lemma \ref{truncation} and the estimate established for bounded functions.
\end{proof}

The following technical lemma establishes an almost subadditivity of the set function $A\mapsto F^{\prime\prime}(\w)(u,A)$.
\begin{proposition}[Almost subadditivity]\label{subadd}
Let $A,B\in\Ard$. Moreover let $A^{\prime}\in\Ard$ be such that $A^{\prime}\subset\subset A$. Then, for all $u\in L^1(D,\R^m)$,
\begin{equation*}
F^{\prime\prime}(\w)(u,A^{\prime}\cup B)\leq F^{\prime\prime}(\w)(u,A)+F^{\prime\prime}(\w)(u,B).
\end{equation*}
\end{proposition}
\begin{proof}
We can assume that $F^{\prime\prime}(u,A)$ and $F^{\prime\prime}(u,B)$ are both finite. Since $A^{\prime}\cup B\in\mathcal{A}(D)$, Lemma \ref{trunc} allows us to reduce the proof to the case $u\in L^{\infty}(D,\R^m)$. Let $u_{\e},v_{\e}\in \mathcal{PC}_{\e}^{\w}$ both converge to $u$ in $L^1(D,\R^m)$ such that
\begin{equation}\label{recovery}
\limsup_{\e\to 0}F_{\e}(\w)(u_{\e},A)=F^{\prime\prime}(\w)(u,A),\quad\limsup_{\e\to 0}F_{\e}(\w)(v_{\e},B)=F^{\prime\prime}(\w)(u,B).
\end{equation}
By Lemma \ref{trunc} we may assume that  $\|u_{\e}\|_{\infty},\|v_{\e}\|_{\infty}\leq 3\|u\|_{\infty}$, so that both sequences actually converge to $u$ in $L^p(D,\R^m)$. Fix $h\leq\dist(A^{\prime},A^c)$ and $N\in\mathbb{N}$. For $i=1,\dots,N$ we define the sets
\begin{equation*}
A_{i}:=\left\{x\in A:\;\dist(x,A^{\prime})<i\frac{h}{2N}\right\}.
\end{equation*}
Let $0\leq\Theta_i\leq 1$ be a cut-off function between the sets $A_i$ and $A_{i+1}$, that means $\Theta_i=1$ on $A_i$ and $\Theta_i=0$ on $\R^d\backslash A_{i+1}$. We may assume that $\|\nabla\Theta_i\|_{\infty}\leq \frac{4N}{h}$. Then define $w^i_{\e}\in \mathcal{PC}_{\e}^{\w}$ by
\begin{equation*}
w^i_{\e}(\e x)=\Theta_i(\e x)u_{\e}(\e x)+(1-\Theta_i(\e x))v_{\e}(\e x).
\end{equation*}
Note that for fixed $i\in\{1,\dots,N\}$ it holds that $w^i_{\e}\to u$ in $L^p(D,\R^m)$. We define the layer-like set
\begin{equation*}
S_{\e}^{i}:=\{x\in A^{\prime}\cup B:\;\dist(x,A_{i+1}\backslash A_{i-1})< 3M\e\}.
\end{equation*}
Then by the definition of the localized functionals we can decompose $F_{\e}(\w)(w^i_{\e},A^{\prime}\cup B)$ via
\begin{align}\label{almostsubeq}
F_{\e}(\w)(w^i_{\e},A^{\prime}\cup B)&\leq\, F_{\e}(\w)(u_{\e},A_{i})+F_{\e}(\w)(v_{\e},B\backslash \overline{A_{i+1}})
+F_{\e}(\w)(w_{\e}^i,S_{\e}^i)\nonumber
\\
&\leq F_{\e}(\w)(u_{\e},A)+F_{\e}(\w)(v_{\e},B)
+F_{\e}(\w)(w_{\e}^i,S_{\e}^i).
\end{align}
We show that the last term is negligible. This will be done by averaging. Observe that
\begin{align*}
w^i_{\e}(\e y)-w^i_{\e}(\e x)=&\Theta_i(\e y)(u_{\e}(\e y)-u_{\e}(\e x))+(1-\Theta_i(\e y))(v_{\e}(\e y)-v_{\e}(\e x))
\\
&+(\Theta_i(\e y)-\Theta_i(\e x))(u_{\e}(\e x)-v_{\e}(\e x))
\end{align*}
for all $x,y\in\Lw$. Applying the convexity inequality $(a+b+c)^p\leq 3^{p-1}(a^p+b^p+c^p)$ and the mean value theorem for $\Theta_i$, we obtain for all $(x,y)\in\mathcal{E}(\w)$ the bound
\begin{equation*}
\e\Big|\frac{w^i_{\e}(\e y)-w_{\e}^i(\e x)}{\e}\Big|^p\leq 3^{p-1}\e\Big|\frac{u_{\e}(\e y)-u_{\e}(\e x)}{\e}\Big|^p+3^{p-1}\e\Big|\frac{v_{\e}(\e y)-v_{\e}(\e x)}{\e}\Big|^p+\frac{(12MN)^p}{3h^p}\e|u_{\e}(\e x)-v_{\e}(\e x)|^p.
\end{equation*}
Summing this estimate over all $\e y\in\e\Lw\cap S_{\e}^i$ with $(x,y)\in\mathcal{E}(\w)$ we infer
\begin{align}
\|\e|\nabla_{\w,\e}(w^i_{\e}, S_{\e}^i)|^p(\e x)\|_1\leq& 3^{p-1}\|\e|\nabla_{\w,\e}(u_{\e}, S_{\e}^i)|^p(\e x)\|_1+3^{p-1}\|\e|\nabla_{\w,\e}(v_{\e}, S_{\e}^i)|^p(\e x)\|_1\nonumber
\\
&+C N^p\e |u_{\e}(\e x)-v_{\e}(\e x)|^p.\label{gradinequ}
\end{align}
Note that for all $\lambda\geq 1$ and $x,y\in\R$ it holds that
\begin{equation*}
\min\{\lambda x+y,1\}\leq\lambda\min\{x,1\}+\min\{y,1\}.
\end{equation*}
We combine this estimate with the bound \eqref{realcut} and the monotonicity of $x\mapsto\min\{x,1\}$ to deduce from \eqref{gradinequ} that
\begin{align*}
\e^{d-1}f(\e|\nabla_{\w,\e}(w^i_{\e}, S_{\e}^i)|^p(\e x))\leq & C\e^{d-1}f(\e|\nabla_{\w,\e}(u_{\e}, S_{\e}^i)|^p(\e x))+C \e^{d-1}f(\e|\nabla_{\w,\e}(v_{\e}, S_{\e}^i)|^p(\e x))
\\
&+C N^p\e^d |u_{\e}(\e x)-v_{\e}(\e x)|^p.
\end{align*} 
%
%
%
%Having in mind this inequality we subdivide the lattice points according to
%\begin{align*}
%I^0_{\e,i}&=\{\e x\in \e\Lw\cap S_{\e}^{i}:\,\|\e|\nabla_{\w,\e}(u_{\e}, S_{\e}^i)|^p(\e x)\|_1+\|\e|\nabla_{\w,\e}(v_{\e}, S_{\e}^i)|^p(\e x)\|_1\leq 3^{-p}\},\\
%I^1_{\e,i}&=(\e\Lw\cap S_{\e}^i)\backslash I^0_{\e,i}.
%\end{align*}
%Since $f(p)\leq C_f\|p\|_1$, the estimate (\ref{gradinequ}) and the lower bound in (\ref{realcut}) imply for $\e x\in I_{\e,i}^0$ that
%\begin{align}
%\e^{d-1}f(\e|\nabla_{\w,\e}(w^i_{\e}, S_{\e}^i)|^p(\e x))\leq& C\e^{d-1}f(\e|\nabla_{\w,\e}(u_{\e}, S_{\e}^i)|^p(\e x))+C\e^{d-1}f(\e|\nabla_{\w,\e}(v_{\e}, S_{\e}^i)|^p(\e x))\nonumber
%\\
%&+CN^p\e^d |u_{\e}(\e x)-v_{\e}(\e x)|^p.\label{convexbound}
%\end{align}
%If $\e x\in I^1_{\e,i}$, we can use (\ref{realcut}) and the inequality $\min\{x+y,1\}\leq \min\{x,1\}+\min\{y,1\}$ to deduce
%\begin{equation}\label{nonconvexbound}
%\e^{d-1}f(\e|\nabla_{\w,\e}(w^i_{\e}, S_{\e}^i)|^p(\e x))
%\leq C\e^{d-1}f(\e|\nabla_{\w,\e}(u_{\e}, S_{\e}^i)|^p(\e x))+C\e^{d-1}f(\e|\nabla_{\w,\e}(v_{\e}, S_{\e}^i)|^p(\e x)).
%\end{equation}
%Summing (\ref{convexbound}) and (\ref{nonconvexbound}) yields
Summing this inequality yields
\begin{align*}
F_{\e}(\w)(w_{\e}^i,S_{\e}^i)\leq &C\big(F_{\e}(\w)(u_{\e}, S_{\e}^i)+F_{\e}(\w)(v_{\e},S_{\e}^i)\big)+CN^p\sum_{\e x\in \e\Lw\cap S^i_{\e}}\e^d|u_{\e}(\e x)-v_{\e}(\e x)|^p.
\end{align*}
For $\e$ small enough we have $S_{\e}^i\cap S_{\e}^j=\emptyset$ for $|i-j|\geq 3$. Moreover, $S_{\e}^i\subset A\cap B$ for $i\geq 2$ as well as $S_{\e}^i\subset\subset A$ with a uniform distance to $\partial A$. Thus averaging the last inequality and applying (\ref{recovery}) gives
\begin{align*}
\frac{1}{N-1}\sum_{i=2}^{N}F_{\e}(\w)(w_{\e}^i,S_{\e}^i)&\leq\frac{C}{N}\left(F_{\e}(\w)(u_{\e},A)+F_{\e}(\w)(v_{\e},B)\right)+CN^{p-1}\sum_{\e \mathcal{C}(x)\subset A}\e^d|u_{\e}(\e x)-v_{\e}(\e x)|^p\nonumber
\\
&\leq\frac{C}{N}+CN^{p-1}\|u_{\e}-v_{\e}\|^p_{L^p(D)}.
\end{align*}
Since we have $u_{\e}-v_{\e}\to 0$ also in $L^p(D,\R^m)$, the last term vanishes when $\e\to 0$. For every $\e>0$ let $i_{\e}\in\{2,\dots,N\}$ be such that
\begin{equation}\label{minimal}
F_{\e}(\w)(w_{\e}^{i_{\e}},S_{\e}^i)\leq \frac{1}{N-1}\sum_{i=2}^{N}F_{\e}(\w)(w_{\e}^i,S_{\e}^i)\leq\frac{C}{N}+CN^{p-1}\|u_{\e}-v_{\e}\|^2_{L^p(D)} 
\end{equation}
and set $w_{\e}:=w_{\e}^{i_{\e}}$. Note that $w_{\e}$ still converges to $u$ strongly in $L^p(D,\R^m)$. Hence, using (\ref{recovery}), (\ref{almostsubeq}) and \eqref{minimal}, we conclude that
\begin{equation*}
F^{\prime\prime}(\w)(u,A^{\prime}\cup B)\leq\limsup_{\e\to 0}F_{\e}(\w)(w_{\e},A^{\prime}\cup B)\leq F^{\prime\prime}(\w)(u,A)+F^{\prime\prime}(\w)(u,B)+\frac{C}{N}.
\end{equation*}
The claim follows by letting $N\to +\infty$.
\end{proof}
In the lemma below we state the last property that we need to prove Proposition \ref{limitsbvp}.
\begin{lemma}[Inner regularity]\label{almostmeasure}
Let $u\in L^1(D)$. Then for any $A\in\Ard$ it holds that
\begin{equation*}
F^{\prime\prime}(\w)(u,A)=\sup_{A^{\prime}\subset\subset A}F^{\prime\prime}(\w)(u,A^{\prime}).
\end{equation*}
\end{lemma}
\begin{proof}
It suffices to prove one inequality since $A\mapsto F^{\prime\prime}(\w)(u,A)$ is monotone with respect to set inclusion. For $k\in\mathbb{N}$ define the set $A_k=\{x\in A:\;\dist(x,\partial A)>2^{-k}\}$. Then for $k$ large enough we have that $A_k,A\backslash A_k\in\Ard$ (see \cite[Lemma 2.2]{glpe}). We first treat the case $u\notin GSBV^p(A,\R^m)$ and prove that $\limsup_k F^{\prime\prime}(\w)(u,A_k)=+\infty$. Assume by contradiction that this sequence is bounded. Then, for each $k$ we have $u\in GSBV^p(A_k,\R^m)$ by Lemma \ref{compact} and thus $u\in GSBV(A,\R^m)$. Since the measure of the jump set and the $L^p$-norm of the gradient in $A_k$ are equibounded with respect to $k$, we reach the contradiction $u\in GSBV^p(A,\R^m)$. Now assume that $u\in GSBV^p(A,\R^m)$. Note that $A=A_{k+1}\cup A\backslash \overline{A_{k}}$. Hence Lemma \ref{bounds} and Proposition \ref{subadd} imply
\begin{align*}
F^{\prime\prime}(\w)(u,A)&\leq F^{\prime\prime}(\w)(u,A_{k+2})+C\Big(\|\nabla u\|^p_{L^p(A\backslash\overline{A_{k}})}+\mathcal{H}^{d-1}(S_u\cap (A\backslash\overline{A_{k}}))\Big)
\\
&\leq \sup_{A^{\prime}\subset\subset A}F^{\prime\prime}(\w)(u,A^{\prime})+C\Big(\|\nabla u\|^p_{L^p(A\backslash\overline{A_{k}})}+\mathcal{H}^{d-1}(S_u\cap (A\backslash\overline{A_{k}}))\Big).
\end{align*}
Letting $k\to +\infty$ proves the claim since $u\in GSBV^P(A,\R^m)$.
\end{proof}

\begin{proof}[Proof of Proposition \ref{limitsbvp}]
Given a sequence $\e\to 0^+$, by the compactness property of $\Gamma$-convergence on separable metric spaces (see \cite[Proposition 1.42]{GCB}) we find a subsequence $\e_n$ such that 
\begin{equation*}
\Gamma\hbox{-}\lim_n F_{\e_n}(\w)(u,R)=:\tilde{F}(\w)(u,R)
\end{equation*}
exists for every $u\in L^1(D,\R^m)$ and all sets $R\in\mathcal{R}$, where $\mathcal{R}$ denotes the class of all subsets of $D$ that are finite unions of rectangles with rational vertices. Due to Lemma \ref{almostmeasure} and monotonicity, for every $u\in L^1(D,\R^m)$ and $A\in\Ard$ we conclude that
\begin{equation*}
\Gamma\hbox{-}\limsup_nF_{\e_n}(\w)(u,A)\leq \sup_{R\in\mathcal{R}}\tilde{F}(\w)(u,R)\leq {\Gamma\hbox{-}\liminf_n}F_{\e_n}(\w)(u,A),
\end{equation*}
so that we can define $\tilde{F}(\w)(u,A):=\Gamma\hbox{-}\lim_nF_{\e_n}(u,A)$ also for all $u\in L^1(D,\R^m)$ and $A\in\Ard$. We extend $\tilde{F}(\w)(u,\cdot)$ to $\mathcal{A}(D)$ via its inner regular envelope
\begin{equation*}
F(\w)(u,A):=\sup\,\{\tilde{F}(\w)(u,A^{\prime}):\;A^{\prime}\subset\subset A,\, A^{\prime}\in\Ard\}.
\end{equation*}
By Lemma \ref{almostmeasure} this functional indeed extends $\tilde{F}(\w)(u,\cdot)$. Next we need to slightly perturb the functional. Given $\eta>0$, for any $u\in SBV^p(D,\R^m)$ and $A\in\mathcal{A}(D)$ we define the auxiliary functional $\mathcal{F}_{\eta}:SBV^p(D,\R^m)\times\mathcal{A}(D)\to[0,+\infty)$ as
\begin{equation*}%\label{perturbation}
\mathcal{F}_{\eta}(u,A)=F(\w)(u,A)+\eta\int_{S_u\cap A}|u^+-u^-|\,\mathrm{d}\mathcal{H}^{d-1}.
\end{equation*}
We argue that $\mathcal{F}_{\eta}$ satisfies the assumptions of \cite[Theorem 1]{BFLM}, that means,
\begin{itemize}
	\item[(i)] ${\mathcal F}_{\eta}(u, \cdot)$ is the restriction to ${\mathcal A}(D)$ of a Radon measure;
	\item[(ii)] ${\mathcal F}_{\eta}(u, A)={\mathcal F}_{\eta}(v, A)$ whenever $u = v$ a.e. on $A\in{\mathcal A}(D)$;
	\item[(iii)] ${\mathcal F}_{\eta}(\cdot, A)$ is $L^1(D,\R^m)$-lower semicontinuous;
	\item[(iv)] there exists $c>0$ such that 
	\begin{equation*}
	\begin{split}\frac 1 c &\left(\int_A|\nabla u|^p\dx+\int_{S_u\cap A}(1+|u^+-u^-|)\,\mathrm{d}\mathcal{H}^{d-1}\right)\leq  {\mathcal F}_{\eta}(u, A)\\
	&\leq c \left(\int_A(1+|\nabla u|^p)+\int_{S_u\cap A}(1+|u^+-u^-|)\,\mathrm{d}\mathcal{H}^{d-1}\right).
	\end{split}
	\end{equation*}
\end{itemize}

(i): We verify the De Giorgi-Letta criterion (see \cite[Theorem 1.62]{FoLe}). Clearly $A\mapsto\mathcal{F}_{\eta}(u,A)$ is a non-negative, increasing and inner regular set function with $\mathcal{F}_{\eta}(u,\emptyset)=0$. Moreover, discrete superadditivity on disjoint sets is transfered from $F_{\e}(\w)(u,\cdot)$ to $F^{\prime}(\w)(u,\cdot)$, which implies that $\mathcal{F}_{\eta}(u,A\cup B)\geq \mathcal{F}_{\eta}(u,A)+\mathcal{F}_{\eta}(u,B)$ whenever $A\cap B=\emptyset$. In order to prove subadditivity, let $A,B\in\mathcal{A}(D)$ and consider $S\subset\subset A\cup B$ such that $S\in\Ard$. Let us define the set $A_k=\{x\in A:\,\dist(x,\partial A)>2^{-k}\}$ and similarly we define $B_k$. Then the family $\{A_k\cup B_k\}_k$ forms an open cover of $\overline{S}$. By compactness we find an index $k_0$ such that $\overline{S}\subset A_{k_0}\cup B_{k_0}$. Next we regularize the sets $A_{2k_0},B_{2k_0}$ and $A_{4k_0},B_{4k_0}$ by standard methods to find further sets $A_0,A_1,B_0,B_1\in\Ard$ such that $\overline{S}\subset A_0\cup B_0$ and $A_0\subset\subset A_1\subset\subset A$ and $B_0\subset\subset B_1\subset\subset B$. Then by Proposition \ref{subadd}
\begin{equation*}
\tilde{F}(\w)(u,S)\leq \Gamma\hbox{-}\limsup_n F_{\e_n}(\w)(u,A_0\cup B_0)\leq \tilde{F}(\w)(u,A_1)+\tilde{F}(\w)(u,B_1)\leq F(\w)(u,A)+F(\w)(u,B).
\end{equation*} 
Taking the supremum over such $S$ yields the subadditivity of $A\mapsto F(\w)(u,A)$. The corresponding property for the perturbation term is straightforward. Thus the De Giorgi-Letta criterion applies and we infer that $\mathcal{F}_{\eta}(u,\cdot)$ is the trace of a Borel measure. By Lemma \ref{bounds} this measure is finite, so it is a Radon measure. 

(ii)+(iii): The locality property follows from Lemma \ref{local} and the definition of $F(\w)$ by inner approximation as well as locality of the perturbation term. By the properties of $\Gamma$-limits we know that $\tilde{F}(\w)(\cdot,A)$ is $L^1(D,\R^m)$-lower semicontinuous and so is $F(\w)(\cdot,A)$ as the supremum of lower semicontinuous functionals. $L^1(D,\R^m)$-lower semicontinuity of the perturbation term along sequences such that $F(\w)(u_n,A)$ remains bounded follows from the bounds established in Lemma \ref{compact} that still hold for $F(\w)(u,A)$. Indeed those bounds yield stronger compactness so that we can combine \cite[Theorems 2.2 and 3.7]{Amb} to conclude lower semicontinuity. Hence $\mathcal{F}(\cdot,A)$ is lower semicontinuous as the sum of (finite) lower semicontinuous functionals. 

(iv): The bounds follow from Lemmata \ref{compact} and \ref{bounds}, which still hold for $F(\w)$ in place of $\tilde{F}(\w)$, and the definition of the perturbation term.

From \cite[Theorem 1]{BFLM} and the fact that $F_{\e}(\w)(u+z,A)=F_{\e}(\w)(u,A)$ for all $z\in\R^m$ we deduce that $\mathcal{F}_{\eta}$ has the representation
\begin{equation}\label{F_eta}
\mathcal{F}_{\eta}(u,A)=\int_Ah_{\eta}(x,\nabla u)\,\mathrm{d}x+\int_{S_u\cap A}\varphi_{\eta}(x,{u^+}-{u^-},\nu_u)\,\mathrm{d}\mathcal{H}^{d-1}
\end{equation}
for all $u\in SBV^p(D,\R^m)$ and $A\in\mathcal{A}(D)$ with the integrands given by the asymptotic formulas
\begin{equation*}
\begin{split}
h_{\eta}(x_0,\xi)&=\limsup_{\varrho\to 0}\frac{1}{\varrho^d}\inf\{\mathcal{F}_{\eta}(v,Q_{\nu}(x_0,\varrho)):\;v=\xi(\cdot-x_0)\text{ in a neighborhood of }\partial Q_{\nu}(x_0,\varrho)\},
\\
\varphi_{\eta}(x_0,a,\nu)&=\limsup_{\varrho\to 0}\frac{1}{\varrho^{d-1}}\inf\{\mathcal{F}_{\eta}(v,Q_{\nu}(x_0,\varrho)):\;v=u_{x_0,\nu}^{a,0}\text{ in a neighborhood of }\partial Q_{\nu}(x_0,\varrho)\}.
\end{split}
\end{equation*}
Hence, due to locality, for every $u\in L^1(D,\R^m)$ and $A\in\Ard$ such that $u\in SBV^p(A,\R^m)$ we have
\begin{equation*}
\Gamma\hbox{-}\lim_nF_{\e_n}(\w)(u,A)=F(\w)(u,A)=\int_A h_{\eta}(x,\nabla u)\,\mathrm{d}x+\int_{S_u\cap A}(\varphi_{\eta}(x,{u^+}-{u^-},\nu_u)-\eta|u^+-u^-|)\,\mathrm{d}\mathcal{H}^{d-1}.
\end{equation*}
It remains to prove the formulas for the integrands stated in Proposition \ref{limitsbvp}. Note that the mapping $\eta\mapsto \varphi_{\eta}(x_0,a,\nu)$ is increasing and non-negative on $(0,+\infty)$. Hence there exists the limit $\varphi(x_0,a,\nu)=\lim_{\eta\to 0}\varphi_{\eta}(x_0,a,\nu)$. By the same reasoning there exists $h(x_0,\xi)=\lim_{\eta\to 0}h_{\eta}(x_0,\xi)$, so that monotone convergence implies that 
\begin{equation*}
F(\w)(u,A)=\int_A h(x,\nabla u)\,\mathrm{d}x+\int_{S_u\cap A}{\varphi}(x,{u^+}-{u^-},\nu_u)\,\mathrm{d}\mathcal{H}^{d-1}
\end{equation*}
for every $A\in\Ard$ and every $u\in L^1(D,\R^m)$ such that $u\in SBV^p(A,\R^m)$. Since $F(\w)(u,A)\leq \mathcal{F}_{\eta}(u,A)$, we further know by the definition of $m(\w)(v,A)$ (cf. the statement of Proposition \ref{limitsbvp}) that
\begin{equation*}
\limsup_{\varrho\to 0}\varrho^{1-d}m(\w)(u_{x_0,\nu}^{a,0},Q_{\nu}(x_0,\varrho))\leq \lim_{\eta\to 0}\varphi_{\eta}(x_0,a,\nu)=\varphi(x_0,a,\nu).
\end{equation*}
In order to show the reverse inequality, we note that Lemma \ref{trunc} implies a very weak maximum principle for $F(\w)$: there exists $u_{\varrho}\in SBV^p(Q_{\nu}(x_0,\varrho),\R^m)$ admissible for the definition of $m(\w)(u_{x_0,\nu}^{a,0},Q_{\nu}(x_0,\varrho))$ and satisfying $\|u\|_{\infty}\leq 3|a|$ such that
\begin{equation*}
\varrho^{1-d}F(\w)(u_{\varrho},Q_{\nu}(x_0,\varrho))\leq m(\w)(u_{x_0,\nu}^{a,0},Q_{\nu}(x_0,\varrho))+\varrho.
\end{equation*}
Then clearly $|u_{\varrho}^+-u_{\varrho}^-|\leq 6|a|$ for $\mathcal{H}^{d-1}$-a.e. $x\in S_{u_{\varrho}}$. With the lower bound of Lemma \ref{compact} we obtain
\begin{align*}
\varphi_{\eta}(x_0,a,\nu)&\leq \limsup_{\varrho\to 0}\varrho^{1-d}\Big(F(\w)(u_{\varrho},Q_{\nu}(x_0,\varrho))+6|a|\eta\mathcal{H}^{d-1}(S_{u_{\varrho}}\cap Q_{\nu}(x_0,\varrho))\Big)
\\
&\leq \limsup_{\varrho\to 0}\varrho^{1-d}\Big(m(\w)(u_{x_0,\nu}^{a,0},Q_{\nu}(x_0,\varrho))+C|a|\eta F(\w)(u_{\varrho},Q_{\nu}(x_0,\varrho))\Big)
\\
&\leq (1+C|a|\eta)\limsup_{\varrho\to 0}\varrho^{1-d}m(\w)(u_{x_0,\nu}^{a,0},Q_{\nu}(x_0,\varrho)).
\end{align*}
Since the term on the right hand side is finite, we conclude by taking the limit as $\eta\to 0$. The proof of the formula for $h$ is the same except that we can choose $u_{\varrho}$ even such that $\|u_{\varrho}\|_{\infty}\leq C |\xi|\varrho$, so that there is no need to let $\eta\to 0$ at the end. This finishes the proof of Proposition \ref{limitsbvp}.
\end{proof}  

\subsection{Characterization of the bulk density}
In this section we show that the function $h$ given by Proposition \ref{limitsbvp} agrees with the density of the $\Gamma$-limit of the discrete functionals defined in \eqref{auxelastic}. 

Let us briefly explain the strategy: for both functionals $F_{\e}(\w)$ and $E_{\e}(\w)$ we consider recovery sequences for affine functions instead of minimization problems with affine boundary conditions since this approach shortens the proof. Our aim is to modify those sequences on a small set such that $\|\e|\nabla_{\w,\e}(u_{\e},A)|^p\|_1$ is equiintegrable because at that level $F_{\e}(\w)$ and $E_{\e}(\w)$ become comparable. It is well-known that if a sequence is bounded in $W^{1,p}$, then up to a subsequence one can modify it on a set whose measure goes to zero in such a way that the gradients become $p$-equiintegrable (see \cite[Lemma 1.2]{FoMuPe}). This is usually achieved via an abstract Lipschitz-extension on the set where the maximal function of the norm of the gradient is very large. Indeed, this classical strategy goes back at least to \cite{AcFu}. To control the size of the set where the function is modified, one estimates the maximal function by the gradient itself, which is possible only for $p>1$. In our case no a priori $L^p$-bounds on the gradient are available. However, in \cite{Larsen} the modification procedure was extended to $BV$-sequences with vanishing singular part. In our setting the basic idea is the following observation: on a ball of radius $\varrho$ the energy of a recovery sequence can be trivially bounded by $\sim\varrho^d$. Morally speaking, this shows that the set where no $L^p$-bounds are available is relatively small with a quantitative rate in $\varrho$. In suitable diagonal regimes of $\varrho$ and $\e$ this can be exploited to still perform a modification on small sets which makes the gradients equiintegrable. However, we emphasize that we have to transfer these ideas to the discrete environment.
\\

The analysis splits into three different parts. As a first step, we argue that affine functions indeed fully characterize the function $h$ and establish the framework for the diagonal argument on small balls $B_{\varrho}(x_0)$. Then we introduce the notion of discrete maximal functions on stochastic lattices and prove the doubling property of the counting measure as well as a Poincar\'e inequality. Those properties allow to perform modifications on small sets by abstract Lipschitz extensions. Finally, with these tools at hand, we can carefully modify recovery sequences on smaller and smaller balls to conclude by a blow-up argument.  
\\

Although it is just a technical detail, we first prove that $h$ is a Carath\'eodory function, which is necessary to conclude its quasiconvexity.

\begin{lemma}\label{caratheodory}
Let $h:D\times\R^{m\times d}\to [0,+\infty)$ be given by Proposition \ref{limitsbvp}. Then, for every $\xi\in\mathbb{R}^{m\times d}$ the map $x\mapsto h(x,\xi)$ is measurable and, for every $x\in D$, the map $\xi\mapsto h(x,\xi)$ is continuous.
\end{lemma}
\begin{proof}
Denoting by $u_{x_0,\xi}$ the affine function $u_{x_0,\xi}(x)=\xi (x-x_0)$, the function $x\mapsto h_{\eta}(x,\xi)$ defined in the proof of Proposition \ref{limitsbvp} is the Radon-Nikodym derivative of the measure $\mathcal{F}_{\eta}(u_{x_0,\xi},\cdot)$ with respect to the Lebesgue-measure. Hence it is measurable and so is $x\mapsto h(x,\xi)$ as the pointwise limit. In order to prove continuity in $\xi$, let us write $Q_{\varrho}=Q_{\nu}(x_0,\varrho)$ and fix $\xi_1,\xi_2\in\mathbb{R}^{m\times d}$ and $\eta>0$. Consider a smooth function $0\leq\Theta\leq 1$ such that $\Theta=1$ on $Q_{\varrho}$ and ${\rm supp}(\Theta)\subset Q_{(1+\eta)\varrho}$ satisfying in addition $\|\nabla \Theta\|_{\infty}\leq\frac{2}{\varrho\eta}$. Given $u_1\in SBV^p(Q_{\varrho},\R^m)$ such that $u_1=u_{x_0,\xi_1}$ in a neighborhood of $\partial Q_{\varrho}$ (extended to the whole space) we define $u_2\in SBV^p(Q_{(1+2\eta)\varrho},\R^m)$ by
\begin{equation*}
u_2=\Theta u_1+(1-\Theta)u_{x_0,\xi_2}.
\end{equation*}
Then $\mathcal{H}^{d-1}(S_{u_2}\cap (Q_{(1+2\eta)\varrho}\backslash Q_{\varrho}))=0$ and $u_2=u_{x_0,\xi_2}$ in a neighborhood of $\partial Q_{(1+2\eta)\varrho}$. Hence by Lemma \ref{bounds}
\begin{align*}
\frac{m(\w)(u_{x_0,\xi_2},Q_{(1+2\eta)\varrho})}{((1+2\eta)\varrho)^d}&\leq \frac{F(\w)(u_2,Q_{(1+2\eta)\varrho})}{\varrho^d}\leq \frac{F(\w)(u_1,Q_{\varrho})}{\varrho^d}+C\varrho^{-d}\int_{Q_{(1+2\eta)\varrho}\backslash \overline{Q_{\varrho}}}|\nabla u_2|^p\,\mathrm{d}z
\\
&\leq \frac{F(\w)(u_1,Q_{\varrho})}{\varrho^d}+C((1+2\eta)^d-1) \Big((|\xi_1|^p+|\xi_2|^p) +\frac{\varrho^p}{(\varrho\eta)^p}|\xi_1-\xi_2|^p\Big).
\end{align*}
Since $u_1$ was arbitrary, by the definition of $h$ we obtain for fixed $0<\eta<\frac{1}{2}$ the inequality
\begin{equation*}
h(x_0,\xi_2)=\limsup_{\varrho\to 0}\frac{m(\w)(u_{x_0,\xi_2},Q_{(1+2\eta)\varrho})}{((1+2\eta)\varrho)^d}\leq h(x_0,\xi_1)+C\eta \Big((|\xi_1|^p+|\xi_2|^p) +\frac{1}{\eta^p}|\xi_1-\xi_2|^p\Big).
\end{equation*}
Exchanging the roles of $\xi_1$ and $\xi_2$, for any $0<\eta<\frac{1}{2}$ we infer
\begin{equation*}
|h(x_0,\xi_1)-h(x_0,\xi_2)|\leq C\eta \Big((|\xi_1|^p+|\xi_2|^p) +\frac{1}{\eta^p}|\xi_1-\xi_2|^p\Big).
\end{equation*}
This estimate implies continuity since the right hand side can be made arbitrarily small on bounded sets by first adjusting $\eta$ and then the difference $|\xi_1-\xi_2|$.
\end{proof}

In the next lemma we prove that $h$ is determined by the behavior of $F(\w)(u_{x_0,\xi},B_{\varrho}(x_0))$ for $\varrho\to 0$.

\begin{lemma}\label{cauchy-born}
Let $\e_n$ and $F(\w)$ be as in Proposition \ref{limitsbvp}. Then there exists a null set $N\subset D$ such that for every $x_0\in D\backslash N$ and all $\xi\in\mathbb{R}^{m\times d}$ it holds that
\begin{equation*}
|B_1|h(x_0,\xi)=\lim_{\varrho\to 0}\varrho^{-d}F(\w)(u_{x_0,\xi},B_{\varrho}(x_0)),
\end{equation*}
where $u_{x_0,\xi}$ denotes the affine function $x\mapsto \xi(x-x_0)$.
\end{lemma}
\begin{proof}
It follows from the definition of $F(\w)$ that
\begin{equation*}
\varrho^{-d}F(\w)(u_{x_0,\xi},B_{\varrho}(x_0))=|B_1|\dashint_{B_{\varrho}(x_0)}h(z,\xi)\,\mathrm{d}z.
\end{equation*}
Due to Lebesgue's differentiation theorem there exists a null set $N_{\xi}\subset D$ such that for any $x_0\in D\backslash N_{\xi}$ we have
\begin{equation}\label{claimxi}
|B_1|h(x_0,\xi)=\lim_{\delta\to 0}\varrho^{-d}F(\w)(u_{x_0,\xi},B_{\varrho}(x_0)).
\end{equation}
It remains to show that the null sets can be chosen independent of $\xi$. To this end, observe that the restriction of $F(\w)$ to $W^{1,p}(D,\R^m)$ is $L^1(D,\R^m)$-lower semicontinuous. Using the defining formula provided by Proposition \ref{limitsbvp} and testing an affine function, by Lemma \ref{bounds} we find that $0\leq h(x_0,\xi)\leq C|\xi|^p$ and Lemma \ref{caratheodory} yields that $h$ is a Carath\'eodory-function. Hence by standard results (see \cite[Theorem 1.13]{Dac})  there exists a null set $N^{\prime}\subset D$ such that for every $x\in D\backslash N^{\prime}$ the function $\xi\mapsto h(x,\xi)$ is quasiconvex on $\mathbb{R}^{m\times d}$ and therefore locally Lipschitz continuous. Moreover the Lipschitz constant is uniformly bounded on compact sets. Let us define the null set $N=N^{\prime}\cup\bigcup_{\xi\in\mathbb{Q}^{m\times d}}N_{\xi}$. Then for any $x_0\in D\backslash N$ it holds that $h(x_0,\xi)=\lim_nh(x_0,\xi_n)$, where $\mathbb{Q}^{m\times d}\ni \xi_n\to \xi$. On the other hand, whenever $\xi_n\to \xi$, then without loss of generality $\sup_n|\xi_n|\leq (|\xi|+1)$ and by the uniform local Lipschitz continuity we obtain 
\begin{equation*}
\dashint_{B_{\varrho}(x_0)}|h(x,\xi_n)-h(x,\xi)|\,\mathrm{d}x\leq C_{\xi}|\xi-\xi_n|.
\end{equation*}
Thus, for $x_0\in D\backslash N$ we can pass to the limit in $n$ on both sides in (\ref{claimxi}) finishing the proof.
\end{proof}
Motivated by the previous lemma we analyze the limit functional $F(\w)(u_{x_0,\xi},B_{\varrho}(x_0))$ by studying recovery sequences for the affine function $u_{x_0,\xi}$. As explained at the beginning of this subsection, we have to localize the analysis to small balls $B_{\varrho}(x_0)$ and consider diagonal sequences $(\e_{\varrho},\varrho)$. Therefore we need the auxiliary result below, which is an immediate consequence of a change of variables. We omit its proof.
\begin{lemma}\label{changeofvar}
Let $\e_n$ and $E(\w)$ be as in Theorem \ref{ACGmain}. For $\varrho>0$ and $x_0\in D$ such that $B_{\varrho}(x_0)\subset D$, define the functional $G_{\e_n,\varrho}(x_0,\w):L^p(B_1,\R^m)\to [0,+\infty]$ to be finite only for $u:\frac{\e_n}{\varrho}(\Lw-\frac{x_0}{\e_n})\to\R^m$ with value
\begin{equation*}
G_{\e_n,\varrho}(x_0,\w)(u)=
\alpha\sum_{\substack{(x,y)\in\mathcal{E}(\w)-\frac{x_0}{\e_n}\\ \frac{\e_n}{\varrho}x,\frac{\e_n}{\varrho}y\in B_1}}\Big(\frac{\e_n}{\varrho}\Big)^d\Big|\frac{u(\frac{\e_n}{\varrho}x)-u(\frac{\e_n}{\varrho}y)}{\e_n\varrho^{-1}}\Big|^p.
\end{equation*}
Then $G_{\e_n,\varrho}(x_0,\w)$ $\Gamma(L^p(B_1,\R^m))$-converges to the functional $E_{\varrho}(x_0,\w):L^p(B_1,\R^m)\to [0,+\infty]$ with domain $W^{1,p}(B_1,\R^m)$, where it is given by
\begin{equation*}
E_{\varrho}(x_0,\w)(u)=\int_{B_1}q(x_0+\varrho y,\nabla u(y))\,\mathrm{d}y.
\end{equation*}
\end{lemma}
In order to use diagonal sequences of the functionals $G_{\e,\varrho}(x_0,\w)$, we first identify the $\Gamma$-limit of the functionals $E_{\varrho}(x_0,\w)(u,A)$ when $\varrho\to 0$. This is already contained in \cite[Lemma 2.1]{GiPo} for scalar problems. Here we provide a short proof in the vectorial case. We essentially follow the lines of \cite[Theorem 5.14]{DM} up to some necessary modifications.
\begin{lemma}[Blow-up by $\Gamma$-convergence]\label{gammablowup}
Let $E_{\varrho}(x_0,\w)$ be a functional as in Lemma \ref{changeofvar}. Then there exists a null set $N\subset D$ such that for all $x_0\in D\backslash N$ it holds that
\begin{equation*}
\Gamma(L^p(B_1,\R^m))\hbox{-}\lim_{\varrho\to 0}E_{\varrho}(x_0,\w)(u)=\begin{cases}
\displaystyle\int_{B_1}q(x_0,\nabla u(y))\,\mathrm{d}y &\mbox{if $u\in W^{1,p}(B_1,\R^m)$,}\\
+\infty &\mbox{otherwise.}
\end{cases}
\end{equation*}
\end{lemma}
\begin{proof}
By coercivity, the $\Gamma$-liminf can be finite only on $W^{1,p}(B_1,\R^m)$. Let $N^{\prime}\subset D$ be the null set where the function $\xi\mapsto q(x,\xi)$ is not quasiconvex or does not fulfill the bound $\frac{1}{C}|\xi|^p-C\leq q(x,\xi)\leq C(|\xi|^p+1)$. We redefine $q(x,\xi)=|\xi|^p$ for all $x\in N^{\prime}$ not changing the functional $E_{\varrho}(x_0,\w)$. Then, according to \cite[Proposition 2.32]{Dac}, there exists a constant $C_q$ such that 
\begin{equation}\label{locallip}
|q(x,\xi)-q(x,\zeta)|\leq C_q(1+|\xi|^{p-1}+|\zeta|^{p-1})|\xi-\zeta|
\end{equation}
for all $x\in D$ and all $\xi,\zeta\in\mathbb{R}^{m\times d}$. Moreover, by Lebesgue's differentiation theorem, for every $\xi\in\R^{m\times d}$ there exists a null set $N_{\xi}$ such that for every $x_0\in D\backslash N_{\xi}$ we have
\begin{equation}\label{lebesgue}
\lim_{\varrho\to 0}\dashint_{B_1}|q(x_0+\varrho y,\xi)-q(x_0,\xi)|\,\mathrm{d}y=\lim_{\varrho\to 0}\dashint_{B_{\varrho}(x_0)}|q(z,\xi)-q(x_0,\xi)|\,\mathrm{d}z=0.
\end{equation}
Set $N=N^{\prime}\cup\bigcup_{\xi\in\mathbb{Q}^{m\times d}}N_{\xi}$, fix $x_0\in D\backslash N$ and consider a sequence $\varrho_j\to 0$. We first prove that, along a suitable subsequence, for each $\xi\in\R^{m\times d}$ the functions $y\mapsto q_j(y,\xi)=q(x_0+\varrho_jy,\xi)$ converge a.e. on $B_1$ to $q(x_0,\xi)$. Indeed, due to (\ref{lebesgue}) the convergence holds true in $L^1(B_1)$ for every $\xi\in\mathbb{Q}^{m\times d}$. Thus we find a ($\xi$-dependent) subsequence such that $q_{j_k}(\cdot,\xi)\to q(x_0,\xi)$ a.e. on $B_1$. Enumerating the $\xi\in\mathbb{Q}^{m\times d}$ we can ensure that the subsequences are nested. Then, by a diagonal argument, we find a common subsequence ${j_k}$ such that $q_{j_k}(\cdot,\xi)\to q(x_0,\xi)$ a.e. on $B_1$ for every $\xi\in\mathbb{Q}^{m\times d}$. Using (\ref{locallip}) we can extend the convergence to all $\xi\in\R^{m\times d}$.

We now prove the $\Gamma$-convergence along this subsequence. By dominated convergence the existence of a recovery sequence is provided through the pointwise limit. In order to prove the lower bound, we can assume that $u_j\to u$ in $L^p(B_1,\R^m)$ and that $u_j$ is bounded in $W^{1,p}(B_1,\R^m)$. By \cite[Lemma 1.2]{FoMuPe} there exists a subsequence (not relabeled) and another sequence $z_j\in W^{1,p}(B_1,\R^m)$ such that $|\nabla z_j|^p$ is equiintegrable and
\begin{equation}\label{inmeasure}
\lim_j\left|\{z_j\neq u_j\text{ or }\nabla z_j\neq\nabla u_j\}\right|=0.
\end{equation}
Note that the above property implies that $z_j$ is also bounded in $W^{1,p}(B_1,\R^m)$. Otherwise, along a subsequence, $w_j=\frac{z_j}{\|z_j\|_{L^p}}$ converges in $W^{1,p}(B_1,\R^m)$ to a non-zero constant, but the sequence $v_j=\frac{u_j}{\|z_j\|_{L^p}}$ converges to $0$ strongly in $W^{1,p}(B_1,\R^m)$. This contradicts (\ref{inmeasure}) which remains valid for $w_j$ and $v_j$. Thus we deduce from (\ref{inmeasure}) that $z_j\rightharpoonup u$ in $W^{1,p}(B_1,\R^m)$. Moreover, equiintegrability of $|\nabla z_j|^p$, the bound $0\leq q(x,\xi)\leq C(|\xi|^p+1)$ and (\ref{inmeasure}) imply that
\begin{equation}\label{betterequi}
\liminf_j E_{\varrho_j}(x_0,\w)(u_j)\geq \liminf_j E_{\varrho_j}(x_0,\w)(z_j).
\end{equation}
Next, for any $\delta>0$ we find again by equiintegrability and the upper bound on $q$ a number $\eta_{\delta}>0$ such that for any measurable set $G\subset B_1$ with $|G|\leq\eta_{\delta}$ it holds that
\begin{equation}\label{equidef}
\sup_j\int_Gq(x_0,\nabla z_j(y))\,\mathrm{d}y<\delta.
\end{equation}
Hence let us choose $t_{\delta}>0$ such that $\sup_j|\{|\nabla z_j|>t_{\delta}\}|\leq\eta_{\delta}$
and set $K_{\delta}=C_q(1+2t_{\delta}^{p-1})$ with $C_q$ given by (\ref{locallip}). By a compactness argument we find $\xi_1,\dots,\xi_N\in\mathbb{R}^{m\times d}$ such that $|\xi_i|< t_{\delta}$ and
\begin{equation*}
\{\xi\in\mathbb{R}^{m\times d}:\;|\xi|\leq t_{\delta}\}\subset \bigcup_{i=1}^N \{\xi\in\mathbb{R}^{m\times d}:\;|\xi-\xi_i|<\frac{\delta}{K_{\delta}}\}.
\end{equation*}
Due to Egorov's theorem there exists a set $G\subset B_1$ with $|G|\leq \eta_{\delta}$ such that the sequences $q_j(y,\xi_i)$ converge uniformly to $q(x_0,\xi_i)$ on $B_1\backslash G$. Hence there exists $j_{\delta}\in\mathbb{N}$ such that for all $j\geq j_{\delta}$, all $i$ and all $y\in B_1\backslash G$ we have $|q_j(y,\xi_i)-q(x_0,\xi_i)|<\delta$. Using (\ref{locallip}) and the triangle inequality we obtain
\begin{equation*}
|q_j(y,\xi)-q(x_0,\xi)|\leq 2C_q(1+2t_{\delta}^{p-1})\min_i|\xi_i-\xi|+\delta\leq 3\delta
\end{equation*}
for all $y\in B_1\backslash G$, all $\xi\in\mathbb{R}^{m\times d}$ with $|\xi|\leq t_{\delta}$ and all $j\geq j_{\delta}$. Since $q(x,\xi)\geq 0$, we infer that for those $j$
\begin{align*}
E_{\varrho_j}(x_0,\w)(z_j)&\geq \int_{\{|\nabla z_j|\leq t_{\delta}\}\backslash G}q_j(y,\nabla z_j(y))\,\mathrm{d}y\geq \int_{\{|\nabla z_j|\leq t_{\delta}\}\backslash G}q(x_0,\nabla z_j(y))\,\mathrm{d}y-3|B_1|\delta
\\
&\geq\int_{B_1}q(x_0,\nabla z_j(y))\,\mathrm{d}y-(3|B_1|+2)\delta,
\end{align*}
where we used (\ref{equidef}) and the definition of $t_{\delta}$. By quasiconvexity and the growth conditions on $q$ the last term is lower semicontinuous with respect to weak convergence in $W^{1,p}(B_1,\R^m)$. Hence (\ref{betterequi}) implies
\begin{equation*}
\liminf_jE_{\varrho_j}(x_0,\w)(u_j)\geq \int_{B_1}q(x_0,\nabla u(y))\,\mathrm{d}y-C\delta.
\end{equation*}
By the arbitrariness of $\delta$ we obtain the lower bound. Since the limit functional is independent of any subsequence, we established the full $\Gamma$-convergence result.
\end{proof}
\begin{remark}\label{alsopointwise}
Taking the same null set $N\subset D$ as in Lemma \ref{gammablowup}, the convergence (\ref{lebesgue}) holds for all $x_0\in D\backslash N$ and all $\xi\in\R^d$ again by (\ref{locallip}).	
\end{remark}

Now we are almost in a position to use a diagonal sequence to recover the function $q(x_0,\xi)$. However, in general there exists no metric characterizing the $\Gamma$-convergence when equicoercivity fails, so that diagonal arguments are not always available. Therefore we provide an explicit construction similar to \cite{DaMo1} in the appendix. With this metric at hand we can derive the following result.
\begin{lemma}\label{diagonal}
Under the assumptions of Lemma \ref{changeofvar}, let $x_0\in D\backslash N$ where $N$ is given by Lemma \ref{gammablowup}. For every $\varrho$ there exists $\e(\varrho)>0$ such that whenever we chose $\e_{n(\varrho)}\leq \e(\varrho)$ it holds that
\begin{equation*}
\Gamma(L^p(B_1,\R^m))\hbox{-}\lim_{\varrho\to 0}G_{\e_{n(\varrho)},\varrho}(x_0,\w)(u)=\begin{cases}
\displaystyle\int_{B_1}q(x_0,\nabla u(y))\,\mathrm{d}y &\mbox{if $u\in W^{1,p}(B_1,\R^m)$,}\\
+\infty &\mbox{otherwise.}
\end{cases}.
\end{equation*}
\end{lemma}
\begin{proof}
This result is an immediate consequence of Lemma \ref{metric} and Remark \ref{continuumcase}, which allow to combine Lemma \ref{gammablowup} and a diagonal argument with respect to the metric $\mathfrak{d}$ constructed in the appendix.
\end{proof}
We now come to the second part and derive Lipschitz-estimates from bounds on the gradient's discrete maximal function. For the remainder of this subsection it will be convenient to view a stochastic lattice also as an undirected graph $G=(\Lw,\mathbb{B}(\w))$ with edges $\mathbb{B}(\w)=\mathcal{E}(\w)\cup \{(y,x)\in\Lw^2:\;(x,y)\in\mathcal{E}(\w)\}$. We say that $P$ is a path of length $n$ if $P=\{x_0,\dots,x_n\}$ with $(x_{i-1},x_{i})\in \mathbb{B}(\w)$ for all $i=1,\dots,n$. Given $x,y\in \Lw$, we define the graph distance as
\begin{equation*}
\mathrm{d}_{G}(x,y)=\inf\{\text{length of a path }P\text{ such that }x,y\in P\}.
\end{equation*}
We denote by $B_{G}(x,\eta)=\{y\in \Lw:\;\mathrm{d}_{G}(x,y)\leq\eta\}$ the closed ball with radius $\eta$ with respect to the graph metric. In the next lemma we establish a doubling property of the counting measure and a weak Poincar\'e inequality that allow us to relate Lipschitz continuity to discrete maximal functions of gradients. Given $\e>0$ and $u:\e \Lw\to\R^m$ we define the length of its edge gradient $|\nabla_{e,\e}u|:\e \Lw\to\R$ by
\begin{equation}\label{eq:lengthgrad}
|\nabla_{e,\e} u|(\e x)=\sum_{(x,y)\in \mathbb{B}(\w)}\Big|\frac{u(\e x)-u(\e y)}{\e}\Big|.
\end{equation}
Note that in the above definition we take into account the undirected edges $e\in \mathbb{B}(\w)$ (hence the subscript).

\begin{lemma}\label{graphprop}
	Let $G=(\Lw,\mathbb{B}(\w))$ be a graph associated to an admissible stochastic lattice. Then there exists a constant $C=C(r,R,M)>0$ such that for  all $x\in \Lw$, $\eta>0$ and $u:\Lw\to\R^m$ it holds that
	\begin{itemize}
		\item [(i)] $\#B_{G}(x,2\eta)\leq C \#B_{G}(x,\eta)$,
		\item [(ii)]
		\begin{equation*}
		\sum_{y\in B_{G}(x,\eta)}|u(y)-u_{B_{G}(x,\eta)}|\leq C\eta\sum_{y\in B_{G}(x,C\eta)}|\nabla_{e,1} u|(y),
		\end{equation*}
		where the average $u_{B_G(x,\eta)}$ is defined as
		\begin{equation*}
		u_{B_{G}(x,\eta)}=\frac{1}{\#B_{G}(x,\eta)}\sum_{y\in B_{G}(x,\eta)}u(y).
		\end{equation*}
	\end{itemize}
\end{lemma}
\begin{remark}
	Given $\eta>0$ and a scalar function $v:\e\Lw\to\R$ we define the maximal function $\mathcal{M}^{\e}_{\eta}v:\e\Lw\to\R$ as
	\begin{equation}\label{maximalfunc}
	\mathcal{M}^{\e}_{\eta}v(\e x)=\sup_{0<s<\eta}\Big(\frac{1}{\#B_{G}(x,\frac{s}{\e})}\sum_{y\in B_{G}(x,\frac{s}{\e})}|v(\e y)|\Big).
	\end{equation}
	Then, assuming the doubling property, the Poincar\'e inequality is equivalent to the estimate
	\begin{equation}\label{maximal-lip}
	\Big|\frac{u(\e x)-u(\e y)}{\e}\Big|\leq \overline{C}\, \mathrm{d}_{G}(x,y)\Big(\mathcal{M}^{\e}_{\overline{C}\e \mathrm{d}_{G}(x,y)}|\nabla_{e,\e} u|(\e x)+\mathcal{M}^{\e}_{\overline{C}\e \mathrm{d}_{G}(x,y)}|\nabla_{e,\e} u|(\e y)\Big),
	\end{equation}
	where $\overline{C}$ is a constant independent of $u:\e\Lw\to\R^m$ and $x,y\in\Lw$. For $\e,m=1$ this fact can be found in a much more general context in \cite[Lemma 5.15]{HeiKo}. For $\e>0$ and $m=1$ it follows by applying the inequality to $v:\Lw\to\R$ defined as $v(x)=\e^{-1}u(\e x)$ upon noticing that $\mathcal{M}^1_{\eta}|\nabla_{e,1}v|(x)=\mathcal{M}^{\e}_{\e\eta}|\nabla_{e,\e}u|(\e x)$. In particular the constant $\overline{C}$ is independent of $\e$. For $m\geq 2$ the inequality remains true arguing for each component and increasing the constant $\overline{C}$ by at most a factor of $m$.
\end{remark}
\begin{proof}[Proof of Lemma \ref{graphprop}]
	(i): We may assume that $\eta\geq\frac{1}{2}$. Our aim is to compare the graph-metric with the Euclidean one. Given $x,y\in \Lw$ and an optimal path $P=\{x_0=x,x_1,\dots,x_n=y\}$, by (\ref{finiterange}) we have
	\begin{equation}\label{ubdist}
	|x-y|\leq\sum_{i=1}^{\mathrm{d}_G(x,y)}|x_{i-1}-x_i| < M\,\mathrm{d}_G(x,y).
	\end{equation}
	On the other hand, it follows from \eqref{nncontained} and Lemma \ref{l.paths} that
	\begin{equation}\label{lbdist}
	\mathrm{d}_G(x,y)\leq  C_{r,R} |x-y|.
	\end{equation}
	Using again Remark \ref{voronoi} and (\ref{ubdist}), for $\eta\geq\frac{1}{2}$ we deduce that 
	\begin{equation*}
	\#B_G(x,2\eta)\leq \#(\Lw\cap B_{2M\eta}(x))\leq |B_{\frac{r}{2}}|^{-1}|B_{4M\eta}(x)|, 
	\end{equation*}
	while due to \eqref{lbdist} for any $\varrho>0$ it holds that
	\begin{equation}\label{conttograph}
	|B_{\varrho}(x)|\leq |B_R(0)|\,\#(\Lw\cap B_{2\varrho}(x))\leq |B_R(0)|\,\#B_G(x,2C_{r,R}\varrho).
	\end{equation}
	The claim now follows from the scaling properties of the Lebesgue measure by choosing $\varrho=\frac{\eta}{2C_{r,R}}$.
	\\
	(ii): We can assume that $\eta\geq 1$. Due to the triangle inequality we have
	\begin{equation}\label{triangle1}
	\sum_{y\in B_G(x,\eta)}|u(y)-u_{B_G(x,\eta)}|\leq \frac{1}{\#B_G(x,\eta)}\sum_{y,z\in B_G(x,\eta)}|u(y)-u(z)|.
	\end{equation}
	Fix $y,z\in B_G(x,\eta)$ and consider the path $P(y,z)=\{x_0=y,x_1,\dots,x_n=z\}$ given by Lemma \ref{l.paths}. Then
	\begin{equation}\label{triangle2}
	|u(y)-u(z)|\leq \sum_{i=1}^{n}|u(x_{i-1})-u(x_i)|.
	\end{equation} 
	The triangle inequality, Lemma \ref{l.paths} and (\ref{ubdist}) imply that for all $x_i\in P(y,z)$
	\begin{align*}
	\mathrm{d}_G(x_i,x)&\leq \mathrm{d}_G(x_i,y)+\mathrm{d}_G(y,x)\leq\#P(y,z)+\eta\leq C_{r,R}|y-z|+\eta\\
	&\leq C_{r,R}(|y-x|+|x-z|)+\eta
	\leq (2MC_{r,R}+1)\eta.
	\end{align*}
	Setting $C=2MC_{r,R}+1$, we deduce that $P(y,z)\subset B_G(x,C\eta)$. Conversely, given an edge $(x^{\prime},y^{\prime})\in \mathbb{B}(\w)$ with $x^{\prime},y^{\prime}\in B_G(x,C\eta)$, we denote by
	\begin{equation*}
	N(x^{\prime},y^{\prime})=\#\{(y,z)\in B_G(x,\eta)\times B_G(x,\eta) :\;x^{\prime},y^{\prime}\subset P(y,z)\}
	\end{equation*}
	the number of pairs $(y,z)$ such that this edge is contained in the path given by Lemma \ref{l.paths}. As a consequence of (\ref{conttograph}), (\ref{triangle1}), (\ref{triangle2}) and (\ref{nncontained}) we have
	\begin{align}\label{poincare1}
	\sum_{y\in B_G(x,\eta)}|u(y)-u_{B_G(x,\eta)}|&\leq\frac{1}{\#B_G(x,\eta)}\sum_{\substack{(x^{\prime},y^{\prime})\in \mathbb{B}(\w) \\ x^{\prime},y^{\prime}\in B_G(x,C\eta)}}N(x^{\prime},y^{\prime})|u(x^{\prime})-u(y^{\prime})|\nonumber
	\\
	&\leq C\sup_{\substack{(x^{\prime},y^{\prime})\in \mathbb{B}(\w)\\ x^{\prime},y^{\prime}\in B_G(x,C\eta)}}\frac{N(x^{\prime},y^{\prime})}{\eta^d}\sum_{y\in B_G(x,C\eta)}|\nabla_{e,1} u|(y).
	\end{align}
	It remains to prove a suitable upper bound for $N(x^{\prime},y^{\prime})$. Since $x^{\prime},y^{\prime}\in B_G(x,C\eta)$, it follows by (\ref{ubdist}) that $x^{\prime},y^{\prime}\in B_{MC\eta}(x)$ and therefore
	\begin{equation*}
	N(x^{\prime},y^{\prime})\leq \#\{(y,z)\in B_{2MC\eta}(x^{\prime})\times B_{2MC\eta}(x^{\prime}):\;x^{\prime},y^{\prime}\subset P(y,z)\},
	\end{equation*}
	where we used that $C\geq 1$. Consider then the boundary-like set $\Gamma=\{b\in\Lw:\;\mathcal{C}(b)\cap \partial B_{2MC\eta}(x^{\prime})\neq\emptyset\}$ and for each $b\in\Gamma$ let us define the cylinder-type set
	\begin{equation*}
	Z(b,x^{\prime})=\{a+\lambda (x^{\prime}-b):\;\lambda\in [0,2],\,a\in B_{6R}(b)\}.
	\end{equation*}
	For the moment fix any $y,z\in B_{2MC\eta}(x^{\prime})$ such that $x^{\prime},y^{\prime}\in P(y,z)$. By Lemma \ref{l.paths} there exists a point $x_*\in {\rm co}(y,z)$ such that $|x_*-x^{\prime}|\leq 2R$. Without loss of generality we assume that $|y-x_*|\leq |z-x_*|$ and denote by $p$ the unique point $p\in\{x^{\prime}+\lambda (z-x^{\prime}):\;\lambda\geq 0\}\cap \partial B_{2MC\eta}(x^{\prime})$. Then we find $b\in\Gamma$ with $|p-b|\leq R$. We argue that $y,z\in Z(b,x^{\prime})$. For $z$ this follows upon writing $z=p+\lambda_z(x^{\prime}-p)$ with $\lambda_z\in [0,1]$ and choosing $a=\lambda_zb+(1-\lambda_z)p$ and $\lambda=\lambda_z$ in the definition of $Z(b,x^{\prime})$. Regarding $y$, recall that we assume $|y-x_*|\leq |z-x_*|$, so that for some $\lambda_y\in [1,2]$ we can write
	\begin{equation*}
	y=z+\lambda_y(x_*-z)=p+\lambda_z(x^{\prime}-p)+\lambda_y(x_*-p-\lambda_z(x^{\prime}-p)).
	\end{equation*}
	Setting $\lambda=\lambda_z+\lambda_y-\lambda_z\lambda_y\in [0,2]$ and using the ansatz $a=b+\xi$ in the definition of $Z(b,x^{\prime})$, we find
	\begin{equation*}
	\xi=(p-b)(1-\lambda_y)+\lambda_y(x_*-x^{\prime}),
	\end{equation*}
	so that $|\xi|\leq 5R$ and therefore $y\in Z(b,x^{\prime})$ as well. Thus we have proven that
	\begin{equation}\label{sliced}
	N(x^{\prime},y^{\prime})\leq 2(\#\Gamma)\Big(\sup_{b\in\Gamma}\#(\Lw\cap Z(b,x^{\prime}))\Big)^2.
	\end{equation}
	Now we use again Remark \ref{voronoi} combined with a covering argument and the fact that $\eta\geq 1$ to find
	\begin{equation*}
	\#\Gamma\leq |B_{\frac{r}{2}}(0)|^{-1}\Big(|B_{2MC\eta+R}(x^{\prime})|-|B_{2MC\eta-R}(x^{\prime})|\Big)\leq C \Big((2MC\eta+R)^d-(2MC\eta-R)^d\Big)\leq C\eta^{d-1}.
	\end{equation*}
	With the same technique we obtain the bound
	\begin{equation*}
	\#(\Lw\cap Z(b,x^{\prime}))\leq 2|B_{\frac{r}{2}}(0)|^{-1}(14R)^{d-1}(|x^{\prime}-b|+7R)\leq C\eta.
	\end{equation*}
	Combining the last two estimates with (\ref{poincare1}) and (\ref{sliced}) we conclude the proof.
\end{proof}
The next lemma is a discrete analogue of \cite[Lemma 1.2]{FoMuPe}.
\begin{lemma}\label{l.fomupediscrete}
Let $G=(\Lw,\mathbb{B}(\w))$ be admissible, let $x_0\in\R^d$, $\lambda>0$ and set $\bar{k}=3+6\overline{C}C_{r,R}M$ with $\overline{C}$ and $C_{r,R}$ given by (\ref{maximal-lip}) and Lemma \ref{l.paths}, respectively. If $\e\downarrow 0$ and $u_{\e}:\e\Lw\to\R^m$ is a sequence such that
\begin{equation*}
\sup_{\e>0}\sum_{\e x\in\e\Lw\cap B_{\bar{k}\lambda}(x_0)}\e^d|\nabla_{e,\e}u_{\e}|^p(\e x)\leq C,
\end{equation*}	
then there exists a subsequence $\e_j$ and a sequence $w_j:\e_{j}\Lw\to\R^m$ such that $|\nabla_{e,\e_{j}}w_{j}|^p$ is equiintegrable on $B_{2\lambda}(x_0)$ and moreover
\begin{equation*}
\lim_j\e_{j}^d\#\left\{\e_{j}x\in\e_{j}\Lw\cap B_{2\lambda}(x_0):\,u_{\e_j}\not\equiv w_j\text{ on }\e_jB_G(x,1)\right\}=0.
\end{equation*}
\end{lemma}
\begin{proof}
Define a function $V_{\e}:\e\Lw\to [0,+\infty)$ by
\begin{equation*}
V_{\e}(\e x)=
\begin{cases}
|\nabla_{e,\e}u_{\e}|(\e x) &\mbox{if $\e x\in B_{\bar{k}\lambda}(x_0)$,}
\\
0 &\mbox{otherwise.}
\end{cases}
\end{equation*}
By piecewise constant interpolation on the Voronoi cells $\{\e\mathcal{C}(x)\}$ of $\e\Lw$ we can view $V_{\e}$ as an element of $L^p(\R^d)$. From our assumption we deduce that
\begin{equation*}
\|V_{\e}\|^p_{L^p(\R^d)}\leq C \sum_{\e x\in\e\Lw\cap B_{\bar{k}\lambda}(x_0)}\e^d|\nabla_{e,\e}u_{\e}|^p(\e x)\leq C.
\end{equation*}
Hence the sequence $V_{\e}$ is bounded in $L^p(\R^d)$.  We claim that $\mathcal{M}^{\e}_{\infty}V_{\e}$ is also bounded in $L^p(\R^d)$, where $\mathcal{M}^{\e}_{\infty}V_{\e}$ is the discrete maximal function of $V_{\e}$ defined in \eqref{maximalfunc}. To this end, we show that it can be pointwise controlled by the standard Hardy-Littlewood maximal function. Indeed, by (\ref{ubdist}), (\ref{conttograph}) and Remark \ref{voronoi}, for any function $v:\e\Lw\to\mathbb{R}$  we can estimate
\begin{align*}
|\mathcal{M}^{\e}_{\infty}v(\e x)|&\leq\sup_{\eta>\frac{\e}{2}}\Big(\frac{1}{\e^d\#B_{G}(x,\frac{\eta}{\e})}\sum_{y\in B_{G}(x,\frac{\eta}{\e})}\e^d|v(\e y)|\Big)\leq \sup_{\eta>\frac{\e}{2}}\Big(\frac{C}{|B_{\eta}(x)|}\sum_{y\in B_{M\frac{\eta}{\e}}(x)}\e^d|v(\e y)|\Big)
\\
&\leq \sup_{\eta>\frac{\e}{2}}\Big(\frac{C}{|B_{\eta}(\e x)|}\int_{B_{3M\eta}(\e x)}|v(z)|\,\mathrm{d}z\Big)\leq C\sup_{\eta>0}\Big(\frac{1}{|B_{\eta}(\e x)|}\int_{B_{\eta}(\e x)}|v(z)|\,\mathrm{d}z\Big).
\end{align*}
Thus boundedness of the Hardy-Littlewood maximal function operator on $L^p(\R^d)$ (see \cite[Theorem 2.91]{FoLe}) implies that 
\begin{equation}\label{eq:lpmaxbound}
\|\mathcal{M}_{\infty}^{\e}V_{\e}\|^p_{L^p(\R^d)}\leq C \|V_{\e}\|^p_{L^p(\R^d)}\leq C.
\end{equation} 
For $l>0$ we introduce the scalar truncation operator $\delta_l:\R\to\R$ given by $\delta_l(x)=\min\{l,|x|\}\tfrac{x}{|x|}$ (with $\delta_l(0)=0$). Applying \cite[Lemma 2.31]{FoLe}, we know that there exists a subsequence $\e_j$ and an increasing sequence of positive integers $l_j^p\to +\infty$ such that the sequence $\delta_{l_j^p}\circ(|\mathcal{M}_{\infty}^{\e_{j}}V_{\e_j}|^p)=\big(\delta_{l_j}\circ|\mathcal{M}_{\infty}^{\e_j}V_{\e_j}|\big)^p$ is equiintegrable on $\R^d$. Define the sets
\begin{equation*}
\begin{split}
&R^{\prime}_j:=\{x\in\Lw:\; \mathcal{M}_{\infty}^{\e_{j}}V_{\e_j}(\e_{j}x)\leq l_j\},
\\
&R_j:=\bigcup_{x\in R_j^{\prime}}B_G(x,1).
\end{split}
\end{equation*}
Note that $R_j^{\prime}\subset R_j$. By Remark \ref{voronoi} and \eqref{eq:lpmaxbound} we have
\begin{equation}\label{eq:mod2small}
\e_{j}^d\#(\Lw\backslash R_j)\leq \e_j^d\#(\Lw\backslash R_j^{\prime})\leq \frac{C}{l_j^p}\int_{\R^d}|\mathcal{M}^{\e_j}_{\infty}V_{\e_j}(z)|^p\,\mathrm{d}z\leq \frac{C}{l_j^p}.
\end{equation}
Next we bound $\mathcal{M}_{\infty}^{\e_j}V_{\e_j}$ on the set $R_j$. Given $y\in B_G(x,1)$ with $x\in R^{\prime}_j$, for $\eta\geq 2\e_j$ it holds that
\begin{equation*}
B_{G}\left(x,\frac{\eta}{2\e_j}\right)\subset B_{G}\left(x,\frac{\eta}{\e_j}-1\right)\subset B_{G}\left(y,\frac{\eta}{\e_j}\right)\subset B_{G}\left(x,\frac{\eta}{\e_j}+1\right)\subset B_{G}\left(x,\frac{2\eta}{\e_j}\right).
\end{equation*}
Hence applying twice the doubling property proven in Lemma \ref{graphprop} (i) we conclude that
\begin{equation*}
\sup_{\eta\geq 2\e_j}\Big(\frac{1}{\#B_{G}(y,\frac{\eta}{\e_j})}\sum_{z\in B_{G}(y,\frac{\eta}{\e_j})}|V_{\e_j}(\e_j z)|\Big)\leq C \Big(\mathcal{M}_{\infty}^{\e_j}V_{\e_j}(\e_j x)\Big)\leq Cl_j.
\end{equation*}
If $\eta< 2\e_j$, then by Remark \ref{voronoi} and (\ref{finiterange}) we still have the inequality and inclusion
\begin{equation*}
\frac{1}{C}\#B_{G}(x,2)\leq\#B_{G}\left(y,\frac{\eta}{\e_j}\right)
\quad\text{ and }\quad
B_{G}\left(y,\frac{\eta}{\e_j}\right)\subset B_{G}(x,2),
\end{equation*}
so that for $y\in R_j$ we obtain the estimate
\begin{equation}\label{eq:maximalneighbours}
\mathcal{M}_{\infty}^{\e_j}V_{\e_j}(\e_j y)\leq Cl_j.
\end{equation}
Next we want to use (\ref{maximal-lip}) and Kirszbraun's extension theorem on $\e_j\Lw\cap B_{3\lambda}(x_0)$. To this end, note that by \eqref{ubdist} and \eqref{lbdist}, for all $x,y\in R_j\cap \e_j^{-1}B_{3\lambda}(x_0)$ and $z\in B_{G}(y,\overline{C}\mathrm{d}_{G}(x,y))$, it holds that 
\begin{equation*}
|\e_j z-x_0|< 3\lambda+\e_j|z-y|\leq 3\lambda+\e_j M\overline{C}\mathrm{d}_{G}(x,y)\leq 3\lambda+6\lambda M\overline{C}C_{r,R}= \bar{k}\lambda.
\end{equation*}
Since $V_{\e_j}$ agrees with $|\nabla_{e,\e_j}u_{\e_j}|$ on $\e_j\Lw\cap B_{\bar{k}\lambda}(x_0)$, it follows from \eqref{eq:maximalneighbours} that for $\e_jy\in \e_jR_j\cap B_{3\lambda}(x_0)$ we have the estimate
\begin{equation*}
\mathcal{M}_{\overline{C}\e_j \mathrm{d}_{G}(x,y)}^{\e_j}|\nabla_{e,\e_j}u_{\e_j}|(\e_j y)=\mathcal{M}_{\overline{C}\e_j \mathrm{d}_{G}(x,y)}^{\e_j}V_{\e_j}(\e_j y)\leq \mathcal{M}_{\infty}^{\e_j}V_{\e_j}(\e_j y)\leq Cl_j.
\end{equation*} 
Combining (\ref{maximal-lip}) and Kirszbraun's extension theorem we find Lipschitz-functions $w_j:\e_j\Lw\to\R^m$ that agree with $u_{\e_j}$ on the set $\e_jR_j\cap B_{3\lambda}(x_0)$ and with Lipschitz constant bounded by $Cl_j$ (up to enlarging $C$ due to \eqref{lbdist}). We claim that $|\nabla_{e,\e_j}w_j|^p$ is equiintegrable on $B_{2\lambda}(x_0)$. To verify this assertion, we observe that for $j=j(\lambda)$ large enough, by the definition of $R_j$ we have that $|\nabla_{e,\e_j} w_j|=|\nabla_{e,\e_j}u_{\e_j}|$ on $\e_jR_j^{\prime}\cap B_{5\lambda/2}(x_0)$. Hence for $\e_jx\in \e_jR_j^{\prime}\cap B_{5\lambda/2}(x_0)$ we deduce
\begin{equation*}
|\nabla_{e,\e_j}w_j|^p(\e_j x)=|\nabla_{e,\e_j}u_{\e_j}|^p(\e_j x)=|V_{\e_j}(\e_j x)|^p\leq |\mathcal{M}^{\e_j}_{\infty}V_{\e_j}(\e_jx)|^p=\big(\delta_{l_j}\circ|\mathcal{M}_{\infty}^{\e_j}V_{\e_j}|\big)^p(\e_j x),
\end{equation*}
while on $\e_j\left(\Lw\backslash R^{\prime}_j\right)$ the bound on the Lipschitz constant implies
\begin{equation*}
|\nabla_{e,\e_j}w_j|^p(\e_j x)
\leq Cl_j^p=C\big(\delta_{l_j}\circ|\mathcal{M}_{\infty}^{\e_j}V_{\e_j}|\big)^p(\e_j x).
\end{equation*}
Hence equiintegrability on $B_{2\lambda}(x_0)$ transfers from $\big(\delta_{l_j}\circ|\mathcal{M}_{\infty}^{\e_j}V_{\e_j}|\big)^p$ to $|\nabla_{e,\e_j}w_j|^p$. Finally, note that
\begin{equation*}
\left\{\e_{j}x\in\e_{j}\Lw\cap B_{2\lambda}(x_0):\,u_{\e_j}\not\equiv w_j\text{ on }\e_jB_G(x,1)\right\}\subset \e_j\left(\Lw\backslash R^{\prime}_j\right),
\end{equation*}
so that the second claim of the lemma follows from \eqref{eq:mod2small}.
\end{proof}

Now we are finally in a position to compare the two discrete functionals $F_{\e}(\w)$ and $E_{\e}(\w)$.
\begin{proposition}[Separation of bulk effects]\label{p.gradientpartsequal}
Let $\e_n$ and $F(\w)$ be as in Proposition \ref{limitsbvp}. Then for a.e. $x_0\in D$ and every $\xi\in\R^{m\times d}$ it holds that
\begin{equation*}
|B_1|h(x_0,\xi)=\lim_{\varrho\to 0}\varrho^{-d}F(\w)(u_{x_0,\xi},B_{\varrho}(x_0))=|B_1|q(x_0,\xi),
\end{equation*}
where $q$ is an (equivalent) integrand given by the $\Gamma$-limit of $E_{\e_n}(\w)(\cdot,D)$, which in particular exists.	
\end{proposition}
\begin{proof}
The first equality follows from Lemma \ref{cauchy-born}, so we turn to the proof of the second one. We apply Theorem \ref{ACGmain}, so that, passing to a further subsequence (not relabeled), we may assume that $E_{\e_n}(\w)$ $\Gamma$-converges to some integral functional $E(\w)$ with density $q(x,\xi)$. Let us fix $x_0\in D$ satisfying the first equality and such that Lemmata \ref{gammablowup} and \ref{diagonal} hold. Choose $0<\varrho_0<1$ such that $B_{\varrho_0}(x_0)\subset D$. Lemma \ref{trunc} yields a sequence $u_{\e_n}\in\mathcal{PC}_{\e_n}^{\w}$ that is equibounded in $L^{\infty}(\R^d,\R^m)$, $u_{\e_n}\to u_{x_0,\xi}$ in $L^p(D,\R^m)$ and
\begin{equation*}
\lim_{\e_n\to 0}F_{\e_n}(\w)(u_{\e_n},B_{\varrho_0}(x_0))=F(\w)(u_{x_0,\xi},B_{\varrho_0}(x_0)).
\end{equation*}
Now consider $0<\varrho<\varrho_0$. By discrete superadditivity and $\Gamma$-convergence on $\Ard$ we find that
\begin{align}\label{universalrec}
\limsup_{\e_n\to 0}F_{\e_n}(\w)(u_{\e_n},B_{\varrho}(x_0))&\leq \lim_{\e_n\to 0}F_{\e_n}(\w)(u_{\e_n},B_{\varrho_0}(x_0))-\liminf_{\e_n\to 0}F_{\e_n}(\w)(u_{\e_n},B_{\varrho_0}(x_0)\backslash\overline{B_{\varrho}(x_0)})\nonumber
\\
&\leq F(\w)(u_{x_0,\xi},B_{\varrho_0}(x_0))-F(\w)(u_{x_0,\xi},B_{\varrho_0}(x_0)\backslash \overline{B_{\varrho}(x_0)})\nonumber
\\
&=F(\w)(u_{x_0,\xi},B_{\varrho}(x_0)),
\end{align}
where in the last equality we used that the limit energy of $u_{x_0,\xi}$ does not concentrate on $\partial B_{\varrho}(x_0)$. (\ref{universalrec}) shows that $u_{\e_n}$ is also a recovery sequence on each $B_{\varrho}(x_0)$ for $0<\varrho<\varrho_0$. Next we introduce a constant whose value will become clear later in the proof (cf. the constant $\bar{k}$ in Lemma \ref{l.fomupediscrete}). Choose $k$ satisfying
\begin{equation*}
3+6\,\overline{C}\,C_{r,R}M+|\xi|\leq k,
\end{equation*}
where $\overline{C}$ and $C_{r,R}$ are given by (\ref{maximal-lip}) and Lemma \ref{l.paths}, respectively.
Since $|u_{x_0,\xi}|\leq |\xi|\varrho$ on $B_{\varrho}(x_0)$, Lemma \ref{trunc} implies that the truncated functions $T_{k\varrho}u_{\e_n}$ also yield a recovery sequence on $B_{\varrho}(x_0)$. Now consider a sequence $\varrho_j\to 0$. For any $\varrho=\varrho_j\in (0,(3Mk^2)^{-1}\varrho_0)$ we choose $\e_{\varrho}=\e_{n(\varrho)}\leq\min\{\varrho^{\frac{p}{p-1}},\varrho^2\}$ non-decreasing in $\varrho$, satisfying Lemma \ref{diagonal} and such that
\begin{equation}\label{choicerho}
\begin{split}
F_{\e_{\varrho}}(\w)(T_{k\varrho}u_{\e_{\varrho}},B_{3Mk^2\varrho}(x_0))&\leq C|\xi|^p\varrho^{d},
\\
\dashint_{B_{\varrho}(x_0)}|T_{k\varrho}u_{\e_{\varrho}}-u_{x_0,\xi}|^p\,\mathrm{d}x&\leq \varrho^{p+1}.
\end{split}
\end{equation}
Note that the first estimate is realizable due to Lemma \ref{trunc} and the fact that $u_{\e_n}$ is a recovery sequence also on $B_{3Mk^2\varrho}(x_0)$. 
Finally, we can also require that
\begin{equation}\label{onesequence}
\lim_{\varrho\to 0}\varrho^{-d}F(\w)(u_{x_0,\xi},B_{\varrho}(x_0))=\lim_{\varrho\to 0}\varrho^{-d}F_{\e_{\varrho}}(\w)(T_{k\varrho}u_{\e_{\varrho}},B_{\varrho}(x_0)).
\end{equation}
Our analysis relies on several modifications of the sequence $T_{k\varrho}u_{\e_{\varrho}}$ and a rescaling to $B_1$.

\medskip
\textbf{Step 1} Construction of Lipschitz competitors
\\
The following argument is well-known for continuum functionals and we adapt it carefully to the discrete setting. Let us set $v_{\varrho}:\Lw\to\R^m$ as $v_{\varrho}(x)=\e_{\varrho}^{-1}T_{k\varrho}u_{\e_{\varrho}}(\e_{\varrho} x)$. Then, for given $\Lambda>0$, we define the sets
\begin{equation*}
\begin{split}
&R^{\Lambda}_{\varrho}:=\{x\in\Lw\cap \e_{\varrho}^{-1}B_{k\varrho}(x_0):\;\mathcal{M}^1_{k^2\varrho\e_{\varrho}^{-1}}|\nabla_{e,1} v_{\varrho}|(x)\leq\Lambda\},
\\
&S^{\Lambda}_{\varrho}:=\{x\in\Lw:\;|\nabla_{e,1}v_{\varrho}|(x)\geq\frac{\Lambda}{2}\},
\end{split}
\end{equation*}
where $\mathcal{M}^1_{\eta}$ denotes the discrete maximal function operator defined in (\ref{maximalfunc}) and the norm of the discrete gradient is given by the formula \eqref{eq:lengthgrad}. First we estimate the cardinality of $(\Lw\cap \e_{\varrho}^{-1}B_{k\varrho}(x_0))\backslash R^{\Lambda}_{\varrho}$. To this end, note that for every $x\in (\Lw\cap\e_{\varrho}^{-1}B_{k\varrho}(x_0))\backslash R^{\Lambda}_{\varrho}$ there exists a number $0<\eta_x\leq k^2\varrho\e_{\varrho}^{-1}$ such that
\begin{equation*}
\Lambda\, \#B_G(x,\eta_x)<\sum_{y\in B_G(x,\eta_x)}|\nabla_{e,1} v_{\varrho}|(y)=:|\nabla_1 v_{\varrho}|(B_G(x,\eta_x)).
\end{equation*}
Applying Vitali's covering lemma on separable metric spaces we find a (finite) collection of disjoint balls $B_G(x_i,\eta_i)$ with $x_i\in (\Lw\cap\e_{\varrho}^{-1}B_{k\varrho}(x_0))\backslash R^{\Lambda}_{\varrho}$ satisfying the above inequality and 
\begin{equation*}
(\Lw\cap\e_{\varrho}^{-1}B_{k\varrho}(x_0))\backslash R^{\Lambda}_{\varrho}\subset\bigcup_i B_G(x_i,5\eta_i).
\end{equation*}
Since the balls are disjoint we conclude that
\begin{equation*}
\Lambda\,\#\Big(\bigcup_i B_G(x_i,\eta_i)\Big)< |\nabla_1 v_{\varrho}|\Big( \bigcup_i B_G(x_i,\eta_i)\Big)\leq|\nabla_1 v_{\varrho}|\Big( \bigcup_i B_G(x_i,\eta_i)\cap S^{\Lambda}_{\varrho}\Big)+\frac{\Lambda}{2}\,\#\Big(\bigcup_i B_G(x_i,\eta_i)\Big),
\end{equation*}
where we used the definition of $S^{\Lambda}_{\varrho}$ in the last estimate. Rearranging terms we obtain
\begin{equation}\label{cardbound1}
\#\Big(\bigcup_i B_G(x_i,\eta_i)\Big)< \frac{2}{\Lambda}|\nabla_1 v_{\varrho}|\Big( \bigcup_i B_G(x_i,\eta_i)\cap S^{\Lambda}_{\varrho}\Big).
\end{equation}
To reduce notation, for $x\in\Lw$ we set $N_x=\{y\in\Lw:\,(x,y)\in \mathbb{B}(\w)\}$. Moreover define
\begin{equation*}
\mathcal{J}_{\varrho}=\{x\in\Lw:\;\e_{\varrho}|\nabla_{e,1}v_{\varrho}|^p(x)\geq 1\}.
\end{equation*}
Note that for any $x\in\mathcal{J}_{\varrho}$ there exists a point $y\in N_x$ with $\e_{\varrho}|v_{\varrho}(x)-v_{\varrho}(y)|^p\geq \frac{1}{M^p}$. Hence the growth condition (\ref{realcut}) implies the inequality
\begin{equation}\label{localsymmetry}
1\leq C\Big(f(\e_{\varrho}|\nabla_{\w,1}(v_{\varrho},\R^d)|^p(x))+f(\e_{\varrho}|\nabla_{\w,1}(v_{\varrho},\R^d)|^p(y))\Big).
\end{equation} 
In order to control the location of such $y$, observe that by (\ref{ubdist}) we have $B_G(x_i,\eta_i)\subset B_{M\eta_i}(x_i)$, which in turn implies (for $M\geq 1$) that
\begin{equation}\label{localisation}
\bigcup_i\bigcup_{x\in B_G(x_i,\eta_i)}N_x\subset \e_{\varrho}^{-1}B_{3Mk^2\varrho}(x_0).
\end{equation}
Here we also used (\ref{finiterange}) and that $Mk^2\varrho\geq M\e_{\varrho}$. Next we sum the estimate (\ref{localsymmetry}) over $x$. Note that due to (\ref{neighbours}) every term can appear at most $M+1$ times. By the definition of $v_{\varrho}$ we obtain
\begin{equation*}
\e_{\varrho}^{d-1}\#\Big(\bigcup_i B_G(x_i,\eta_i)\cap\mathcal{J}_{\varrho}\Big)\leq C F_{\e_{\varrho}}(\w)(T_{k\varrho}u_{\e_{\varrho}},B_{3Mk^2\varrho}(x_0))\leq C |\xi|^p \varrho^d,
\end{equation*}
where we applied the first bound in (\ref{choicerho}). By truncation we further know that $|\nabla_{e,1} v_{\varrho}|(x)\leq C\e_{\varrho}^{-1}k\varrho$ for all $x\in\Lw$, so that
\begin{equation}\label{cardbound2}
|\nabla_1 v_{\varrho}|\Big( \bigcup_i B_G(x_i,\eta_i)\cap \mathcal{J}_{\varrho}\Big)\leq C |\xi|^p \Big(\frac{\varrho}{\e_{\varrho}}\Big)^{d}\varrho.
\end{equation}
In order to estimate the remaining contributions in the right hand side of \eqref{cardbound1} we use H\"older's inequality in the form
\begin{equation}\label{holder}
|\nabla_1 v_{\varrho}|\Big( \bigcup_i B_G(x_i,\eta_i)\cap S_{\varrho}^{\Lambda}\backslash\mathcal{J}_{\varrho}\Big)\leq \#\Big(\bigcup_i B_G(x_i,\eta_i)\cap S_{\varrho}^{\Lambda}\backslash \mathcal{J}_{\varrho}\Big)^{\frac{p-1}{p}}\Big(\sum_{x\in \bigcup_iB_G(x_i,\eta_i)\backslash \mathcal{J}_{\varrho}}|\nabla_{e,1}v_{\varrho}|^p(x)\Big)^{\frac{1}{p}}.
\end{equation}
In the last term we have to pass from the undirected gradient to the directed version: for $x\notin\mathcal{J}_{\varrho}$ and $y\in N_x$ it holds that $\e_{\varrho}|v_{\varrho}(x)-v_{\varrho}(y)|^p\leq 1$. Hence we infer from the bound (\ref{realcut}) that for $x\in \bigcup_iB_G(x_i,\eta_i)\backslash\mathcal{J}_{\varrho}$
\begin{align*}
|\nabla_{e,1}v_{\varrho}(x)|^p(x)\leq& C\sum_{y\in N_x}|v_{\varrho}(x)-v_{\varrho}(y)|^p=\frac{C}{\e_{\varrho}}\sum_{y\in N_x}\min\{\e_{\varrho}|v_{\varrho}(x)-v_{\varrho}(y)|^p,1\}
\\
\leq& \frac{C}{\e_{\varrho}}\sum_{y\in N_x}f(\e_{\varrho}|\nabla_{\w,1}(v_{\varrho},\e_{\varrho}^{-1}B_{3Mk^2\varrho}(x_0))|^p(x) )+f(\e_{\varrho}|\nabla_{\w,1}(v_{\varrho},\e_{\varrho}^{-1}B_{3Mk^2\varrho}(x_0))|^p(y)),
\end{align*}
where we used again (\ref{localisation}). We sum this estimate and by (\ref{neighbours}) each term is counted at most $2M$ times. Thus in combination with the first estimate in (\ref{choicerho}) we have
\begin{equation}\label{globalsymmetry}
\Big(\sum_{x\in \bigcup_iB_G(x_i,\eta_i)\backslash\mathcal{J}_{\varrho}}|\nabla_{e,1}v_{\varrho}|^p(x)\Big)^{\frac{1}{p}}\leq C\e_{\varrho}^{-\frac{d}{p}}\Big(F_{\e_{\varrho}}(\w)(T_ku_{\e_{\varrho}},B_{3Mk^2\varrho}(x_0))\Big)^{\frac{1}{p}}\leq\Big(\frac{\varrho}{\e_{\varrho}}\Big)^{\frac{d}{p}}C|\xi|. 
\end{equation}
Combining this estimate with (\ref{holder}) leads to
\begin{equation*}
|\nabla_1 v_{\varrho}|\Big( \bigcup_i B_G(x_i,\eta_i)\cap S_{\varrho}^{\Lambda}\backslash\mathcal{J}_{\varrho}\Big)\leq 
C\#\Big(\bigcup_i B_G(x_i,\eta_i)\cap S_{\varrho}^{\Lambda}\backslash \mathcal{J}_{\varrho}\Big)^{\frac{p-1}{p}}\Big(\frac{\varrho}{\e_{\varrho}}\Big)^{\frac{d}{p}}|\xi|.
\end{equation*}
In order to bound the cardinality term, note that by the definition of $S^{\Lambda}_{\varrho}$ and (\ref{globalsymmetry}) it holds that
\begin{equation*}
\#\Big(\bigcup_i B_G(x_i,\eta_i)\cap S_{\varrho}^{\Lambda}\backslash \mathcal{J}_{\varrho}\Big)\Big(\frac{\Lambda}{2}\Big)^p\leq \sum_{x\in \bigcup_i B_G(x_i,\eta_i)\backslash \mathcal{J}_{\varrho}}|\nabla_{e,1} v_{\varrho}|^p(x)\leq C|\xi|^p\Big(\frac{\varrho}{\e_{\varrho}}\Big)^d.
\end{equation*}
Plugging this estimate into the previous one yields
\begin{equation}\label{cardbound3}
|\nabla_1 v_{\varrho}|\Big( \bigcup_i B_G(x_i,\eta_i)\cap S_{\varrho}^{\Lambda}\backslash\mathcal{J}_{\varrho}\Big)\leq C\Lambda^{1-p}|\xi|^p\Big(\frac{\varrho}{\e_{\varrho}}\Big)^d.
\end{equation}
Applying trice the doubling property of Lemma \ref{graphprop} and combining (\ref{cardbound1}), (\ref{cardbound2}) and (\ref{cardbound3}) we infer that
\begin{equation*}
\#\big((\Lw\cap\e_{\varrho}^{-1}B_{k\varrho}(x_0))\backslash R^{\Lambda}_{\varrho}\big)\leq \#\Big(\bigcup_i B_G(x_i,5\eta_i)\Big)\leq C\#\Big(\bigcup_i B_G(x_i,\eta_i)\Big)
\leq C|\xi|^p\Big(\frac{\varrho}{\e_{\varrho}}\Big)^d\Big(\varrho\Lambda^{-1}+\Lambda^{-p}\Big).
\end{equation*}
We choose $\Lambda=\Lambda_{\varrho}$ as $\Lambda^{p-1}=\varrho^{-1}$, so that the last inequality can be written as
\begin{equation}\label{cardboundfinal}
\#\big((\Lw\cap\e_{\varrho}^{-1}B_{k\varrho}(x_0))\backslash R^{\Lambda_{\varrho}}_{\varrho}\big)\leq C|\xi|^p\Big(\frac{\varrho}{\e_{\varrho}}\Big)^d\Lambda_{\varrho}^{-p}.
\end{equation}
With this choice of $\Lambda_{\varrho}$, we now construct the Lipschitz competitor. First observe that for any $x,y\in \Lw\cap\e_{\varrho}^{-1}B_{k\varrho}(x_0)$ the definition of $k$ yields
\begin{equation*}
\overline{C}\mathrm{d}_G(x,y)\leq \overline{C}\,C_{r,R}|x-y|\leq 2\,\overline{C}\,C_{r,R}k\varrho\e_{\varrho}^{-1}\leq k^2\varrho\e_{\varrho}^{-1},
\end{equation*}
so that (\ref{maximal-lip}) and (\ref{lbdist}) imply for any $x,y\in R^{\Lambda_{\varrho}}_{\varrho}$ the Lipschitz estimate
\begin{equation*}
|v_{\varrho}(x)-v_{\varrho}(y)|\leq 2\,\overline{C}\,C_{r,R}\Lambda_{\varrho}|x-y|\leq k\Lambda_{\varrho}|x-y|.
\end{equation*}
Using Kirszbraun's extension theorem we find a Lipschitz function $\tilde{v}_{\varrho}:\Lw\to\R^m$ with Lipschitz constant $k\Lambda_{\varrho}$ that agrees with $v_{\varrho}$ on $ R^{\Lambda_{\varrho}}_{\varrho}$. Moreover, by truncation via the operator $T_{3\e_{\rho}^{-1}k\rho}$ we can additionally assume that $\|\tilde{v}_{\varrho}\|_{\infty}\leq 9\e_{\varrho}^{-1}k\varrho$.

\medskip
\textbf{Step 2} From Lipschitz continuity to equiintegrability of discrete gradients
\\
It will be convenient to rescale the function $\tilde{v}_{\varrho}$ constructed in the first step onto $B_1$. First we introduce some notation. We set $\sigma_{\varrho}=\frac{\e_{\varrho}}{\varrho}$ and $\Lw_{\varrho}:=\Lw-\frac{x_0}{\e_{\varrho}}$. For any $x\in\Lw_{\varrho}$ we further denote by $N_{x,\varrho}=N_{x+\frac{x_0}{\e_{\varrho}}}-\frac{x_0}{\e_{\varrho}}$ the set of adjacent points in the undirected shifted graph. In the notation for the discrete gradients we will replace $e$ by $e_0$ and $\w$ by $\w_0$, respectively. Define $u^{\varrho}:\sigma_{\varrho}\Lw_{\varrho}\to \R^m$ via 
\begin{equation*}
u^{\varrho}(\sigma_{\varrho}x)=\sigma_{\varrho}\tilde{v}_{\varrho}\left(x+\frac{x_0}{\e_{\varrho}}\right).
\end{equation*}
By the properties of $\tilde{v}_{\varrho}$ established in the first step, the function $u^{\varrho}$ satisfies
\begin{itemize}
	\item[(i)] $\|u^{\varrho}\|_{\infty}\leq 9k$;
	\item [(ii)] $|u^{\varrho}(\sigma_{\varrho} x)-u^{\varrho}(\sigma_{\varrho} y)|\leq k\Lambda_{\varrho}\sigma_{\varrho}|x-y|$ for all $x,y\in\Lw_{\varrho}$;
	\item[(iii)] $u^{\varrho}(\sigma_{\varrho} x)=\sigma_{\varrho} v_{\varrho}(x+\frac{x_0}{\e_{\varrho}})=\varrho^{-1}T_{k\varrho}u_{\e_{\varrho}}(\e_{\varrho}x+x_0)$ for all $x\in R_{\varrho}^{\Lambda_{\varrho}}-\frac{x_0}{\e_{\varrho}}$.
\end{itemize}
We aim at applying Lemma \ref{l.fomupediscrete} with $x_0=0$, $\lambda=1$ and the vanishing sequence $\sigma_{\rho}$ (note that the shift of the graph preserves admissibility). Due to (ii) we have the bound $|\nabla_{e_0,\sigma_{\varrho}}u^{\varrho}|^p(\sigma_{\varrho} x)\leq C\Lambda^p_{\varrho}$. In combination with (iii) and a change of variables we derive the bound
\begin{align*}
&\sum_{\substack{\sigma_{\varrho} x\in\sigma_{\varrho}\Lw_{\varrho}\cap B_{k}\\ N_{x,\varrho}\backslash (R_{\varrho}^{\Lambda_{\varrho}}-\frac{x_0}{\e_{\varrho}})\neq\emptyset}}\sigma_{\varrho}^d|\nabla_{e_0,\sigma_{\varrho}}u^{\varrho}|^p(\sigma_{\varrho} x)+ \sum_{\substack{\sigma_{\varrho} x\in\sigma_{\varrho}\Lw_{\varrho}\cap B_{k}\\ N_{x,\varrho}\subset (R_{\varrho}^{\Lambda_{\varrho}}-\frac{x_0}{\e_{\varrho}})}}\sigma_{\varrho}^d|\nabla_{e_0,\sigma_{\varrho}}u^{\varrho}|^p(\sigma_{\varrho} x)
\\
\leq& C\Lambda_{\varrho}^p\sigma_{\varrho}^d\#\Big((\Lw\cap \e_{\varrho}^{-1}B_{k\varrho}(x_0))\backslash R_{\varrho}^{\Lambda_{\varrho}}\big)\Big)+\sum_{x\in R_{\e,\varrho}^{\Lambda_{\varrho}}}\sigma_{\varrho}^{d}|\nabla_{e,1}v_{\varrho}|^p(x).
\end{align*}
By definition of the maximal function operator it holds that $\e_{\varrho}|\nabla_{e,1}v_{\varrho}|^p(x)\leq \e_{\varrho}\Lambda_{\varrho}^p\leq 1$ for $x\in R^{\Lambda_{\varrho}}_{\varrho}$, so that we can use (\ref{cardboundfinal}) and the same reasoning as for (\ref{globalsymmetry}) to obtain the estimate
\begin{equation*}
\sum_{\sigma_{\varrho} x\in\sigma_{\varrho}\Lw_{\varrho}\cap B_{k}}\sigma_{\varrho}^d|\nabla_{e_0,\sigma_{\varrho}}u^{\varrho}|^p(\sigma_{\varrho} x)\leq C|\xi|^p +C\varrho^{-d}F_{\e}(\w)(T_{k\varrho}u_{\e_{\varrho}},B_{3Mk^2\varrho}(x_0))\leq C|\xi|^p.
\end{equation*}
Hence our choice of $k$ allows to apply Lemma \ref{l.fomupediscrete} and we obtain a subsequence $\varrho_j$ and a sequence $w_j:\sigma_{\varrho_j}\Lw_{\varrho_j}\to\R^m$ such that, setting $\sigma_j=\sigma_{\varrho_j}$, the sequence $|\nabla_{e_0,\sigma_j}w_j|^p$ is equiintegrable on $B_2$ and 
\begin{equation}\label{eq:smallmod}
\lim_j\sigma_j^d\#\{\sigma_j x\in\sigma_j\Lw_{\varrho_j}\cap B_2:\,\,u^{\varrho_j}\not\equiv w_j\text{ on }\sigma_jB_{G_{\varrho_j}}(x,1) \}=0,
\end{equation}
where $G_{\varrho_j}$ denotes the undirected shifted graph. Finally, by a truncation argument based on the operator $T_{9k}$ we can assume that $\|w_{j}\|_{\infty}\leq 27k$.

\medskip
\textbf{Step 3} Proof of $\displaystyle\lim_{\varrho\to 0}\varrho^{-d}F(\w)(u_{x_0,\xi},B_{\varrho}(x_0))\geq|B_1|q(x_0,\xi)$
\\
Let $w_j$ be the sequence constructed in Step 2. First we estimate the $L^p(B_1)$-norm of the sequence $w_j-u_{0,\xi}$. Define the set
\begin{equation*}
U_j=\left\{\sigma_j x\in \sigma_j\Lw_{\varrho_j}:\, B_{G_{\varrho_j}}(x,1)\subset R^{\Lambda_{\varrho_j}}_{\varrho_j}-x_0\e^{-1}_{\varrho_j}\right\}\backslash R_j^{\prime},
\end{equation*}
where $R_j^{\prime}$ denotes the set in \eqref{eq:smallmod}. Then by construction $w_j(\sigma_j x)=\varrho_j^{-1}T_{k\varrho_j}u_{\e_{\varrho_j}}(x_0+\e_j x)$ for all $\sigma_j x\in B_2\cap U_j$. Moreover, by \eqref{neighbours} and \eqref{cardboundfinal} we can bound the cardinality of the complement via
\begin{equation*}
\sigma_j^d\#\left(\sigma_j\Lw_{\varrho_j}\cap B_2\backslash U_j\right)\leq \sigma_j^d\#R^{\prime}_j+C\sigma_j^d\#\left(\Lw\cap\e_{\varrho_j}^{-1}B_{k\varrho_j}(x_0)\backslash R_{\varrho_j}^{\Lambda_{\varrho_j}}\right)\leq \sigma_j^d\#R_j^{\prime}+C|\xi|^p\Lambda_{\varrho_j}^p,
\end{equation*}
so that by \eqref{eq:smallmod} and the choice of $\Lambda_{\varrho_j}$ we have
\begin{equation}\label{eq:smallmod2}
\lim_{j\to +\infty}\sigma_j^d\#\left(\sigma_j\Lw_{\varrho_j}\cap B_2\backslash U_j\right)=0.
\end{equation}
Due to the $L^{\infty}$-bound on $w_j$, for $j$ large enough a change of variables yields
\begin{equation*}
\|w_j-u_{0,\xi}\|^p_{L^p(B_1)}\leq \frac{C}{\varrho_j^{p}}\dashint_{B_{\varrho_j}(x_0)}|T_{k\varrho_j}u_{\e_{\varrho_j}}-u_{x_0,\xi}|^p\,\mathrm{d}z
+C \sigma_j^d\#\left(\sigma_j\Lw_{\varrho_j}\cap B_2\backslash U_j\right).
\end{equation*}
Hence \eqref{choicerho} and \eqref{eq:smallmod2} imply that $w_j\to u_{0,\xi}$ in $L^p(B_1,\R^m)$. Now we turn to the energy estimates. Fix $\eta>0$. Since $|\nabla_{e_0,\sigma_j}w_j|^p$ is piecewise constant with respect to the Voronoi tessellation of $\sigma_j\Lw_{\varrho_j}$, Remark \ref{voronoi}, \eqref{eq:smallmod2} and the equiintegrability of $|\nabla_{e_0,\sigma_j}w_j|^p$ on $B_2$ imply that there exists $j_{\eta}$ such that for all  $j\geq j_{\eta}$
\begin{equation}\label{equint1}
\sum_{\sigma_j x\in \sigma_j\Lw_{\varrho_j}\cap B_1\backslash U_j}\sigma_j^d|\nabla_{e_0,\sigma_j}w_j|^p(\sigma_j x)\leq \eta.
\end{equation}
For $t>0$ let us further introduce the sets
\begin{equation*}
S_j(t)=\{\sigma_jx\in\sigma_j\Lw_{\varrho_j}\cap B_1:\;|\nabla_{e_0,\sigma_j}w_j|^p(\sigma_j x)>t\}.
\end{equation*}
Again due to the equiintegrability of $|\nabla_{e_0,\sigma_j}w_j|^p$ on $B_2$ we find $t_{\eta}>0$ such that for $j\geq j_{\eta}$ we have
\begin{equation}\label{equint2}
\sum_{\sigma_jx\in S_j(t_{\eta})}\sigma_j^d|\nabla_{e_0,\sigma_j}w_j|^p(\sigma_j x)\leq C\int_{B_2\cap\{|\nabla_{e_0,\sigma_j}w_j|^p>t_{\eta}\}}|\nabla_{e_0,\sigma_j}w_j|^p(z)\,\mathrm{d}z\leq\eta.
\end{equation}
Moreover, if $\e_{\varrho_j} x\in (x_0+\varrho_j (B_1\cap U_j\backslash S_j(t_{\eta}))$, then 
\begin{equation}\label{smallgrad}
\|\e_{\varrho_j}|\nabla_{\w,\e_{\varrho_j}}(T_{k\varrho_j}u_{\e_{\varrho_j}},B_{\varrho_j}(x_0))|^p(\e_j x)\|_1\leq\e_{\varrho_j}|\nabla_{e_0,\sigma_j}w_j|^p(\sigma_jx-\varrho_j^{-1}x_0)\leq \e_{\varrho_j}t_{\eta}.
\end{equation}
The right hand side converges to zero. Thus, after enlarging $j_{\eta}$, assumption \eqref{linearatzero} yields
\begin{equation*}
f\big(\e_{\varrho_j}|\nabla_{\w,\e_{\varrho_j}}(T_{k\varrho_j}u_{\e_{\varrho_j}},B_{\varrho_j}(x_0))|^p(\e_j x)\big)\geq (1-\eta)\,\alpha\,\|\e_{\varrho_j}|\nabla_{\w,\e_{\varrho_j}}(T_{k\varrho_j}u_{\e_{\varrho_j}},B_{\varrho_j}(x_0))|^p(\e_{\varrho_j} x)\|_1
\end{equation*}
for all $j\geq j_{\eta}$ and all $\e_{\varrho_j} x\in (x_0+\varrho_j (B_1\cap U_j\backslash S_j(t_{\eta}))$. For the remaining lattice points we can use (\ref{equint1}) and (\ref{equint2}), so that from a change of variables we deduce the lower bound
\begin{align*}
\varrho_j^{-d}F_{\e_{\varrho_j}}(\w)(T_{k\varrho_j}u_{\e_{\varrho_j}},B_{\varrho_j}(x_0))&\geq(1-\eta)\alpha\sum_{\sigma_jx\in\sigma_j\Lw_{\varrho_j}\cap B_1}\sigma_j^{d-1}\|\sigma_j|\nabla_{\w_0,\sigma_j}(w_j,B_{1})|^p(\sigma_j x)\|_1-2\eta
\\
&=(1-\eta)G_{\e_{\varrho_j},\varrho_j}(x_0,\w)(w_j)-2\eta
\end{align*}
with the functional $G_{\e,\varrho}(x_0,\w)$ defined in Lemma \ref{changeofvar}. Since we have chosen $x_0$ and $\e_{\varrho}$ such that Lemma \ref{diagonal} holds, we deduce from (\ref{onesequence}) and the convergence $w_j\to u_{0,\xi}$ in $L^p(B_1,\R^m)$ that
\begin{align*}
\lim_{\varrho\to 0}\varrho^{-d}F(\w)(u_{x_0,\xi},B_{\varrho}(x_0))&=\lim_j \varrho_{j}^{-d}F_{\e_{\varrho_j}}(\w)(T_{k\varrho_j}u_{\e_{\varrho_j}},B_{\varrho_j}(x_0))
\geq \liminf_j(1-\eta)G_{\e_{\varrho_j},\varrho_j}(x_0,\w)(w_j)-2\eta\nonumber
\\
&\geq (1-\eta)\int_{B_1}q(x_0,\xi)\,\mathrm{d}z-2\eta=(1-\eta)|B_1|q(x_0,\xi)-2\eta.
\end{align*}
Since $\eta>0$ was arbitrary, we conclude that for a.e. $x_0\in D$ and all $\xi\in\R^{m\times d}$
\begin{equation}\label{bulklb}
\lim_{\varrho\to 0}\varrho^{-d}F(\w)(u_{x_0,\xi},B_{\varrho}(x_0))\geq |B_1|q(x_0,\xi).
\end{equation}
Note that the exceptional set may depend on the subsequence chosen at the beginning.

\medskip
\textbf{Step 4} Proof of $\displaystyle\lim_{\varrho\to 0}\varrho^{-d}F(\w)(u_{x_0,\xi},B_{\varrho}(x_0))\leq|B_1|q(x_0,\xi)$
\\
To prove the reverse inequality in (\ref{bulklb}) we take a sequence $u_{\e_n}\in\mathcal{PC}_{\e_n}^{\w}$ converging to $u_{x_0,\xi}$ in $L^p(D,\R^m)$ and such that 
\begin{equation*}
\lim_{\e_n\to 0}E_{\e_n}(\w)(u_{\e_n},B_{\varrho_0}(x_0))=E(\w)(u_{x_0,\xi},B_{\varrho_0}(x_0)).
\end{equation*}
The arguments are very similar to Steps 2 and 3, so we just sketch them. As in (\ref{universalrec}) one can show that the truncated functions $T_{k\varrho}u_{\e_n}$ form a recovery sequence on all balls $B_{\varrho}(x_0)$ with $0<\varrho<\varrho_0$. This time we apply Lemma \ref{l.fomupediscrete} to $T_{k\rho}u_{\e_n}$ with the chosen $x_0$ and $\lambda=\varrho$. Note that the assumptions are satisfied since for $\varrho\leq\varrho_0/(2k)$ we have that
\begin{equation*}
\sum_{\e_n x\in\e_n\Lw\cap B_{k\varrho}(x_0)}\e_n^d|\nabla_{e,\e_n}T_{k\varrho}u_{\e_n}|^p(\e_n x) \leq  CE_{\e_n}(\w)(T_{k\varrho}u_{\e_n},B_{2k\varrho}(x_0))\leq C E_{\e_n}(\w)(u_{\e_n},B_{\varrho_0}(x_0)).
\end{equation*}
Hence we find a subsequence $\e_{n_j}$ and a sequence $v_j\in\mathcal{PC}_{\e_{n_j}}^{\w}$ (both depending on $\varrho$) such that, setting $\e_j=\e_{n_j}$, the sequence $|\nabla_{e,\e_j}v_j|^p$ is equiintegrable on $B_{2\varrho}(x_0)$ and
\begin{equation}\label{eq:smallmod3}
\lim_j\e_{j}^d\#\left\{\e_{j}x\in\e_{j}\Lw\cap B_{2\varrho}(x_0):\,T_{k\varrho}u_{\e_j}\not\equiv v_j\text{ on }\e_jB_G(x,1)\right\}=0.
\end{equation}
By truncation we may further assume that $\|v_j\|_{\infty}\leq 9k\varrho$. Fix $\eta>0$ and for $t>0$ define the sets $S_j(t)$ by
\begin{equation*}
S_j(t):=\left\{\e_jx\in\e_j\Lw\cap B_{\varrho}(x_0):\;|\nabla_{e,\e_j}v_j|^p(\e_jx)>t\right\}.
\end{equation*}
We choose $t_{\eta}>0$ (possibly depending on $\varrho$) such that, keeping in mind the inequality $f(\mathfrak{p})\leq C_f\|\mathfrak{p}\|_1$, for $j$ large enough it holds that
\begin{equation}\label{equi1}
\varrho^{-d}\sum_{\e_jx\in S_j(t_{\eta})}\e_j^{d-1}f\Big(\e_j|\nabla_{\w,\e_j}(v_j,B_{\varrho}(x_0))|^p(\e_jx)\Big)
\leq C\rho^{-d}\int_{B_{2\varrho}(x_0)\cap\{|\nabla_{e,\e_j}v_j|^p>t_{\eta}\}}|\nabla_{e,\e_j }v_j|^p(z)\,\mathrm{d}z\leq \eta,
\end{equation}
which is possible due to the equiintegrability of $|\nabla_{e,\e_j}v_j|^p$ on $B_{2\varrho}(x_0)$. Denoting by $W_j$ the set in \eqref{eq:smallmod3}, the same arguments yield $j_{\eta}\in\mathbb{N}$ such that for all $j\geq j_{\eta}$ we have
\begin{align}\label{equi2}
\varrho^{-d}\sum_{\e_jx\in W_j\cap B_{\varrho}(x_0)}\e_j^{d-1}f\Big(\e_j|\nabla_{\w,\e_j}(v_j,B_{\varrho}(x_0))|^p(\e_jx)\Big)\leq \eta.
\end{align}
Similar to \eqref{smallgrad}, for $\e_j x\in\e_j\Lw\cap B_{\varrho}(x_0)\backslash (W_j\cup S_j(t_{\eta}))$ we know that $\|\e_j|\nabla_{w,\e_j}(v_j,B_{\varrho}(x_0))|^p\|_1\leq\e_jt_{\eta}$. Hence assumption (\ref{linearatzero}) and the bounds (\ref{equi1}) and (\ref{equi2}) imply for large enough $j$ the estimate
\begin{equation*}
\varrho^{-d}E_{\e_j}(\w)(T_{k\varrho}u_{\e_j},B_{\varrho}(x_0))\geq \varrho^{-d}(1-\eta)F_{\e_j}(\w)(v_j,B_{\varrho}(x_0))-2\eta.
\end{equation*}
From \eqref{eq:smallmod3} and the uniform boundedness of $v_j$  we infer that $v_j\to u_{x_0,\xi}$ in $L^1(B_{\varrho}(x_0),\R^m)$. By a modification on $\e_j\Lw\backslash B_{\varrho}(x_0)$ not affecting the energy, this convergence also holds in $L^1(D,\R^m)$ and therefore we deduce from $\Gamma$-convergence along the subsequence $\e_j$ and the previous inequality that
\begin{equation*}
\varrho^{-d}E(\w)(u_{x_0,\xi},B_{\varrho}(x_0))=\lim_j\varrho^{-d}E_{\e_j}(\w)(T_{k\varrho}u_{\e_j},B_{\varrho}(x_0))\geq\varrho^{-d}(1-\eta)F(\w)(u_{x_0,\xi},B_{\varrho}(x_0))-2\eta.
\end{equation*}
In view of Remark \ref{alsopointwise} and the arbitrariness of $\eta$ we conclude that
\begin{equation*}
|B_1|q(x_0,\xi)=\lim_{\varrho\to 0}\varrho^{-d}E(\w)(u_{x_0,\xi},B_{\varrho}(x_0))\geq \lim_{\varrho\to 0}\varrho^{-d}F(\w)(u_{x_0,\xi},B_{\varrho}(x_0)).
\end{equation*}
Combined with (\ref{bulklb}), this estimate yields the claim along the chosen subsequence. In the general case, we obtain that along any subsequence of $\e_n$ the $\Gamma$-limit of $E_{\e}(\w)$ is uniquely defined by the integrand $h(x,\xi)$, so that the $\Gamma$-limit along the sequence $\e_n$ exists by the Urysohn-property of $\Gamma$-convergence, although the integrand might differ on a negligible set depending on the subsequence.
\end{proof}

\subsection{Characterization of the surface density and conclusion} 

Having identified the bulk term, we now show that the computation of the surface integrand $\varphi(x,a,\nu)$ can be performed with the discrete functional $F_{\e}(\w)$ restricted to functions taking only the two values $a$ and $0$. Then we prove that the surface density $\varphi$ agrees with the function $s$ defined in Remark \ref{r.blowupformulas} via the energy $I_{\e}(\w)$. 

We study the asymptotic minimization problems given by Proposition \ref{limitsbvp} and their connection to boundary value problems for the discrete functionals $F_{\e}(\w)$. More precisely, as a first step we compare the two quantities
\begin{equation*}
\begin{split}
m^{\delta}_{\e}(\w)(\bar{u},A)&=\inf\{F_{\e}(\w)(v,A):\;v\in\mathcal{PC}^{\w}_{\e,\delta}(\bar{u},A)\},
\\
m(\w)(\bar{u},A)&=\inf\{F(\w)(v,A):\;v\in SBV^p(A,\R^m),\,v=\bar{u}\text{ in a neighborhood of }\partial A\},
\end{split}
\end{equation*}
where the limit functional $F(\w)$ is given (up to subsequences) by Proposition \ref{limitsbvp} and the set $\mathcal{PC}_{\e,\delta}^{\w}(\bar{u},A)$, which takes into account discrete boundary conditions, is defined in \eqref{bdryclass}. We restrict the class of boundary conditions to pointwise well-defined functions $\bar{u}\in SBV^p(D,\R^m)\cap L^{\infty}(D,\R^m)$ such that, setting $\bar{u}_{\e}\in \mathcal{PC}_{\e}^{\w}$ as $\bar{u}_{\e}(\e x)=\bar{u}(\e x)$, it holds that
\begin{equation}\label{boundarydata}
\begin{split}
&\limsup_{\e\to 0}F_{\e}(\bar{u}_{\e},B)\leq C\int_B|\nabla \bar{u}|^p\,\mathrm{d}x+C\mathcal{H}^{d-1}(S_{\bar{u}}\cap \overline{B}),\\
&\bar{u}_{\e}\to \bar{u} \text{ in }L^1(D,\R^m),\quad\quad\mathcal{H}^{d-1}(S_{\bar{u}}\cap\partial A)=0
\end{split}
\end{equation}
for some $C>0$ uniformly for $B\in\Ard$. In particular, as seen in the proof of Lemma \ref{bounds}, we allow for piecewise smooth functions with polyhedral jump set that has no mass on $\partial A$. We have the following convergence result.
\begin{lemma}[Approximation of minimum values]\label{approxminprob}
	Let $\e_n$ and $F(\w)$ be as in Proposition \ref{limitsbvp}. Then, for any $A\in\Ard$ and ${\bar{u}}$ as in (\ref{boundarydata}), it holds that
	\begin{equation*}
	\lim_{\delta\to 0}\liminf_nm^{\delta}_{\e_n}(\w)({\bar{u}},A)=\lim_{\delta\to 0}\limsup_{n}m^{\delta}_{\e_n}(\w)({\bar{u}},A)=m(\w)({\bar{u}},A).
	\end{equation*}
\end{lemma}
\begin{proof}
	First note that by monotonicity the limits with respect to $\delta$ exist. Moreover, from the first assumption in (\ref{boundarydata}) it follows that $m^{\delta}_{\e}(\w)({\bar{u}},A)$ is equibounded. For $n\in\mathbb{N}$ let $u_n\in\mathcal{PC}^{\w}_{\e_n,\delta}({\bar{u}},A)$ be such that $m^{\delta}_{\e_n}(\w)({\bar{u}},A)=F_{\e_n}(\w)(u_n,A)$. Since ${\bar{u}}\in L^{\infty}$ we can apply Lemma \ref{trunc} and assume without loss of generality that $|u_n(\e_n x)|\leq 3\|{\bar{u}}\|_{\infty}$ for all $x\in\Lw$. By Lemma \ref{compact} we know that, up to a subsequence (not relabeled), $u_n\to u$ in $L^1(A,\R^m)$ for some $u\in L^1(A,\R^m)\cap GSBV^p(A,\R^m)$. Using Remark \ref{voronoi} and again (\ref{boundarydata}) we infer that $u={\bar{u}}$ on $\partial_{\delta}A$. Note that $u\in L^{\infty}(A,\R^m)$, which implies $u\in SBV^p(A,\R^m)$. Up to extension we can assume that $u$ is admissible in the infimum problem defining $m(\w)({\bar{u}},A)$ and Proposition \ref{limitsbvp} yields
	\begin{equation*}
	m(\w)({\bar{u}},A)\leq F(\w)(u,A)\leq\liminf_n F_{\e_n}(\w)(u_n,A)\leq\liminf_n m^{\delta}_{\e_n}(\w)({\bar{u}},A).
	\end{equation*}
	As $\delta$ was arbitrary, we conclude that $m(\w)({\bar{u}},A)\leq\lim_{\delta\to 0}\liminf_nm^{\delta}_{\e_n}(\w)({\bar{u}},A)$.
	
	In order to prove the remaining inequality, for given $\theta>0$ we let $u\in SBV^p(A,\R^m)$ be such that $u={\bar{u}}$ in a neighborhood of $\partial A$ and $F(\w)(u,A)\leq m(\w)({\bar{u}},A)+\theta$. Take $u_{n}\in\mathcal{PC}^{\w}_{\e_n}$ converging to $u$ in $L^1(D,\R^m)$ and satisfying
	\begin{equation}\label{rec}
	\lim_{n}F_{\e_n}(\w)(u_n,A)=F(\w)(u,A).
	\end{equation}
	We will modify $u_n$ such that it fulfills the discrete boundary conditions. The argument follows the proof of Proposition \ref{subadd}. Since $u=\bar{u}$ in a neighborhood of $\partial A$, there exist sets $A^{\prime}\subset\subset A^{\prime\prime}\subset\subset A$ such that $A^{\prime},A^{\prime\prime}\in\Ard$ and
	\begin{equation}\label{boundarylayer}
	u={\bar{u}} \quad\text{ on }A\backslash A^{\prime}. 
	\end{equation}
	Fix $N\in\mathbb{N}$. For $h\leq \dist(A^{\prime},\partial A^{\prime\prime})$ and $i\in\{1,\dots,N\}$ we define the sets
	\begin{equation*}
	A_i=\left\{x\in A:\;\dist(x,A^{\prime})<i\frac{h}{2N}\right\}.
	\end{equation*}
	Let $\Theta_i$ be a cut-off function between the sets $A_i$ and $A_{i+1}$ with $\|\nabla\Theta_i\|_{\infty}\leq \frac{4N}{h}$ and define $u^i_{n}\in \mathcal{PC}_{\e_n}^{\w}$ by
	\begin{equation*}
	u^i_{n}(\e_n x)=\Theta_i(\e_n x)u_{n}(\e_n x)+(1-\Theta_i(\e_n x)){\bar{u}}(\e_n x).
	\end{equation*}
	Up to extending $u$ on $D\backslash A$ via $u_{|D\backslash A}={\bar{u}}$, by (\ref{boundarydata}) and (\ref{boundarylayer}) we can assume that for each $i\in\{1,\dots,N\}$ it holds that $u^i_{n}\to u$ in $L^1(D,\R^m)$. Setting $S_{n}^{i}:=\{x\in A:\;\dist(x,A_{i+1}\backslash{A_{i-1}})<3M\e_n\}$, we have
	\begin{equation}\label{splitineq}
	F_{\e_n}(\w)(u^i_{n},A)\leq F_{\e_n}(\w)(u_{n},A)+F_{\e_n}(\w)({\bar{u}}_{\e_n},A\backslash\overline{A^{\prime}})
	+F_{\e_n}(\w)(u_n^i,S_n^i).
	\end{equation}
	Using the same arguments as in the proof of Proposition \ref{subadd} we infer that
	\begin{equation*}
	F_{\e}(\w)(u_n^i,S_n^i)\leq C\left(F_{\e_n}(\w)(u_{n}, S_{n}^i)+F_{\e_n}(\w)({\bar{u}}_{\e_n},S_{n}^i)\right)
	+CN^ph^{-p}\sum_{\e_n x\in \e_n\Lw\cap S^i_{n}}\e_n^d|u_{n}(\e x)-{\bar{u}}_{\e_n}(\e x)|^p.
	\end{equation*}
	By construction $S_{n}^i\cap S_{n}^j=\emptyset$ for $|i-j|\geq 3$ and $S_{n}^i\subset\subset A\backslash A^{\prime}$ for $i\geq 2$. Averaging the previous inequality and using (\ref{boundarydata}) and (\ref{rec}) yields
	\begin{equation*}
	\frac{1}{N-1}\sum_{i=2}^NF_{\e}(\w)(u_n^i,S_n^i)
	\leq\frac{C}{N}+CN^{p-1}h^{-p}\|u_{n}-{\bar{u}}_{\e_n}\|^p_{L^p(A\backslash A^{\prime})}.
	\end{equation*}
	By equiboundedness, properties (\ref{boundarydata}) and (\ref{boundarylayer}) imply that $u_{n}-{\bar{u}}_{\e_n}\to 0$ in $L^p(A\backslash A^{\prime},\R^m)$. For every $n$ we choose $i_{n}\in\{2,\dots,N\}$ such that
	\begin{equation}\label{belowaverage}
	F_{\e_n}(\w)(u_n^{i_n},S_n^i)\leq\frac{C}{N}+CN^{p-1}h^{-p}\|u_{n}-{\bar{u}}_{\e_n}\|^p_{L^p(A\backslash A^{\prime})}. 
	\end{equation}
	Note that $u^{i_n}_{n}$ still converges to $u$ in $L^1(D,\R^m)$. Moreover, $u_n^{i_n}(\e_n x)={\bar{u}}(\e_n x)$ for all $\e_nx\in \e_n\Lw\cap A\backslash A^{\prime\prime}$. Hence $u^{i_n}_n\in\mathcal{PC}_{\e_n,\delta}^{\w}({\bar{u}},A)$ for all $\delta>0$ small enough. From (\ref{rec}), (\ref{splitineq}) and (\ref{belowaverage}) we obtain
	\begin{align*}
	\limsup_n m^{\delta}_{\e_n}(\w)({\bar{u}},A)&\leq\limsup_n F_{\e_n}(\w)(u^{i_n}_n,A)\leq F(\w)(u,A)+\limsup_nF_{\e_n}(\w)({\bar{u}}_{\e_n},A\backslash\overline{A^{\prime}})+ \frac{C}{N}\\
	&\leq m(\w)({\bar{u}},A)+\theta+C\int_{A\backslash\overline{A^{\prime}}}|\nabla {\bar{u}}|^p\,\mathrm{d}x+C\mathcal{H}^{d-1}(S_{\bar{u}}\cap A\backslash A^{\prime})+\frac{C}{N},
	\end{align*}
	where we used (\ref{boundarydata}) with $B=A\backslash\overline{A^{\prime}}$. As $\theta>0$ was arbitrary, the claim follows letting first $\delta\to 0$, then $N\to +\infty$ and finally $A^{\prime}\uparrow A$.
\end{proof}

In view of Proposition \ref{limitsbvp} and Lemma \ref{approxminprob} we can further characterize the surface densities of possible $\Gamma$-limits of the family $F_{\e}(\w)$ by analyzing the quantities $m^{\delta}_{\e}(\w)(u_{x_0,\nu}^{a,0},Q_{\nu}(x_0,\varrho))$.
To this end, we define the class of interfaces
\begin{equation*}
\mathcal{S}_{\e,\delta}^{\w}(u_{x_0,\nu}^{a,0},Q_{\nu}(x_0,\varrho))=\{u\in \mathcal{PC}_{\e,\delta}^{\w}(u_{x_0,\nu}^{a,0},Q_{\nu}(x_0,\varrho)):\;u(\e x)\in \{a,0\}\text{ for all }x\in\Lw\}.
\end{equation*}
We have the following important result, which also implies that the minimization defining the surface energy density can be performed on characteristic functions of sets of finite perimeter instead of general Caccioppoli partitions.
\begin{proposition}[Separation of surface effects]\label{separationofscales1}
	Let $\e_n\to 0$. Then, for all $x_0\in D$, all $a\in\R^m$ and all $\nu\in S^{d-1}$ it holds that
	\begin{align*}
	\limsup_{\varrho\to 0}\varrho^{1-d}\lim_{\delta\to 0}\limsup_{n}&\big(\inf \left\{F_{\e_n}(\w)(u,Q_{\nu}(x_0,\varrho)):\; u\in \mathcal{S}_{\e_n,\delta}^{\w}(u^{a,0}_{x_0,\nu},Q_{\nu}(x_0,\varrho))\right\}\big)
	\\
	&=\limsup_{\varrho\to 0}\varrho^{1-d}\lim_{\delta\to 0}\limsup_{n}m^{\delta}_{\e_n}(\w)(u_{x_0,\nu}^{a,0},Q_{\nu}(x_0,\varrho)).
	\end{align*}
\end{proposition}
\begin{proof}
	Note that it suffices to bound the first term by the second one. To reduce notation, we set $Q_{\varrho}:=Q_{\nu}(x_0,\varrho)$ and write $\e$ instead of $\e_n$. If $a=0$ then both sides are zero. Thus assume that $a\neq 0$. Fix $u_{\e}\in\mathcal{PC}_{\e,\delta}^{\w}(u^{a,0}_{x_0,\nu},Q_{\varrho})$ such that
	\begin{equation}\label{almostopt}
	F_{\e}(\w)(u_{\e},Q_{\varrho})\leq F_{\e}(\w)(u^{a,0}_{x_0,\nu},Q_{\varrho})\leq C\varrho^{d-1},
	\end{equation}
	where the last inequality is a consequence of Remark \ref{voronoi} and the boundedness of the discrete density $f$ (provided $\e$ is small enough). In what follows we construct a sequence $v_{\e}\in \mathcal{S}_{\e,\delta}^{\w}(u^{a,0}_{x_0,\nu},Q_{\varrho})$ that has almost the same energy. Given $\theta>0$, due to the monotonicity (\ref{monotone}) one can choose $L_{\theta}\geq 1$ such that
	\begin{equation}\label{atinf}
	|\beta(k,k)-f(\mathfrak{p})|<\theta f(\mathfrak{p})\quad\forall \mathfrak{p}\in \mathcal{P}_+(M)\text{ with }\mathfrak{p}^{-1}(\mathbb{N})\subset[L_{\theta},+\infty)\;\text{ and }\sum_{v\in \mathfrak{p}^{-1}(\mathbb{N})}\mathfrak{p}(v)=k
	\end{equation}
	for all $1\leq k\leq M$, where $\beta(k,k)$ is given by (\ref{limitatinf}). We fix $L_{\theta}$ from now on and consider the set of edges
	\begin{equation*}
	\mathcal{J}_{u_{\e}}=\{(x,y)\in\mathcal{E}(\w):\;\e x,\e y\in Q_{\varrho}\text{ and } \e^{1-p}|u_{\e}(\e x)-u_{\e}(\e y)|^p\geq L_{\theta}\}.
	\end{equation*} 
	Denote by $u_{\e}^i$ the $i^{th}$ component of $u_{\e}$. If $a_i=0$ we set $v_{\e}^i(\e x)=0$ for all $x\in\Lw$. Otherwise, we assume that $a_i>0$. The remaining case requires only minor modifications. For $t\in\R$ we define
	\begin{equation*}
	S^i_{\e}(t):=\{\e x\in\e\Lw\cap Q_{\varrho}:\;u^i_{\e}(\e x)>t\}.
	\end{equation*}
	To reduce notation, we also introduce the set 
	\begin{equation*}
	\mathcal{R}^i_{\e}(t)=\{(x,y)\in\mathcal{E}(\w):\;\e x\in Q_{\varrho}\cap S_{\e}^i(t),\,\e y\in Q_{\varrho}\backslash S_{\e}^i(t)\text{ or vice versa}\}.
	\end{equation*}
	Observe that for $(x,y)\in\mathcal{E}(\w)$ with $\e x,\e y\in Q_{\varrho}$ we have $(x,y)\in \mathcal{R}^i_{\e}(t)$ if and only if $t\in [u^i_{\e}(\e x),u^i_{\e}(\e y))$ or $t\in [u^i_{\e}(\e y),u^i_{\e}(\e x))$. Hence for such $x,y$ the following coarea-type estimate holds true:
	\begin{equation*}
	\int_{0}^{a_i}\mathds{1}_{\{(x,y)\in\mathcal{R}^i_{\e}(t)\}}\,\mathrm{d}t\leq|u_{\e}(\e x)-u_{\e}(\e y)|. 
	\end{equation*}
	Summing this estimate, we infer from (\ref{neighbours}) and H\"older's inequality that
	\begin{align*}\label{coarea}
	\int_{0}^{a_i}\e^{d-1}\#\big(\mathcal{R}_{\e}^i(t)\backslash  \mathcal{J}_{u_{\e}}\big)\,\mathrm{d}t&\leq \sum_{\substack{ (x,y)\in\mathcal{E}(\w)\backslash\mathcal{J}_{u_{\e}}\\ \e x,\e y\in Q_{\varrho}}}\e^{d-1}|u_{\e}(\e x)-u_{\e}(\e y)|
	\\
	&\leq C\e^{\frac{dp-d}{p}}(\#(\e\Lw\cap Q_{\varrho}))^{\frac{p-1}{p}}\bigg(\sum_{\substack{ (x,y)\in\mathcal{E}(\w)\backslash\mathcal{J}_{u_{\e}}\\ \e x,\e y\in Q_{\varrho}}}\e^{d}\Big|\frac{u_{\e}(\e x)-u_{\e}(\e y)}{\e}\Big|^p\Big|\bigg)^{\frac{1}{p}}.
	\end{align*}
	In order to estimate the last sum, recall that $L_{\theta}\geq 1$. Thus the definition of the set $\mathcal{J}_{u_{\e}}$, (\ref{neighbours}) and assumption (\ref{realcut}) imply for $\e x\in \e\Lw\cap Q_{\varrho}$ the uniform bound
	\begin{align*}
	&\sum_{\substack{\e y\in\e\Lw\cap Q_{\varrho}\\ (x,y)\in\mathcal{E}(\w)\backslash\mathcal{J}_{u_{\e}}}}\e\Big|\frac{u_{\e}(\e x)-u_{\e}(\e y)}{\e}\Big|^p\leq CL_{\theta}\min\Big\{\e\sum_{\substack{\e y\in\e\Lw\cap Q_{\varrho}\\ (x,y)\in\mathcal{E}(\w)\backslash\mathcal{J}_{u_{\e}}}}\Big|\frac{u_{\e}(\e x)-u_{\e}(\e y)}{\e}\Big|^p,1\Big\}
	\\
	\leq &CL_{\theta}\min\{\|\e|\nabla_{\w,\e}(u_{\e},Q_{\varrho})|^p(\e x)\|_1,1\}
	\leq CL_{\theta}f\left(\e|\nabla_{\w,\e}(u_{\e},Q_{\varrho})|^p(\e x)\right).
	\end{align*}
	Moreover, for $\e=\e(\varrho)$ small enough the cardinality term can be bounded by $\#(\e\Lw\cap Q_{\varrho})\leq C(\varrho\e^{-1})^d$, so that 
	\begin{equation*}
	\int_{0}^{a_i}\e^{d-1}\#\big(\mathcal{R}_{\e}^i(t)\backslash  \mathcal{J}_{u_{\e}}\big)\,\mathrm{d}t\leq C\varrho^{\frac{dp-d}{p}}\big(L_{\theta}F_{\e}(\w)(u_{\e},Q_{\varrho})\big)^{\frac{1}{p}}\leq CL_{\theta}\varrho^{\frac{dp-1}{p}},
	\end{equation*}
	where we applied (\ref{almostopt}) in the second inequality. Hence there exists $t^i_{\e}\in (0,a_i)$ such that
	\begin{equation}\label{goodchoice}
	\e^{d-1}\#\big(\mathcal{R}_{\e}^i(t^i_{\e})\backslash  \mathcal{J}_{u_{\e}}\big)\leq C|a_i|^{-1}L_{\theta}\varrho^{\frac{dp-1}{p}}.
	\end{equation}
	Define $v^i_{\e}$ by its values on $\e\Lw$ setting
	\begin{equation*}
	v^i_{\e}(\e x)=
	\begin{cases}
	0 &\mbox{if $u_{\e}(\e x)\leq t^i_{\e}$,}\\
	a_i &\mbox{if $u_{\e}(\e x)>t^i_{\e}$.}
	\end{cases}
	\end{equation*}
	As $t^i_{\e}\in (0,a_i)$, the boundary conditions imposed on $u_{\e}$ imply that the function $v_{\e}$ satisfies $v_{\e}(\e x)=u_{x_0,\nu}^{a,0}(\e x)$ for all $\e x\in\e\Lw\cap\partial_{\delta}Q_{\varrho}$, so that $v_{\e}\in\mathcal{S}_{\e,\delta}^{\w}(u_{x_0,\nu}^{a,0},Q_{\varrho})$. In order to estimate the energy difference, let $\e x\in\e\Lw\cap Q_{\varrho}$ be such that $\|\e|\nabla_{\w,\e}(v_{\e},Q_{\varrho})|^p(\e x)\|_1\neq 0$. We distinguish two exhaustive cases: either $(x,y)\in\mathcal{J}_{u_{\e}}$ for all $\e y\in\e\Lw\cap Q_{\varrho}$ with $(x,y)\in\mathcal{E}(\w)$, so that (\ref{monotone}) and (\ref{atinf}) yield  
	\begin{equation}\label{gradlarge}
	f\left(\e|\nabla_{\w,\e}(v_{\e},Q_{\varrho})|^p(\e x)\right)\leq (1+\theta)f\left(\e|\nabla_{\w,\e}(u_{\e},Q_{\varrho})|^p(\e x)\right),
	\end{equation}
	or there exists $y\in\Lw$ with $(x,y)\in\mathcal{R}^i_{\e}(t^i_{\e})\backslash\mathcal{J}_{u_{\e}}$ for some $i$. In this case we can use the estimate (\ref{goodchoice}) to bound the number of such $x$. Since $f$ is bounded by assumption, we deduce from (\ref{gradlarge}) that
	\begin{equation*}
	\varrho^{1-d}F_{\e}(\w)(v_{\e},Q_{\varrho})\leq (1+\theta)\varrho^{1-d}F_{\e}(\w)(u_{\e},Q_{\varrho})+C\sum_{i: a_i\neq 0}|a_i|^{-1}L_{\theta}\varrho^\frac{p-1}{p}.
	\end{equation*}
	Taking the appropriate infimum on each side, then letting $\e\to 0$ before $\delta\to 0$ and $\varrho\to 0$, we conclude the proof as $\theta>0$ was arbitrary.
\end{proof} 
Now we can relate the surface density $\varphi$ to the $\Gamma$-limit of the functionals $I_{\e}(\w)$ defined in \eqref{Ising}.
\begin{proposition}\label{p.surfacedensityequal}
Let $\e_n$ and $F(\w)$ be as in Proposition \ref{limitsbvp}. Then for every $x_0\in D$, $a\in\R^m\setminus\{0\}$ and $\nu\in S^1$ it holds that
\begin{equation*}
\varphi(x_0,a,\nu)=s(x_0,\nu),
\end{equation*}	
where $s$ is the surface tension of the $\Gamma$-limit of $I_{\e_n}(\w)(\cdot,D)$, which in particular exists.
\end{proposition}
\begin{proof}
Choosing any subsequence of $\e_n$ (not relabeled), the $\Gamma$-limit of $F_{\e_n}(\w)$ remains the same. Hence, combining Propositions \ref{limitsbvp} and \ref{separationofscales1} with Lemma \ref{approxminprob} yields the formula
\begin{equation}\label{eq:blowupphi}
\varphi(x_0,a,\nu)=\limsup_{\varrho\to 0}\varrho^{1-d}\lim_{\delta\to 0}\limsup_n\big(\inf \left\{F_{\e_n}(\w)(u,Q_{\nu}(x_0,\varrho)):\; u\in \mathcal{S}_{\e_n,\delta}^{\w}(u^{a,0}_{x_0,\nu},Q_{\nu}(x_0,\varrho))\right\}\big).
\end{equation}
For any $u\in \mathcal{S}_{\e_n,\delta}^{\w}(u^{a,0}_{x_0,\nu},Q_{\nu}(x_0,\varrho))$ we define the function $v\in \mathcal{S}_{\e_n,\delta}^{\w}(u^{-e_1,e_1}_{x_0,\nu},Q_{\nu}(x_0,\varrho))$ by
\begin{equation*}
v(\e x)=
\begin{cases}
-e_1 &\mbox{if $u(\e x)=a$},
\\
e_1 &\mbox{otherwise.}	
\end{cases}
\end{equation*}
Let $I_{\e}(\w)$ be the Ising-type energy defined in \eqref{Ising}. Then by the monotonicity (\ref{monotone}) it holds that
\begin{align*} I_{\e_n}(\w)(v,Q_{\nu}(x_0,\varrho))&\geq F_{\e_n}(\w)(u,Q_{\nu}(x_0,\varrho))
\\
&\geq\min_{1\leq l\leq k\leq M}\Big\{\frac{f(p)}{\beta(l,k)}:\,p=l\mathds{1}_{\{|a|^p\e_n^{1-p}\}}+(k-l)\mathds{1}_{\{0\}}\Big\}I_{\e_n}(\w)(v,Q_{\nu}(x_0,\varrho)).
\end{align*}
Since we have chosen a subsequence at the beginning of the proof, we may assume that $I_{\e_n}(\w)$ $\Gamma$-converges to some surface integral functional $I(\w)$ as in Theorem \ref{t.ACRmain} with density $s(x,\nu)$. Since $|a|\neq 0$, the definition of $\beta(l,k)$ in (\ref{limitatinf}) and Remark \ref{r.blowupformulas} imply
\begin{equation*}
\varphi(x_0,a,\nu)=\limsup_{\varrho\to 0}\varrho^{1-d}\lim_{\delta\to 0}\limsup_n\inf \left\{I_{\e_n}(\w)(v,Q_{\nu}(x_0,\varrho)):\, v\in \mathcal{S}_{\e_n,\delta}^{\w}(u^{-e_1,e_1}_{x_0,\nu},Q_{\nu}(x_0,\varrho))\right\}
=s(x_0,\nu).
\end{equation*}
Since the subsequence was arbitrary, the Urysohn-property of $\Gamma$-convergence yields that $I_{\e_n}(\w)$ indeed $\Gamma$-converges and the surface densities of the limits $F(\w)$ and $I(\w)$ agree.	
\end{proof}
Eventually, we can prove our first main result.
\begin{proof}[Proof of Theorem \ref{t.sepofscales}]
If $F_{\e_n}(\w)$ $\Gamma$-converges, then due to Propositions \ref{p.gradientpartsequal} and \ref{p.surfacedensityequal} both $E_{\e_n}(\w)$ and $I_{\e_n}(\w)$ $\Gamma$-converge, too. Also the reverse statement follows from the same propositions, since (up to subsequences) any $\Gamma$-limit of $F_{\e_n}$ is characterized by the $\Gamma$-limits of $E_{\e_n}(\w)$ and $I_{\e_n}(\w)$, which do not depend on further subsequences. Taking also into account Proposition \ref{limitsbvp} and Lemma \ref{compact} the $\Gamma$-limit is given by
\begin{equation*}
F(\w)(u)=
\begin{cases}
\int_D q(x,\nabla u)\dx+\int_{S_u}s(x,\nu_u)\,\mathrm{d}\mathcal{H}^{d-1} &\mbox{if $u\in SBV^p(D,\R^m)$,}
\\
+\infty &\mbox{if $u\in L^1(D,\R^m)\backslash GSBV^p(D,\R^m)$.}
\end{cases}
\end{equation*}
Hence it remains to characterize the functional $F(\w)(u)$ for $u\in L^1(D,\R^m)\cap GSBV^p(D,\R^m)$ such that $u\notin SBV^p(D,\R^m)$.
This will be achieved via truncation. Given $k>0$ we have that $T_ku\in SBV^p(D,\R^m)$. Lemma \ref{trunc} implies that $F(\w)(u)=\lim_{k\to +\infty}F(\w)(T_ku)$. In order to pass to the limit in the integral formula, we use Lemma \ref{truncation} and the symmetry $s(x,\nu)=s(x,-\nu)$, which yield
	\begin{equation*}
	F(\w)(T_k,A)=\int_{D\cap \{u\leq k\}}q(x,\nabla u)\,\mathrm{d}x+\int_{D\cap\{u>k\}}q(x,\nabla T_ku)\,\mathrm{d}x+\int_{S_{T_ku}}s(x,\nu_u)\,\mathrm{d}\mathcal{H}^{d-1}.
	\end{equation*}
	Since $S_u=\bigcup_k S_{T_ku}\cup N$ with $\mathcal{H}^{d-1}(N)=0$ and the second term vanishes due to dominated convergence, we can pass to the limit and conclude the proof.
\end{proof}

\section{Stochastic homogenization: proof of Theorems \ref{mainthm1} and \ref{convfull}}\label{s.stochhom}
In this section we derive the results in the random setting.

\begin{proof}[Proof of Theorem \ref{mainthm1}]
By \cite[Theorem 2]{ACG2} and \cite[Theorem 5.5]{ACR}\footnote{As noted before, in \cite{ACR} the proofs were given only for pairwise interactions. Nevertheless the same arguments apply in our setting (see also \cite[Theorem 6.7]{ACR}). Note that stationarity of the edges is important to apply the subadditive ergodic theorem of \cite{AkKr}.}, the $\Gamma$-limits of the two functionals $E_{\e}(\w)$ and $I_{\e}(\w)$ defined in \eqref{auxelastic} and \eqref{Ising} exist almost surely and have deterministic, spatially homogeneous densities. Hence the claim on the existence and form of the $\Gamma$-limit follows from Theorem \ref{t.sepofscales}. It remains to establish the properties of the integrands. Convexity and $p$-homogeneity of $h$ follow from the fact that the $\Gamma$-limit of the sequence of convex and $p$-homogeneous functionals $E_{\e}(\w)$ is again convex and $p$-homogeneous \cite[Theorem 11.1 and Proposition 11.6]{DM}, whereas convexity of the one-homogeneous extension of $\varphi$ follows from standard $L^1$-lower-semicontinuity results for functionals defined on sets of finite perimeter (see for instance \cite[Theorem 3.1]{AmBrI}).
\end{proof}
Finally we prove Theorem \ref{convfull}. For the convergence of minimizers we exploit the notion of biting convergence, which we recall here for reader's convenience.
\begin{definition}[Biting convergence]\label{biting}
Let $u_n\in L^1(D)$ be such that $\sup_n\|u_n\|_{L^1(D)}<+\infty$. We say that $u_n$ converges weakly to $u\in L^1(D)$ in the biting sense and write $u_n\overset{b}{\rightharpoonup}u$, if there exists a decreasing sequence $S_j\subset D$ of measurable sets such that $|S_j|\to 0$ and $u_n\rightharpoonup u$ in $L^1(D\backslash S_j)$ for all $j\in\mathbb{N}$.	
\end{definition}

\begin{remark}\label{bitingproperty}
Note that if $u_n\overset{b}{\rightharpoonup}u$ and $u_n\to v$ a.e., then $u=v$. This is a consequence of the uniqueness of the biting limit and equiintegrability of $L^1$-weakly convergent sequences.	
\end{remark}
\begin{proof}[Proof of Theorem \ref{convfull}]
We first construct a candidate for the constant $\gamma$. Define the sequence of non-negative equibounded functions $\gamma_{\e}(\w)\in L^{\infty}(D)$ by 
\begin{equation}\label{defbeps}
\gamma_{\e}(\w)(z)=\sum_{x\in\Lw}\frac{1}{|\mathcal{C}(x)|}\mathds{1}_{\e \mathcal{C}(x)}(z).
\end{equation}
We apply the ergodic theorem in order to establish weak$^*$-convergence of $\gamma_{\e}(\w)$. To this end, we introduce the family of half-open boxes with integer vertices $\mathcal{I}:=\{[a,b):a,b\in\mathbb{Z}^d, a_i<b_i\;\text{for all }i\}$ and define the rescaled integral averages $\tilde{\gamma}:\mathcal{I}\to L^1(\Omega)$ by
\begin{equation*}
\tilde{\gamma}(I,\w)=\int_I\sum_{x\in\Lw}\frac{1}{|\mathcal{C}(x)|}\mathds{1}_{ \mathcal{C}(x)}(z)\,\mathrm{d}z=\sum_{x\in \Lw}\frac{|\mathcal{C}(x)\cap I|}{|\mathcal{C}(x)|}.
\end{equation*}
The following three properties can be verified:
\begin{itemize}
	\item[(i)] $0\leq \tilde{\gamma}(I,\w)\leq C|I|$ for all $I\in\mathcal{I}$,
	\item[(ii)] If $I=\bigcup_{i}I_i\in\mathcal{I}$ with finitely many, pairwise disjoint $I_i\in\mathcal{I}$, then
	$\tilde{\gamma}(I,\w)=\sum_i\tilde{\gamma}(I_i,\w)$,	\item[(iii)] $\tilde{\gamma}(I,\tau_z\w)=\tilde{\gamma}(I-z,\w)$ for all $z\in\mathbb{Z}^d$.
\end{itemize}
Moreover, arguing as in \cite[Lemma A.1]{CiRu}, one can show that $\w\mapsto \tilde{\gamma}(I,\w)$ is $\mathcal{F}$-measurable. Hence we can apply the multi-parameter additive ergodic theorem (see \cite[Chapter 6, Theorem 2.8]{Krengel}) and conclude that $\mathbb{P}$-a.s. and for all $I\in\mathcal{I}$
\begin{equation*}
\gamma:=\mathbb{E}\left[\tilde{\gamma}\right]=\lim_{n\to +\infty} \frac{\tilde{\gamma}(nI,\w)}{|nI|},
\end{equation*}
where $\mathbb{E}$ denotes the expectation. It is straightforward to extend this convergence to all sequences $t_n\to +\infty$ and then to all cubes in $\R^d$ by a continuity argument. Now we identify the weak$^*$-limit of $\gamma_{\e}(\w)$. By a density argument it is enough to compute averages on cubes $Q\subset D$. A change of variables yields 
\begin{equation*}
\int_Q\gamma_{\e}(\w)(z)\,\mathrm{d}z=\e^d\tilde{\gamma}(Q/\e,\w)\to \gamma|Q|,
\end{equation*}
whence $\gamma_{\e}(\w)\overset{*}{\rightharpoonup}\gamma$ in $L^{\infty}(D)$ almost surely. 
	
Next we prove the lower bound for the $\Gamma$-convergence. Passing to a subsequence, for the $\liminf$-inequality it suffices to consider $u\in L^1(D,\R^m)$ and a sequence $u_{\e}\in\mathcal{PC}_{\e}^{\w}$ such that $u_{\e}\to u$ in $L^1(D,\R^m)$ and
\begin{equation}\label{boundedenergy}
\liminf_{\e\to 0}F_{\e,g}(\w)(u_{\e})=\lim_{\e\to 0}F_{\e,g}(\w)(u_{\e})\leq C<+\infty.
\end{equation}
Without affecting the convergence properties or the functional we redefine $g_{\e}(\w)(\e x)=u_{\e}(\e x)=0$ for all $\e x\in\e\Lw\backslash D$. Then by Remark \ref{voronoi} we have
\begin{equation*}
\|u_{\e}-g_{\e}(\w)\|^q_{L^q(D)}\leq C\sum_{\e x\in\e\Lw\cap D}\e^d|u_{\e}(\e x)-g_{\e}(\w)(\e x)|^q\leq C,
\end{equation*}
which in combination with (\ref{approxassumption}) implies that $u_{\e}$ is bounded in $L^q(D,\R^m)$. Thus we obtain $u\in L^q(D,\R^m)$, while Theorem \ref{mainthm1} and (\ref{boundedenergy}) yield $u\in GSBV^p(D,\R^m)$. Moreover, for any $1\leq r<q$ we deduce the following convergence properties:
\begin{equation}\label{weakly}
u_{\e}\to u\quad\text{in }L^r(D,\R^m),\quad u_{\e}\rightharpoonup u \quad\text{in }L^q(D,\R^m).
\end{equation}
Observe that by the definition of the function $\gamma_{\e}(\w)$ it holds that
\begin{equation*}
\sum_{\e x\in\e\Lw\cap D}\e^d|u_{\e}(\e x)-g_{\e}(\w)(\e x)|^q\geq \int_{D}|\gamma_{\e}(\w)(z)|\,|u_{\e}(z)-g_{\e}(\w)(z)|^q\,\mathrm{d}z.
\end{equation*} 
Due to (\ref{approxassumption}) and (\ref{weakly}), the sequence $u_{\e}-g_{\e}(\w)$ converges to $u-g$ in $L^r(D,\R^m)$ for any $1\leq r<q$. The lower semicontinuity result for pairs of weak-strong convergent sequences in \cite[Theorem 7.5]{FoLe} and the $\Gamma$-convergence of Theorem \ref{mainthm1} imply
\begin{align}\label{separateliminf}
\liminf_{\e\to 0}F_{\e,g}(\w)(u_{\e})&\geq \liminf_{\e\to 0} F_{\e}(\w)(u_{\e})+\liminf_{\e\to 0}\sum_{\e x\in\e\Lw\cap D}\e^d|u_{\e}(\e x)-g_{\e}(\w)(\e x)|^q\nonumber
\\
&\geq F(u)+\gamma\int_{D}|u-g|^q\,\mathrm{d}z,
\end{align}
where we used that $\gamma>0$ to avoid the modulus. This finishes the proof of the lower bound.
	
For the upper bound, it suffices to consider $u\in L^q(D,\R^m)\cap GSBV^p(D,\R^m)$. Note that for such $u$ we can equivalently compute the $\Gamma$-limit of $F_{\e}(\w)$ with respect to convergence in $L^q(D,\R^m)$. Indeed, by Lemma \ref{trunc} this is true for all truncated functions $T_ku$ with $k>0$ and by lower semicontinuity with respect to $L^q$-convergence and again Lemma \ref{trunc} we obtain
\begin{equation*}
\Gamma(L^q(D))\hbox{-}\limsup_{\e\to 0}F_{\e}(\w)(u)\leq \liminf_{k\to +\infty}\Big(\Gamma(L^q(D))\hbox{-}\limsup_{\e\to 0}F_{\e}(\w)(T_ku)\Big)\leq\liminf_{k\to +\infty}F(T_ku)=F(u).
\end{equation*}
Hence we find a sequence $u_{\e}\in\mathcal{PC}_{\e}^{\w}$ such that $u_{\e}\to u$ in $L^q(D,\R^m)$ and
\begin{equation}\label{recoverysequence}
\lim_{\e\to 0}F_{\e}(\w)(u_{\e})=F(u).
\end{equation}
Since $D$ has Lipschitz boundary, it satisfies an interior cone condition. Thus we find $c_D>0$ such that
\begin{equation*}
|\e\mathcal{C}(x)\cap D|\geq c_D\e^d\quad\text{for all }\e x\in\e\Lw\cap D. 
\end{equation*} 
Setting $D_{\e}=\{z\in D:\,\dist(z,\partial D)\leq 2R\e\}$, we deduce from the above estimate that 
\begin{align*}
\sum_{\e x\in\e\Lw\cap D}\e^d|u_{\e}(\e x)-g_{\e}(\w)(\e x)|^q\leq& \int_D\gamma_{\e}(\w)(z)|u_{\e}(z)-g_{\e}(\w)(z)|^q\,\mathrm{d}z
\\
&+C\int_{D_{\e}}|u_{\e}(z)-g_{\e}(\w)(z)|^q\,\mathrm{d}z.
\end{align*}
The last term vanishes when $\e\to 0$ since the sequence $|u_{\e}-g_{\e}|^q$ is equiintegrable on $D$. Moreover, by its product structure the sequence $\gamma_{\e}(\w)|u_{\e}-g_{\e}(\w)|^q$ converges weakly in $L^1(D,\R^m)$ to $\gamma|u-g|^q$. Therefore the last inequality implies
\begin{equation*}
\limsup_{\e\to 0}\sum_{\e x\in\e\Lw\cap D}\e^d|u_{\e}(\e x)-g_{\e}(\w)(\e x)|^p\leq \gamma\int_D|u(z)-g(z)|^p\,\mathrm{d}z.
\end{equation*}
Combined with (\ref{recoverysequence}) we obtain the upper bound.

Now we come to the second claim of the theorem. Existence of minimizers for fixed $\e$ follows from $L^{\infty}$-coercivity of the fidelity term in $F_{\e,g}(\w)$ and the lower semicontinuity assumption (\ref{flsc}). The last statement is true due to the fundamental property of $\Gamma$-convergence except that we have to prove compactness in $L^q(D,\R^m)$. As shown for the lower bound, any sequence $u_{\e}$ as in the statement is bounded in $L^q(D,\R^m)$ and therefore Lemma \ref{compact} yields that, up to subsequences, $u_{\e}\to u$ in $L^1(D,\R^m)$ for some $u\in L^q(D,\R^m)$. Clearly $u_{\e}$ is a recovery sequence for this $u$, so that
\begin{equation*}
F_g(u)= \lim_{\e\to 0}F_{\e,g}(\w)(u_{\e}).
\end{equation*}
Repeating the reasoning for (\ref{separateliminf}) we conclude from the above limit that
\begin{equation}\label{normconserved}
\lim_{\e\to 0}\int_D\gamma_{\e}(\w)(z)\,|u_{\e}(z)-g_{\e}(\w)(z)|^q\,\mathrm{d}z=\int_D\gamma|u(z)-g(z)|^q\,\mathrm{d}z,
\end{equation}
where $\gamma_{\e}(\w)$ is defined in (\ref{defbeps}). Now consider the non-negative sequence $a_{\e}:=|u_{\e}-g_{\e}(\w)|^q$. By (\ref{normconserved}) and the qualitative lower bound $\gamma_{\e}(\w)(z)\geq c$ this sequence is bounded in $L^1(D)$. By the biting lemma (see \cite[Lemma 2.63]{FoLe}) and Remark \ref{bitingproperty} we find a subsequence (not relabeled) such that $a_{\e}\overset{b}{\rightharpoonup}|u-g|^q$. Taking the same sets $S_j$ as for the biting convergence of $a_{\e}$ one can prove that the product $\gamma_{\e}(\w)a_{\e}$ converges in the biting sense to $\gamma|u-g|^q$. Indeed, on $D\backslash S_j$ the sequence $a_{\e}$ is equiintegrable by the Dunford-Pettis theorem and thus strongly convergent in $L^1(D\backslash S_j)$. Then by the usual product rules we obtain $\gamma_{\e}(\w)a_{\e}\rightharpoonup \gamma|u-g|^q$ in $L^1(D\backslash S_j)$, which shows biting convergence. Now we use that $\gamma_{\e}(\w)a_{\e}$ is nonnegative. By (\ref{normconserved}) and \cite[Proposition 2.67]{FoLe} this yields that $\gamma_{\e}(\w)a_{\e}\rightharpoonup \gamma|u-g|^q$ also in $L^1(D)$. Thus both sequences $\gamma_{\e}(\w)a_{\e}$ and $a_{\e}$ are equiintegrable on $D$. By Vitali's convergence theorem we obtain that
\begin{equation*}
\lim_{\e\to 0}\|u_{\e}-g_{\e}\|_{L^q(D)}=\|u-g\|_{L^q(D)},
\end{equation*}
which, by uniform convexity of $L^q(D,\R^m)$ and (\ref{approxassumption}), implies that $u_{\e}\to u$ in $L^q(D,\R^m)$ as claimed.
\end{proof}
\appendix
\section{}
In this appendix we provide a suitable framework to use diagonal arguments along $\Gamma$-converging discrete energies. We need this step since the general theory \cite[Chapter 10]{DM} to construct a metric for $\Gamma$-convergence requires an $L^p$-coercive lower bound for the discrete energies.

Given any function $G:L^p(B_1,\R^m)\to [0,+\infty]$ not identically $+\infty$ we define its Moreau-Yosida approximation for $\gamma>0$ as
\begin{equation*}
G^{\gamma}(u)=\inf_{v\in L^p(B_1,\R^m)}\Big(G(v)+\gamma\|u-v\|_{L^p(B_1)}^p\Big).
\end{equation*}
For $p\geq 1$ the functional $G^{\gamma}$ is locally Lipschitz-continuous on $L^p(B_1,\R^m)$ (see \cite[Theorem 9.15]{DM}). Let $\{w_k\}_{k\in\mathbb{N}}$ be a dense subset of $L^p(B_1,\R^m)$ containing $0$. Given two lower semicontinuous functions $G,H:L^p(B_1,\R^m)\to [0,+\infty]$ not identically $+\infty$ we define their distance by
\begin{equation*}
\mathfrak{d}(G,H)=\sum_{i,k\in\mathbb{N}}\frac{1}{2^{i+k}}|\arctan(G^{i}(w_k))-\arctan(H^{i}(w_k))|.
\end{equation*}
Note that on lower semicontinuous functions $\mathfrak{d}$ is indeed a distance, since $\mathfrak{d}(G,H)=0$ implies by local Lipschitz continuity that $G^i=H^i$ for all $i\in\mathbb{N}$. Letting $i\to +\infty$ it follows by lower semicontinuity that $G=H$ (see \cite[Remark 9.11]{DM}).
In order to state our result we need further notation: let $h:B_1\times\R^{m\times d}\to [0,+\infty)$ be a Carath\'eodory-function such that $\xi\mapsto h(x,\xi)$ is quasiconvex for a.e. $x\in B_1$ and
\begin{equation*}
\frac{1}{C}|\xi|^p-C\leq h(x,\xi)\leq C(|\xi|^p+1). 
\end{equation*}
Define the functional $\mathfrak{E}_h:L^p(B_1,\R^m)\to[0,+\infty]$ by
\begin{equation*}
\mathfrak{E}_h(u)=
\begin{cases}
\int_{B_1} h(x,\nabla u(x))\,\mathrm{d}x &\mbox{if $u\in W^{1,p}(B_1,\R^m)$,}\\
+\infty &\mbox{otherwise.}
\end{cases}
\end{equation*}

Then we have the following result.
\begin{lemma}\label{metric}
Consider a sequence $\e_j\to 0$ and let $x_0\in D$ and $\varrho_j>0$ be such that $B_{\varrho_j}(x_0)\subset D$ and $\e_j/\varrho_j\to 0$. Let further $h:B_1\times\R^{m\times d}\to[0,+\infty)$ be a Carath\'eodory-function as above and define $G_{\e_j,\varrho_j}(x_0,\w)$ as in Lemma \ref{changeofvar}. Then the following are equivalent: 
\begin{itemize}
	\item[(i)] $\Gamma(L^p(B_1,\R^m))\hbox{-}\lim_jG_{\e_j,\varrho_j}(x_0,\w)=\mathfrak{E}_h$,
	\item[(ii)] $\lim_j\mathfrak{d}(G_{\e_j,\varrho_j}(x_0,\w),\mathfrak{E}_h)=0$.
\end{itemize} 
\end{lemma}
\begin{proof}
(ii) $\Rightarrow$ (i): First note that both functionals are lower semicontinuous on $L^p(B_1,\R^m)$ and not identically $+\infty$. Assumption (ii) implies that 
\begin{equation*}
\lim_j G^i_{\e_j,\varrho_j}(x_0,\w)(w_k)=\mathfrak{E}^{i}_h(w_k)
\end{equation*}
for all $i,k\in\mathbb{N}$. Since $0\in\{w_k\}$ and $G^i_{\e_j,\varrho_j}(x_0,\w)(0)=0$, we deduce from \cite[Theorem 9.15]{DM} that the sequence $G^i_{\e_j,\varrho_j}(x_0,\w)$ is locally equicontinuous, so that the convergence extends to $L^p(B_1,\R^m)$ by density. The claim then follows from a general characterization of $\Gamma$-convergence (see \cite[Theorem 9.5]{DM}) when we let $i\to +\infty$.

(i) $\Rightarrow$ (ii): Clearly $G^{\gamma}_{\e_j,\varrho_j}(x_0,\w)(u)\leq \gamma\|u\|_{L^p(B_1)}$, so that given a sequence $u_j$ such that
\begin{equation}\label{eq:almostYosida}
G_{\e_j,\varrho_j}(x_0,\w)(u_j)+\gamma \|u_j-u\|_{L^p(B_1)}^p\leq G^{\gamma}_{\e_j,\varrho_j}(x_0,\w)(u)+\frac{1}{j},
\end{equation}
$u_j$ is bounded in $L^p(B_1,\R^m)$. Since the discrete density $f$ satisfies $f(\mathfrak{p})\leq C_f\|\mathfrak{p}\|_1$, we have by definition
\begin{equation*}
G_{\e_j,\varrho_j}(x_0,\w)(u_j)=E_{\e_j/\varrho_j}(\w)(u_j(\cdot-x_0/\varrho_j),B_1(x_0))\geq \frac{1}{C}F_{\e_j/\varrho_j}(\w)(u_j(\cdot-x_0/\varrho_j),B_1(x_0)).
\end{equation*}
From Lemma \ref{compact} and a change of variables we conclude that $u_j$ is compact in $L^1(B_1,\R^m)$. Hence there exists $v\in L^p(B_1,\R^m)$ such that, up to subsequences, $u_j\to v$ in $L^1(B_1,\R^m)$ and additionally $u_j\rightharpoonup v$ in $L^p(B_1,\R^m)$. In order to use the assumption (i), we use a truncation argument. Note that $T_ku_{j}\to T_kv$ in $L^p(D,\R^m)$. Then by $\Gamma$-convergence in $L^p(D,\R^m)$ and decrease by truncation of $G_{\e_j,\varrho_j}(x_0,\w)$ we obtain
\begin{equation*}
\liminf_{j}G_{\e_j,\varrho_j}(x_0,\w)(u_{j})\geq\liminf_{j}G_{\e_j,\varrho_j}(x_0,\w)(T_k u_{j})\geq\int_{B_1}h(x,\nabla T_kv),\mathrm{d}x.
\end{equation*}
Using Lemma \ref{truncation} we can pass to the limit in $k$ by dominated convergence. Combined with the weak lower semicontinuity of the $L^p$-norm we infer from \eqref{eq:almostYosida} that
\begin{equation*}
\mathfrak{E}_h(v)+\gamma \|v-u\|^p_{L^p(B_1)}\leq\liminf_j G^{\gamma}_{\e_j,\varrho_j}(x_0,\w)(u).
\end{equation*}
Moreover, given any $\tilde{u}\in L^p(B_1,\R^m)$ let us consider a sequence $\tilde{u}_j\to \tilde{u}$ in $L^p(B_1,\R^m)$ such that $\lim_j G_{\e_j,\varrho_j}(x_0,\w)(\tilde{u}_j)= \mathfrak{E}_h(\tilde{u})$. Then by the definition of the Moreau-Yosida transformation
\begin{equation*}
\limsup_j G_{\e_j,\varrho_j}^{\gamma}(x_0,\w)(u)\leq \lim_j(G_{\e_j,\varrho_j}(x_0,\w)(\tilde{u}_j)+\gamma\|\tilde{u}_j-u\|^p_{L^p(B_1)})= \mathfrak{E}_h(\tilde{u})+\gamma \|\tilde{u}-u\|^p_{L^p(B_1)}.
\end{equation*}
Combined with the previous inequality we obtain that $\mathfrak{E}_h^{\gamma}(u)=\mathfrak{E}_h(v)+\gamma \|v-u\|^p_{L^p(B_1)}$ and by setting $\tilde{u}=v$ we showed that
\begin{equation*}
\lim_j G_{\e_j,\varrho_j}^{\gamma}(x_0,\w)(u)=\mathfrak{E}_h^{\gamma}(u)
\end{equation*}
for all $u\in L^p(B_1,\R^m)$ and all $\gamma>0$. This property implies (ii) by the definition of the metric.
\end{proof}
\begin{remark}\label{continuumcase}
The equivalence of Lemma \ref{metric} remains valid for two functionals $\mathfrak{E}_{h_j}$ and $\mathfrak{E}_{h}$ provided the Carath\'eodory functions $h_j$ satisfy growth conditions uniformly in $j$. In this case the proof simplifies since the Moreau-Yosida transformations are equicoercive in $L^p(B_1,\R^m)$ due to the Sobolev embedding.	
\end{remark}

\section*{}
\subsection*{Acknowledgment}
\noindent MR acknowledges financial support from the European Research Council under
the European Community's Seventh Framework Programme (FP7/2014-2019 Grant Agreement
QUANTHOM 335410). The author further appreciates the helpful suggestions of two anonymous referees.

\subsection*{Conflict of interest}
\noindent The author declares that he has no conflict of interest.

\end{document}